\documentclass{amsart}

\usepackage[english]{babel}

\usepackage{amstext,amsmath,amsthm,amssymb,mathrsfs}
\usepackage{float}

\usepackage{bm}

\usepackage{txfonts}
\usepackage{enumitem}
\usepackage[T1]{fontenc}

\newcommand{\Longlra}{\ensuremath{\Longleftrightarrow}}

\newcommand{\longra}{\ensuremath{\longrightarrow}}

\newcommand{\B}{\varmathbb{B}}
\newcommand{\C}{\varmathbb{C}}
\newcommand{\R}{\varmathbb{R}}
\newcommand{\N}{\varmathbb{N}}
\newcommand{\F}{\varmathbb{F}}
\newcommand{\G}{\varmathbb{G}}
\newcommand{\Z}{\varmathbb{Z}}

\newcommand{\E}{\varmathbb{E}}
\newcommand{\D}{\varmathbb{D}}

\newcommand{\U}{\varmathbb{U}}
\newcommand{\I}{\varmathbb{I}}

\newcommand{\norm}[1]{||#1||}

\newcommand{\normB}[1]{\Big|\Big|#1\Big|\Big|}
\newcommand{\normb}[1]{\big|\big|#1\big|\big|}

\newcommand{\ip}[2]{\langle #1, #2 \rangle}

\DeclareMathAlphabet{\mathpzc}{OT1}{pzc}{m}{it}

\renewcommand{\vec}[1]{\bm{#1}}

\DeclareMathOperator{\supp}{supp}

\theoremstyle{plain}
\newtheorem{thm}{Theorem}[section]
\newtheorem{lemma}[thm]{Lemma}
\newtheorem{prop}[thm]{Proposition}
\newtheorem{cor}[thm]{Corollary}

\theoremstyle{definition}
\newtheorem{definitie}[thm]{Definition}

\newtheorem{ex}[thm]{Example}

\theoremstyle{remark}
\newtheorem{remark}[thm]{Remark}

\begin{document}

\title[Parabolic Problems with Inhomogeneous Boundary Conditions]{Maximal Regularity with Weights for Parabolic Problems with Inhomogeneous Boundary Conditions}

\author{Nick Lindemulder}
\address{Delft Institute of Applied Mathematics \\
Delft University of Technology \\
P.O. Box 5031 \\
2600 GA Delft \\
The Netherlands}
\email{N.Lindemulder@tudelft.nl}

\subjclass[2010]{Primary: 35K50, 46E35, 46E40; Secondary: 42B15, 42B25}

\keywords{anisotropic spaces, Besov, Bessel potential, inhomogeneous boundary conditions, maximal regularity, mixed-norms, parabolic initial-boundary value problems, Sobolev, traces, Triebel-Lizorkin, vector-valued, weights}
\date{\today}

\thanks{The author is supported by the Vidi subsidy 639.032.427 of the Netherlands Organisation for Scientific Research (NWO)}

\begin{abstract}
In this paper we establish weighted $L^{q}$-$L^{p}$-maximal regularity for linear vector-valued parabolic initial-boundary value problems with inhomogeneous boundary conditions of static type. The weights we consider are power weights in time and in space, and yield flexibility in the optimal regularity of the initial-boundary data and allow to avoid compatibility conditions at the boundary.
The novelty of the followed approach is the use of weighted anisotropic mixed-norm Banach space-valued function spaces of Sobolev, Bessel potential, Triebel-Lizorkin and Besov type, whose trace theory is also subject of study.
\end{abstract}

\maketitle

\section{Introduction}\label{PBIBVP:sec:intro}

This paper is concerned with weighted maximal $L^{q}$-$L^{p}$-regularity for vector-valued parabolic initial-boundary value problems of the form
\begin{equation}\label{intro:pibvp}
\begin{array}{rllll}
\partial_{t}u(x,t) + \mathcal{A}(x,D,t)u(x,t) &= f(x,t), & x \in \mathscr{O}, & t \in J,  \\
\mathcal{B}_{j}(x',D,t)u(x',t) &= g_{j}(x',t), & x' \in \partial\mathscr{O}, & t \in J, & j=1,\ldots,n, \\
u(x,0) &= u_{0}(x), & x \in \mathscr{O}.
\end{array}
\end{equation}
Here, $J$ is a finite time interval, $\mathscr{O} \subset \R^{d}$ is a $C^{\infty}$-domain with a compact boundary $\partial\mathscr{O}$ and the coefficients of the differential operator $\mathcal{A}$ and the boundary operators $\mathcal{B}_{1},\ldots,\mathcal{B}_{n}$ are $\mathcal{B}(X)$-valued, where $X$ is a UMD Banach space.
One could for instance take $X=\C^{N}$, describing a system of $N$ initial-boundary value problems.
Our structural assumptions on $\mathcal{A},\mathcal{B}_{1},\ldots,\mathcal{B}_{n}$ are an ellipticity condition and a condition of Lopatinskii-Shapiro type.
For homogeneous boundary data (i.e. $g_{j}=0$, $j=1,\ldots,n$) these problems  include linearizations of reaction-diffusion systems and of phase field models with Dirichlet, Neumann and Robin conditions. However, if one wants to use linearization techniques to treat such problems with non-linear boundary conditions, it is crucial to have a sharp theory for the fully inhomogeneous problem.

During the last 25 years, the theory of maximal regularity turned out to be an important tool in the theory of nonlinear PDEs. Maximal regularity means that there is an isomorphism between the data and the solution of the problem in suitable function spaces.
Having established maximal regularity for the linearized problem, the nonlinear problem can be treated with tools as the contraction principle and the implicit function theorem.
Let us mention \cite{Angenent1990,Clement&Simonett2001} for approaches in spaces of continuous functions,
\cite{Acquistapace&Terreni,Lunardi_book1995} for approaches in H\"older spaces and
\cite{Amann_Lin_and_quasilinear_parabolic,Amann_Maximal_regularity_and_quasilinear_pbvps_problems,Clement_Li,
Clement&Pruss1992,EPS,Pruss2002,Pruess&Simonett2016_book} for approaches in $L^{p}$-spaces (with $p \in (1,\infty)$).

As an application of his operator-valued Fourier multiplier theorem, Weis \cite{Weis2001} characterized maximal $L^{p}$-regularity for abstract Cauchy problems in UMD Banach spaces in terms of an $\mathcal{R}$-boundedness condition on the operator under consideration. A second approach to the maximal $L^{p}$-regularity problem is via the operator sum method, as initiated by Da Prato $\&$ Grisvard \cite{PG_intro} and extended by Dore $\&$ Venni \cite{Dore_Veni} and Kalton $\&$ Weis \cite{Kalton_Weis}. For more details on these approaches and for more information on (the history of) the maximal $L^{p}$-regularity problem in general, we refer to \cite{DHP1,Kunstmann&Weis}.

In the maximal $L^{q}$-$L^{p}$-regularity approach to \eqref{intro:pibvp} one is looking for solutions $u$ in the "maximal regularity space"
\begin{equation}\label{PBIVB:eq:max-reg_space}
W^{1}_{q}(J;L^{p}(\mathscr{O};X)) \cap L^{q}(J;W^{2n}_{p}(\mathscr{O};X)).
\end{equation}
To be more precise, the problem \eqref{intro:pibvp} is said to enjoy the property of \emph{maximal $L^{q}$-$L^{p}$-regularity} if there exists a (necessarily unique) space of initial-boundary data $\mathscr{D}_{i.b.} \subset L^{q}(J;L^{p}(\partial\mathscr{O};X))^{n} \times L^{p}(\mathscr{O};X)$ such that for every $f \in L^{q}(J;L^{p}(\mathscr{O};X))$ it holds that \eqref{intro:pibvp} has a unique solution $u$ in \eqref{PBIVB:eq:max-reg_space}
if and only if $(g=(g_{1},\ldots,g_{n}),u_{0}) \in \mathscr{D}_{i.b.}$.
In this situation there exists a Banach norm on $\mathscr{D}_{i.b.}$, unique up to equivalence, with
\[
\mathscr{D}_{i.b.} \hookrightarrow L^{q}(J;L^{p}(\partial\mathscr{O};X))^{n} \oplus L^{p}(\mathscr{O};X),
\]
which makes the associated solution operator a topological linear isomorphism between the data space $L^{q}(J;L^{p}(\mathscr{O};X)) \oplus \mathscr{D}_{i.b.}$ and the solution space $W^{1}_{q}(J;L^{p}(\mathscr{O};X)) \cap L^{q}(J;W^{2n}_{p}(\mathscr{O};X))$.
The \emph{maximal $L^{q}$-$L^{p}$-regularity problem} for \eqref{intro:pibvp} consists of establishing maximal $L^{q}$-$L^{p}$-regularity
for \eqref{intro:pibvp} and explicitly determining the space $\mathscr{D}_{i.b.}$.

The maximal $L^{q}$-$L^{p}$-regularity problem for \eqref{intro:pibvp} was solved by Denk, Hieber $\&$ Pr\"uss \cite{DHP2}, who used operator sum methods in combination with tools from vector-valued harmonic analysis.
Earlier works on this problem are \cite{LSU} ($q=p$) and \cite{Weidemaier_intro} ($p \leq q$) for scalar-valued 2nd order problems with Dirichlet and Neumann boundary conditions. Later, the results of \cite{DHP2} for the case that $q=p$ have been extended by Meyries $\&$ Schnaubelt \cite{MeySchnau2} to the setting of temporal power weights $v_{\mu}(t)=t^{\mu}$, $\mu \in [0,q-1)$; also see \cite{Mey_PHD-thesis}.
Works in which maximal $L^{q}$-$L^{p}$-regularity of other problems with inhomogeneous boundary conditions are studied, include \cite{DPZ,Denk&Saal&Seiler,Denk&Schnaubelt,EPS,MeySchnau2} (the case $q=p$) and \cite{Meyries&Veraar_Traces,Shibata&Shimuzu} (the case $q \neq p$).

It is desirable to have maximal $L^{q}$-$L^{p}$-regularity for the full range $q,p \in (1,\infty)$, as this enables one to treat more nonlinearities.
For instance, one often requires large $q$ and $p$ due to better Sobolev embeddings, and $q \neq p$ due to scaling invariance of PDEs (see e.g. \cite{Giga}). However, for \eqref{intro:pibvp} the case $q \neq p$ is more involved than the case $q=p$ due to the inhomogeneous boundary conditions.
This is not only reflected in the proof, but also in the space of initial-boundary data (\cite[Theorem~2.3]{DHP2} versus \cite[Theorem~2.2]{DHP2}).
Already for the heat equation with Dirichlet boundary conditions, the boundary data $g$ has to be in the intersection space
\begin{equation}\label{PIBVP:eq:intro;int_space}
F^{1-\frac{1}{2p}}_{q,p}(J;L^{p}(\partial\mathscr{O})) \cap L^{q}(J;B^{2-\frac{1}{p}}_{p,p}(\partial\mathscr{O})),
\end{equation}
which in the case $q=p$ coincides with $W^{1-\frac{1}{2p}}_{p}(J;L^{p}(\partial\mathscr{O})) \cap L^{p}(J;W^{2-\frac{1}{p}}_{p}(\partial\mathscr{O}))$;
here $F^{s}_{q,p}$ denotes a Triebel-Lizorkin space and $W^{s}_{p}=B^{s}_{p,p}$ a non-integer order Sobolev-Slobodeckii space.

In this paper we will extend the results of \cite{DHP2,MeySchnau2}, concerning the maximal $L^{q}$-$L^{p}$-regularity problem for \eqref{intro:pibvp}, to the setting of power weights in time and in space for the full range $q,p \in (1,\infty)$.
In contrast to \cite{DHP2,MeySchnau2}, we will not only view the spaces \eqref{PBIVB:eq:max-reg_space} and \eqref{PIBVP:eq:intro;int_space} as intersection spaces, but also as anisotropic mixed-norm function spaces on $J \times \mathscr{O}$ and $J \times \partial\mathscr{O}$, respectively.
Identifications of intersection spaces of the type \eqref{PIBVP:eq:intro;int_space} with anisotropic mixed-norm Triebel-Lizorkin spaces have been considered in a previous paper \cite{Lindemulder_Intersection}, all in a generality including the weighted vector-valued setting.
The advantage of these identifications is that they allows us to use weighted vector-valued versions of trace results of Johnsen $\&$ Sickel \cite{JS_traces}. These trace results will be studied in their own right in the present paper.

The weights we consider are the power weights
\begin{equation}\label{PIBVP:eq:intro:weights}
v_{\mu}(t) = t^{\mu} \:\quad (t \in J)  \quad\quad \mbox{and} \quad\quad
w^{\partial\mathscr{O}}_{\gamma}(x) = \mathrm{dist}(\,\cdot\,,\partial\mathscr{O})^{\gamma} \:\quad (x \in \mathscr{O}),
\end{equation}
where $\mu \in (-1,q-1)$ and $\gamma \in (-1,p-1)$.
These weights yield flexibility in the optimal regularity of the initial-boundary data and allow to avoid compatibility conditions at the boundary,
which is nicely illustrated by the result (see Example~\ref{PIBVP:ex:thm:main_result;example}) that the corresponding version of \eqref{PIBVP:eq:intro;int_space} becomes
\begin{equation*}
F^{1-\frac{1}{2p}(1+\gamma)}_{q,p}(J,v_{\mu};L^{p}(\partial\mathscr{O})) \cap L^{q}(J,v_{\mu};B^{2-\frac{1}{p}(1+\gamma)}_{p,p}(\partial\mathscr{O})).
\end{equation*}
Note that one requires less regularity of $g$ by increasing $\gamma$.

The idea to work in weighted spaces equipped with weights like \eqref{PIBVP:eq:intro:weights} has already proven to be very useful in several situations. In an abstract semigroup setting temporal weights were introduced by Cl\'ement $\&$ Simonett \cite{Clement&Simonett} and Pr\"uss $\&$ Simonett \cite{pruss_simonett}, in the context of maximal continous regularity and maximal $L^{p}$-regularity, respectively.
Other works on maximal temporally weighted $L^{p}$-regularity are \cite{Koehne&Pruess&Wilke,LeCrone&Pruess&Wilke} for quasilinear parabolic evolution equations and \cite{MeySchnau2} for parabolic problems with inhomogeneous boundary conditions.
Concerning the use of spatial weights, we would like to mention
\cite{Brewster&Mitrea_BVP_weighted_Sobolev_spaces_Lipschitz_manifolds,Maz'ya&Shaposhnikova2005,
Mitrea&Taylor_Poisson_problem_weighted_Sobolev_spaces_Lipschitz_domains}
for boundary value problems and
\cite{Alos&Bonaccorsi_stability_spde_Dirichlet_white-noise_boundary_condition,
Brezniak&Goldys&Peszat&Russo_2nd-order_PDE_Dirichlet_white_noise_boundary_cond,
Fabri&Goldys_halfline_Diricjlet_boundary_control_and_noise,
Roberts_PhD-thesis,
Sowers_mult-dim_reaction-diffusion_eqns_white_noise_boundary_perturbations}
for problems with boundary noise.

\emph{The paper is organized as follows.} In Section~\ref{PIBVP:sec:prelim} we discuss the necessary preliminaries, in Section~\ref{PIBVP:sec:main_result} we state the main result of this paper, Theorem~\ref{PIBVP:thm:main_result}, in Section~\ref{PIBVP:sec:traces} we establish the necessary trace theory, in Section~\ref{PIBVP:sec:Sobolev_embedding_Besov} we consider a Sobolev embedding theorem, and in Section~\ref{pbvp:sec:pibvp} we finally prove Theorem~\ref{PIBVP:thm:main_result}.

\emph{Acknowledgements.} The author would like to thank Mark Veraar for the supervision of his master thesis \cite{Lindemulder_master-thesis}, which led to the present paper.

\section{Preliminaries}\label{PIBVP:sec:prelim}

\subsection{Weighted Mixed-norm Lebesgue Spaces}\label{PIBVP:subsec:sec:prelim;mixed-norm}

A weight on $\R^{d}$ is a measurable function $w:\R^{d} \longra [0,\infty]$ that takes its values almost everywhere in $(0,\infty)$. We denote by $\mathcal{W}(\R^{d})$ the set of all weights on $\R^{d}$. For $p \in (1,\infty)$ we denote by $A_{p} = A_{p}(\R^{d})$ the class of all Muckenhoupt $A_{p}$-weights, which are all the locally integrable weights for which the $A_{p}$-characteristic $[w]_{A_{p}}$ is finite.
Here
\[
[w]_{A_{p}} = \sup_{Q} \left(\fint_{Q}w \right) \left( \fint_{Q}w^{-p'/p} \right)^{p/p'}
\]
with the supremum taken over all cubes $Q \subset \R^{d}$ with sides parallel to the coordinate axes.
We furthermore set $A_{\infty} := \bigcup_{p \in (1,\infty)}A_{p}$. For more information on Muckenhoupt weights we refer to \cite{Grafakos_modern}.

Important for this paper are the power weights of the form $w=\mathrm{dist}(\,\cdot\,,\partial\mathscr{O})^{\gamma}$, where $\mathscr{O}$ is a $C^{\infty}$-domain in $\R^{d}$ and where $\gamma \in (-1,\infty)$.
If $\gamma \in (-1,\infty)$ and $p \in (1,\infty)$, then (see \cite[Lemma~2.3]{Farwig&Sohr_Weighted_Lq-theory_Stokes_resolvent_exterior_domains} or \cite[Lemma~2.3]{Mitrea&Taylor_Poisson_problem_weighted_Sobolev_spaces_Lipschitz_domains})
\begin{equation}\label{DBVP:eq:sec:prelim:power_weight_Ap}
w_{\gamma}^{\partial\mathscr{O}} := \mathrm{dist}(\,\cdot\,,\partial\mathscr{O})^{\gamma} \in A_{p} \quad \Longlra \quad \gamma \in (-1,p-1).
\end{equation}
For the important model problem case $\mathscr{O} = \R^{d}_{+}$ we simply write $w_{\gamma}:= w_{\gamma}^{\partial\R^{d}_{+}} = \mathrm{dist}(\,\cdot\,,\partial\R^{d}_{+})^{\gamma}$.

Replacing cubes by rectangles in the definition of the $A_{p}$-characteristic $[w]_{A_{p}} \in [1,\infty]$ of a weight $w$ gives rise to the $A_{p}^{rec}$-characteristic $[w]_{A_{p}^{rec}} \in [1,\infty]$ of $w$.
We denote by $A_{p}^{rec} = A_{p}^{rec}(\R^{d})$ the class of all weights with $[w]_{A_{p}^{rec}} < \infty$.
For $\gamma \in (-1,\infty)$ it holds that $w_{\gamma} \in A_{p}^{rec}$ if and only if $\gamma \in (-1,p-1)$.

Let $d = |\mathpzc{d}|_{1} = \mathpzc{d}_{1} + \ldots + \mathpzc{d}_{l}$ with $\mathpzc{d} = (\mathpzc{d}_{1},\ldots,\mathpzc{d}_{l}) \in (\Z_{\geq 1})^{l}$. The decomposition
\[
\R^{d} = \R^{\mathpzc{d}_{1}} \times \ldots \times \R^{\mathpzc{d}_{l}}.
\]
is called the $\mathpzc{d}$-\emph{decomposition} of $\R^{d}$.
For $x \in \R^{d}$ we accordingly write $x = (x_{1},\ldots,x_{l})$ and $x_{j}=(x_{j,1},\ldots,x_{j,\mathpzc{d}_{j}})$, where $x_{j} \in \R^{\mathpzc{d}_{j}}$ and $x_{j,i} \in \R$ $(j=1,\ldots,l; i=1,\ldots,\mathpzc{d}_{j})$. We also say that we view $\R^{d}$ as being
$\mathpzc{d}$-\emph{decomposed}. Furthermore, for each $k \in \{1,\ldots,l\}$ we define the inclusion map
\begin{equation*}
\iota_{k} = \iota_{[\mathpzc{d};k]} : \R^{\mathpzc{d}_{k}} \longra \R^{d},\, x_{k} \mapsto (0,\ldots,0,x_{k},0,\ldots,0),
\end{equation*}
and the projection map
\begin{equation*}
\pi_{k} = \pi_{[\mathpzc{d};k]} :  \R^{d} \longra \R^{\mathpzc{d}_{k}},\, x = (x_{1},\ldots,x_{l}) \mapsto x_{k}.
\end{equation*}

Suppose that $\R^{d}$ is $\mathpzc{d}$-decomposed as above.
Let $\vec{p} = (p_{1},\ldots,p_{l}) \in [1,\infty)^{l}$ and $\vec{w}=(w_{1},\ldots,w_{l}) \in \prod_{j=1}^{l}\mathcal{W}(\R^{\mathpzc{d}_{j}})$.
We define the \emph{weighted mixed-norm space} $L^{\vec{p},\mathpzc{d}}(\R^{d},\vec{w})$ as the space of all $f\in L^{0}(\R^{d})$ satisfying
\[
\norm{f}_{L^{\vec{p},\mathpzc{d}}(\R^{d},\vec{w})} :=
 \left( \int_{\R^{\mathpzc{d}_{l}}} \ldots \left(\int_{\R^{\mathpzc{d}_{1}}}|f(x)|^{p_{1}}w_{1}(x_{1})dx_{1} \right)^{p_{2}/p_{1}} \ldots w_{l}(x_{l})dx_{l} \right)^{1/p_{l}} < \infty.
\]
We equip $L^{\vec{p},\mathpzc{d}}(\R^{d},\vec{w})$ with the norm $\norm{\,\cdot\,}_{L^{\vec{p},\mathpzc{d}}(\R^{d},\vec{w})}$, which turns it into a Banach space.
Given a Banach space $X$, we denote by $L^{\vec{p},\mathpzc{d}}(\R^{d},\vec{w};X)$ the associated Bochner space
\[
L^{\vec{p},\mathpzc{d}}(\R^{d},\vec{w};X) := L^{\vec{p},\mathpzc{d}}(\R^{d},\vec{w})[X] = \{ f \in L^{0}(\R^{d};X) : \norm{f}_{X} \in L^{\vec{p},\mathpzc{d}}(\R^{d},\vec{w}) \}.
\]

\subsection{Anisotropy}

Suppose that $\R^{d}$ is $\mathpzc{d}$-decomposed as in Section~\ref{PIBVP:subsec:sec:prelim;mixed-norm}.
Given $\vec{a} \in (0,\infty)^{l}$, we define the $(\mathpzc{d},\vec{a})$-anisotropic dilation $\delta^{(\mathpzc{d},\vec{a})}_{\lambda}$ on $\R^{d}$ by $\lambda > 0$ to be the mapping $\delta^{(\mathpzc{d},\vec{a})}_{\lambda}$ on $\R^{d}$ given by the formula
\[
\delta^{(\mathpzc{d},\vec{a})}_{\lambda}x := (\lambda^{a_{1}}x_{1},\ldots,\lambda^{a_{l}}x_{l}), \quad\quad x \in \R^{d}.
\]

A $(\mathpzc{d},\vec{a})$-anisotropic distance function on $\R^{d}$ is a function $u:\R^{d} \longra [0,\infty)$ satisfying
\begin{itemize}
\item[(i)] $u(x)=0$ if and only if $x=0$.
\item[(ii)] $u(\delta^{(\mathpzc{d},\vec{a})}_{\lambda}x) = \lambda u(x)$ for all $x \in \R^{d}$ and $\lambda > 0$.
\item[(iii)] There exists a $c>0$ such that $u(x+y) \leq c(u(x)+u(y))$ for all $x,y \in \R^{d}$.
\end{itemize}
All $(\mathpzc{d},\vec{a})$-anisotropic distance functions on $\R^{d}$ are equivalent: Given two $(\mathpzc{d},\vec{a})$-anisotropic distance functions $u$ and $v$ on $\R^{d}$, there exist constants $m,M>0$ such that $m u(x) \leq v(x) \leq M u(x)$ for all $x \in \R^{d}$

In this paper we will use the $(\mathpzc{d},\vec{a})$-anisotropic distance function $|\,\cdot\,|_{\mathpzc{d},\vec{a}}:\R^{d} \longra [0,\infty)$ given by the formula
\begin{equation*}\label{functieruimten:eq:anisotropic_distance}
|x|_{\mathpzc{d},\vec{a}} := \left(\sum_{j=1}^{l}|x_{j}|^{2/a_{j}}\right)^{1/2} \quad\quad (x \in \R^{d}).
\end{equation*}

\subsection{Fourier Multipliers}

Let $X$ be a Banach space. The space of $X$-valued tempered distributions on $\R^{d}$ is defined as $\mathcal{S}'(\R^{d};X) := \mathcal{L}(\mathcal{S}(\R^{d};X))$; for the theory of vector-valued distributions we refer to \cite{Amann_distributions} (and \cite[Section~III.4]{Amann_Lin_and_quasilinear_parabolic}).
We write $\widehat{L^{1}}(\R^{d};X) := \mathscr{F}^{-1}L^{1}(\R^{d};X) \subset \mathcal{S}'(\R^{d};X)$.
To a symbol $m \in L^{\infty}(\R^{d};\mathcal{B}(X))$ we associate the Fourier multiplier operator
\[
T_{m}: \widehat{L^{1}}(\R^{d};X) \longra \widehat{L^{1}}(\R^{d};X),\, f \mapsto \mathscr{F}^{-1}[m\hat{f}].
\]
Given $\vec{p} \in [1,\infty)^{l}$ and  $\vec{w} \in \prod_{j=1}^{l}A_{\infty}(\R^{\mathpzc{d}_{j}})$, we call $m$ a \emph{Fourier multiplier} on $L^{\vec{p},\mathpzc{d}}(\R^{d},\vec{w};X)$ if $T_{m}$ restricts to an operator on $\widehat{L^{1}}(\R^{d};X) \cap L^{\vec{p},\mathpzc{d}}(\R^{d},\vec{w};X)$ which is bounded with respect to $L^{\vec{p},\mathpzc{d}}(\R^{d},\vec{w};X)$-norm.
In this case $T_{m}$ has a unique extension to a bounded linear operator on $L^{\vec{p},\mathpzc{d}}(\R^{d},\vec{w};X)$ due to denseness of
$\mathcal{S}(\R^{d};X)$ in $L^{\vec{p},\mathpzc{d}}(\R^{d},\vec{w};X)$, which we still denote by $T_{m}$.
We denote by $\mathcal{M}_{\vec{p},\mathpzc{d},\vec{w}}(X)$ the set of all Fourier multipliers $m \in L^{\infty}(\R^{d};\mathcal{B}(X))$ on $L^{\vec{p},\mathpzc{d}}(\R^{d},\vec{w};X)$.
Equipped with the norm $\norm{m}_{\mathcal{M}_{\vec{p},\mathpzc{d},\vec{w}}(X)} := \norm{T_{m}}_{\mathcal{B}(L^{\vec{p},\mathpzc{d}}(\R^{d},\vec{w};X)}$, $\mathcal{M}_{\vec{p},\mathpzc{d},\vec{w}}(X)$ becomes a Banach algebra (under the natural pointwise operations) for which the natural inclusion $\mathcal{M}_{\vec{p},\mathpzc{d},\vec{w}}(X) \hookrightarrow \mathcal{B}(L^{\vec{p},\mathpzc{d}}(\R^{d},\vec{w};X))$ is an isometric Banach algebra homomorphism; see \cite{Kunstmann&Weis} for the unweighted non-mixed-norm setting.

For each $\vec{a} \in (0,\infty)^{l}$ and $N \in \N$ we define $\mathcal{M}^{(\mathpzc{d},\vec{a})}_{N}$ as the space of all $m \in C^{N}(\R^{d})$ for which
\[
\norm{m}_{\mathcal{M}^{(\mathpzc{d},\vec{a})}_{N}} :=  \sup_{|\alpha| \leq N}\sup_{\xi \in \R^{d}} (1+|\xi|_{\mathpzc,\vec{a}})^{\vec{a} \cdot_{\mathpzc{d}} \alpha}|D^{\alpha}m(\xi)| < \infty.
\]
We furthermore define $\mathscr{RM}(X)$ as the space of all operator-valued symbols $m \in C^{1}(\R \setminus \{0\};\mathcal{B}(X))$ for which we have the $\mathcal{R}$-bound
\[
\norm{m}_{\mathscr{RM}_(X)} :=  \mathcal{R}\big\{ tm^{[k]}(t) : t \neq 0, k=0,1 \,\big\} < \infty;
\]
see e.g. \cite{DHP1,boek} for the notion of $\mathcal{R}$-boundedness.

If $X$ is a UMD space, $\vec{p} \in (1,\infty)^{l}$,
\[
\vec{w} \in \left\{\begin{array}{ll}
\prod_{j=1}^{l}A_{p_{j}}^{rec}(\R^{\mathpzc{d}_{j}}), & l \geq 2, \\
A_{p}(\R^{d}), & l=1,
\end{array}\right.
\]
and $\vec{a} \in (0,\infty)^{l}$, then there exists an $N \in \N$ for which
\begin{equation}\label{PIBVP:eq:prelim:anisotrope_mixed-norm_fm}
\mathcal{M}^{(\mathpzc{d},\vec{a})}_{N} \hookrightarrow \mathcal{M}_{\vec{p},\mathpzc{d},\vec{w}}(X).
\end{equation}
If $X$ is a UMD space, $p \in (1,\infty)$ and $w \in A_{p}(\R)$, then
\begin{equation}\label{PIBVP:eq:prelim:operator-valued_FM}
\mathscr{RM}(X) \hookrightarrow \mathcal{M}_{p,w}(X).
\end{equation}
For these results we refer to \cite{Fackler&Hytonen&Lindemulder2018} and the references given there.

\subsection{Function Spaces}

For the theory of vector-valued distributions we refer to \cite{Amann_distributions} (and \cite[Section~III.4]{Amann_Lin_and_quasilinear_parabolic}). For vector-valued function spaces we refer to \cite{MV_pointwise_mp} (weighted setting) and the references given therein.
Anisotropic spaces can be found in \cite{Amann09,JS_traces,Lindemulder_master-thesis}; for the statements below on weighted anisotropic vector-valued function space we refer to \cite{Lindemulder_master-thesis}.

Suppose that $\R^{d}$ is $\mathpzc{d}$-decomposed as in Section~\ref{PIBVP:subsec:sec:prelim;mixed-norm}.
Let $X$ be a Banach space and let $\vec{a} \in (0,\infty)^{l}$.
For $0 < A < B < \infty$ we define $\Phi^{\mathpzc{d},\vec{a}}_{A,B}(\R^{d})$ as the set of all sequences $\varphi = (\varphi_{n})_{n \in \N} \subset \mathcal{S}(\R^{d})$ which are constructed in the following way: given a $\varphi_{0} \in \mathcal{S}(\R^{d})$ satisfying
\begin{equation*}
0 \leq \hat{\varphi}_{0} \leq 1,\: \hat{\varphi}_{0}(\xi) = 1 \:\mbox{if}\: |\xi|_{\mathpzc{d},\vec{a}} \leq A,\:
\hat{\varphi}_{0}(\xi) = 0 \:\mbox{if}\: |\xi|_{\mathpzc{d},\vec{a}} \geq B,
\end{equation*}
$(\varphi_{n})_{n \geq 1} \subset \mathcal{S}(\R^{d})$ is defined via the relations
\begin{equation*}
\hat{\varphi}_{n}(\xi) = \hat{\varphi}_{1}(\delta^{(\mathpzc{d},\vec{a})}_{2^{-n+1}}\xi) = \hat{\varphi}_{0}(\delta^{(\mathpzc{d},\vec{a})}_{2^{-n}}\xi) - \hat{\varphi}_{0}(\delta^{(\mathpzc{d},\vec{a})}_{2^{-n+1}}\xi),
\quad \xi \in \R^{d}, n \geq 1.
\end{equation*}
Observe that
\begin{equation*}\label{functieruimten:eq:support_LP-sequence}
\supp \hat{\varphi}_{0} \subset \{ \xi \mid |\xi|_{\mathpzc{d},\vec{a}} \leq B \} \quad \mbox{and} \quad
\supp \hat{\varphi}_{n} \subset \{ \xi \mid 2^{n-1}A \leq |\xi|_{\mathpzc{d},\vec{a}} \leq 2^{n}B \}, \quad  n \geq 1.
\end{equation*}

We put $\Phi^{\mathpzc{d},\vec{a}}(\R^{d}) := \bigcup_{0<A<B<\infty}\Phi^{\mathpzc{d},\vec{a}}_{A,B}(\R^{d})$.
In case $l=1$ we write $\Phi^{a}(\R^{d}) = \Phi^{\mathpzc{d},a}(\R^{d})$, $\Phi(\R^{d}) = \Phi^{1}(\R^{d})$, $\Phi^{a}_{A,B}(\R^{d}) = \Phi^{\mathpzc{d},a}_{A,B}(\R^{d})$, and $\Phi_{A,B}(\R^{d}) = \Phi^{1}_{A,B}(\R^{d})$.

To $\varphi \in \Phi^{\mathpzc{d},a}(\R^{d})$ we associate the family of convolution operators
$(S_{n})_{n \in \N} = (S_{n}^{\varphi})_{n \in \N}   \subset \mathcal{L}(\mathcal{S}'(\R^{d};X),\mathscr{O}_{M}(\R^{d};X)) \subset \mathcal{L}(\mathcal{S}'(\R^{d};X))$ given by
\begin{equation}\label{functieruimten:eq:convolutie_operatoren}
S_{n}f = S_{n}^{\varphi}f := \varphi_{n} * f = \mathscr{F}^{-1}[\hat{\varphi}_{n}\hat{f}] \quad\quad (f \in \mathcal{S}'(\R^{d};X)).
\end{equation}
Here $\mathscr{O}_{M}(\R^{d};X)$ denotes the space of slowly increasing $X$-valued smooth functions on $\R^{d}$.
It holds that $f = \sum_{n=0}^{\infty}S_{n}f$ in $\mathcal{S}'(\R^{d};X)$ respectively in $\mathcal{S}(\R^{d};X)$ whenever $f \in \mathcal{S}'(\R^{d};X)$ respectively $f \in \mathcal{S}(\R^{d};X)$.

Given $\vec{a} \in (0,\infty)^{l}$, $\vec{p} \in [1,\infty)^{l}$, $q \in [1,\infty]$, $s \in \R$, and $\vec{w} \in \prod_{j=1}^{l}A_{\infty}(\R^{\mathpzc{d}_{j}})$,
the Besov space $B^{s,\vec{a}}_{\vec{p},q,\mathpzc{d}}(\R^{d},\vec{w};X)$ is defined as the space of all $f \in \mathcal{S}'(\R^{d};X)$ for which
\[
\norm{f}_{B^{s,\vec{a}}_{\vec{p},q,\mathpzc{d}}(\R^{d},\vec{w};X)}
:= \norm{(2^{ns}S_{n}^{\varphi}f)_{n \in \N}}_{\ell^{q}(\N)[L^{\vec{p},\mathpzc{d}}(\R^{d},\vec{w})](X)} < \infty
\]	
and the Triebel-Lizorkin space $F^{s,\vec{a}}_{\vec{p},q,\mathpzc{d}}(\R^{d},\vec{w};X)$ is defined as the space of all $f \in \mathcal{S}'(\R^{d};X)$ for which
\[
\norm{f}_{F^{s,\vec{a}}_{\vec{p},q,\mathpzc{d}}(\R^{d},\vec{w};X)}
:= \norm{(2^{ns}S_{n}^{\varphi}f)_{n \in \N}}_{L^{\vec{p},\mathpzc{d}}(\R^{d},\vec{w})[\ell^{q}(\N)](X)} < \infty.
\]	
Up to an equivalence of extended norms on $\mathcal{S}'(\R^{d};X)$, $\norm{\,\cdot\,}_{B^{s,\vec{a}}_{\vec{p},q,\mathpzc{d}}(\R^{d},\vec{w};X)}$ and $\norm{\,\cdot\,}_{F^{s,\vec{a}}_{\vec{p},q,\mathpzc{d}}(\R^{d},\vec{w};X)}$ do not depend on the particular choice of $\varphi \in \Phi^{\mathpzc{d},\vec{a}}(\R^{d})$.

Let us note some basic relations between these spaces.
Monotonicity of $\ell^{q}$-spaces yields that, for $1 \leq q_{0} \leq q_{1} \leq \infty$,
\begin{equation}
B^{s,\vec{a}}_{\vec{p},q_{0},\mathpzc{d}}(\R^{d},\vec{w};X) \hookrightarrow B^{s,\vec{a}}_{\vec{p},q_{1},\mathpzc{d}}(\R^{d},\vec{w};X), \quad\quad
F^{s,\vec{a}}_{\vec{p},q_{0},\mathpzc{d}}(\R^{d},\vec{w};X) \hookrightarrow F^{s,\vec{a}}_{\vec{p},q_{1},\mathpzc{d}}(\R^{d},\vec{w};X).
\end{equation}
For $\epsilon > 0$ it holds that
\begin{equation}\label{PIBVP:prelim:eq:elem_embd_epsilon}
B^{s,\vec{a}}_{\vec{p},\infty,\mathpzc{d}}(\R^{d},\vec{w};X) \hookrightarrow B^{s-\epsilon,\vec{a}}_{\vec{p},1,\mathpzc{d}}(\R^{d},\vec{w};X).
\end{equation}
Furthermore, Minkowksi's inequality gives
\begin{equation}\label{PIBVP:prelim:eq:elem_embd_BF_rel}
B^{s,\vec{a}}_{\vec{p},\min\{p_{1},\ldots,p_{l},q\},\mathpzc{d}}(\R^{d},\vec{w};X) \hookrightarrow B^{s,\vec{a}}_{\vec{p},q,\mathpzc{d}}(\R^{d},\vec{w};X) \hookrightarrow
B^{s,\vec{a}}_{\vec{p},\max\{p_{1},\ldots,p_{l},q\},\mathpzc{d}}(\R^{d},\vec{w};X).
\end{equation}

Let $\vec{a} \in (0,\infty)^{l}$.
A normed space $\E \subset \mathcal{S}'(\R^{d};X)$ is called $(\mathpzc{d},\vec{a})$-admissible if there exists an $N \in \N$ such that
\[
m(D)f \in \E \quad \mbox{with} \quad \norm{m(D)f}_{\E} \lesssim \norm{m}_{\mathcal{M}^{(\mathpzc{d},\vec{a})}_{N}}\norm{f}_{\E},
\quad\quad (m,f) \in \mathscr{O}_{M}(\R^{d}) \times \E,
\]
where $m(D)f=\mathscr{F}^{-1}[m\hat{f}]$.
The Besov space $B^{s,\vec{a}}_{\vec{p},q,\mathpzc{d}}(\R^{d},\vec{w};X)$ and the Triebel-Lizorkin space $F^{s,\vec{a}}_{\vec{p},q,\mathpzc{d}}(\R^{d},\vec{w};X)$ are examples of  $(\mathpzc{d},\vec{a})$-admissible Banach spaces.

To each $\sigma \in \R$ we associate the operators $\mathcal{J}^{[\mathpzc{d};j]}_{\sigma} \in \mathcal{L}(\mathcal{S}'(\R^{d};X))$ and $\mathcal{J}^{\mathpzc{d},\vec{a}}_{\sigma} \in \mathcal{L}(\mathcal{S}'(\R^{d};X))$ given by
\[
\mathcal{J}^{[\mathpzc{d};j]}_{\sigma}f := \mathscr{F}^{-1}[(1+|\pi_{[\mathpzc{d};j]}|^{2})^{\sigma/2}\hat{f}]
\quad\quad \mbox{and} \quad\quad \mathcal{J}^{\mathpzc{d},\vec{a}}_{\sigma}f := \sum_{k=1}^{l}\mathcal{J}^{[\mathpzc{d};k]}_{\sigma/a_{k}}f.
\]
We call $\mathcal{J}^{\mathpzc{d},\vec{a}}_{\sigma}$ the $(\mathpzc{d},\vec{a})$-anisotropic Bessel potential operator of order $\sigma$.

Let $\E \hookrightarrow \mathcal{S}'(\R^{d};X)$ be a Banach space. Write
\[
J_{\vec{n},\mathpzc{d}} := \left\{\, \alpha \in \bigcup_{j=1}^{l}\iota_{[\mathpzc{d};j]}\N^{\mathpzc{d}_{j}} \,:\, |\alpha_{j}| \leq n_{j} \,\right\}, \quad\quad \vec{n} \in \left( \Z_{\geq1} \right)^{l}.
\]
Given $\vec{n} \in \left( \Z_{\geq1} \right)^{l}$, $\vec{s}, \vec{a} \in (0,\infty)^{l}$, and $s \in \R$,
we define the Banach spaces
$\mathcal{W}^{\vec{n}}_{\mathpzc{d}}[\E], \mathcal{H}^{\vec{s}}_{\mathpzc{d}}[\E], \mathcal{H}^{s,\vec{a}}_{\mathpzc{d}}[\E] \hookrightarrow \mathcal{S}'(\R^{d};X)$ as follows:
\[
\begin{array}{ll}
\mathcal{W}^{\vec{n}}_{\mathpzc{d}}[\E] &:= \{ f \in \mathcal{S}'(\R^{d}) : D^{\alpha}f \in \E, \alpha \in J_{\vec{n},\mathpzc{d}} \}, \\
\mathcal{H}^{\vec{s}}_{\mathpzc{d}}[\E] &:=  \{ f \in \mathcal{S}'(\R^{d}) : \mathcal{J}^{[\mathpzc{d};j]}_{s_{j}}f \in \E, j=1,\ldots,l \},  \\
\mathcal{H}^{s,\vec{a}}_{\mathpzc{d}}[\E] &:= \{ f \in \mathcal{S}'(\R^{d}) : \mathcal{J}^{\mathpzc{d},\vec{a}}_{s}f \in \E \},
\end{array}
\]
with the norms
\[
\norm{f}_{\mathcal{W}^{\vec{n}}_{\mathpzc{d}}[\E]} = \sum_{\alpha \in J_{\vec{n},\mathpzc{d}}}\norm{D^{\alpha }f}_{E}, \quad\quad
\norm{f}_{\mathcal{H}^{\vec{s}}_{\mathpzc{d}}[\E]} =  \sum_{j=1}^{l}\norm{\mathcal{J}^{[\mathpzc{d};j]}_{s_{j}}f}_{\E}, \quad\quad
\norm{f}_{\mathcal{H}^{s,\vec{a}}_{\mathpzc{d}}[\E]} = \norm{ \mathcal{J}^{\mathpzc{d},\vec{a}}_{s}f }_{\E}.
\]
Note that $\mathcal{H}^{\vec{s}}_{\mathpzc{d}}[\E] \hookrightarrow \mathcal{H}^{s,\vec{a}}_{\mathpzc{d}}[\E]$ contractively in case that $\vec{s} = (s/a_{1},\ldots,s/a_{l})$. Furthermore, note that if $\F \hookrightarrow \mathcal{S}'(\R^{d};X)$ is another Banach space, then
\begin{equation}
\E \hookrightarrow \F \quad \mbox{implies} \quad
\mathcal{W}^{\vec{n}}_{\mathpzc{d}}[\E] \hookrightarrow \mathcal{W}^{\vec{n}}_{\mathpzc{d}}[\F],
\mathcal{H}^{\vec{s}}_{\mathpzc{d}}[\E] \hookrightarrow \mathcal{H}^{\vec{s}}_{\mathpzc{d}}[\F],
\mathcal{H}^{s,\vec{a}}_{\mathpzc{d}}[\E] \hookrightarrow \mathcal{H}^{s,\vec{a}}_{\mathpzc{d}}[\F].
\end{equation}

If $\E \hookrightarrow \mathcal{S}'(\R^{d};X)$ is a $(\mathpzc{d},\vec{a})$-admissible Banach space for a given $\vec{a} \in (0,\infty)^{l}$, then
\begin{equation}\label{PIBVP:eq:prelim:identities_Sobolev&Bessel-potential}
\mathcal{W}^{\vec{n}}_{\mathpzc{d}}[\E] = \mathcal{H}^{\vec{n}}_{\mathpzc{d}}[\E] = \mathcal{H}^{s,\vec{a}}_{\mathpzc{d}}[\E], \quad\quad  s \in \R, \vec{n} = s \vec{a}^{-1} \in \left( \Z_{\geq1} \right)^{l},
\end{equation}
and
\begin{equation}
\mathcal{H}^{\vec{s}}_{\mathpzc{d}}[\E] = \mathcal{H}^{s,\vec{a}}_{\mathpzc{d}}[\E], \quad\quad s >0, \vec{s} = s \vec{a}^{-1}.
\end{equation}
Furthermore,
\begin{equation}\label{PIBVP:eq:prelim:differential}
D^{\alpha} \in \mathcal{B}(\mathcal{H}^{s,\vec{a}}_{\mathpzc{d}}[\E],\mathcal{H}^{s- \vec{a} \cdot_{\mathpzc{d}} \alpha,\vec{a}}_{\mathpzc{d}}[\E]), \quad\quad s \in \R, \alpha \in \N^{d}.
\end{equation}

Let $\vec{a} \in (0,\infty)^{l}$, $\vec{p} \in [1,\infty)^{l}$, $q \in [1,\infty]$, and $\vec{w} \in \prod_{j=1}^{l}A_{\infty}(\R^{\mathpzc{d}_{j}})$. For $s,s_{0} \in \R$ it holds that
\begin{equation*}
B^{s+s_{0},\vec{a}}_{\vec{p},q,\mathpzc{d}}(\R^{d},\vec{w};X) = \mathcal{H}^{s,\vec{a}}_{\mathpzc{d}}[B^{s_{0},\vec{a}}_{\vec{p},q,\mathpzc{d}}(\R^{d},\vec{w};X)],
\quad
F^{s+s_{0},\vec{a}}_{\vec{p},q,\mathpzc{d}}(\R^{d},\vec{w};X) = \mathcal{H}^{s,\vec{a}}_{\mathpzc{d}}[F^{s_{0},\vec{a}}_{\vec{p},q,\mathpzc{d}}(\R^{d},\vec{w};X)].
\end{equation*}

Let $X$ be a Banach space, $\vec{a} \in (0,\infty)^{l}$, $\vec{p} \in (1,\infty)^{l}$, $\vec{w} \in \prod_{j=1}^{l}A_{p_{j}}(\R^{\mathpzc{d}_{j}})$, $s \in \R$, $\vec{s} \in (0,\infty)^{l}$ and $\vec{n} \in (\N_{>0})^{l}$. We define
\[
\begin{array}{ll}
W^{\vec{n}}_{\vec{p},\mathpzc{d}}(\R^{d},\vec{w};X) &:= \mathcal{W}^{\vec{n}}_{\mathpzc{d}}[L^{\vec{p},\mathpzc{d}}(\R^{d},\vec{w};X)], \\
H^{\vec{s}}_{\vec{p},\mathpzc{d}}(\R^{d},\vec{w};X) &:= \mathcal{H}^{\vec{s}}_{\mathpzc{d}}[L^{\vec{p},\mathpzc{d}}(\R^{d},\vec{w};X)],  \\ H^{s,\vec{a}}_{\vec{p},\mathpzc{d}}(\R^{d},\vec{w};X)  &:= \mathcal{H}^{s,\vec{a}}_{\mathpzc{d}}[L^{\vec{p},\mathpzc{d}}(\R^{d},\vec{w};X)]. \\
\end{array}
\]
If
\begin{itemize}
\item $\E = W^{\vec{n}}_{\vec{p},\mathpzc{d}}(\R^{d},\vec{w};X)$, $\vec{n} \in (\Z_{\geq 1})^{l}$, $\vec{n}=s\vec{a}^{-1}$; or
\item $\E = H^{s,\vec{a}}_{\vec{p},\mathpzc{d}}(\R^{d},\vec{w};X)$; or
\item $\E = H^{\vec{a}}_{\vec{p},\mathpzc{d}}(\R^{d},\vec{w};X)$, $\vec{a} \in (0,1)^{l}$, $\vec{a}=s\vec{a}^{-1}$,
\end{itemize}
then we have the inclusions
\begin{equation}\label{PIBVP:eq:elem_embedding_FEF}
F^{s,\vec{a}}_{\vec{p},1,\mathpzc{d}}(\R^{d},\vec{w};X) \hookrightarrow  \E
\hookrightarrow F^{s,\vec{a}}_{\vec{p},\infty,\mathpzc{d}}(\R^{d},\vec{w};X).
\end{equation}

\begin{thm}[\cite{Lindemulder_Intersection}]\label{functieruimten:thm:aTL_rep_intersection}
Let $X$ be a Banach space, $l=2$, $\vec{a} \in (0,\infty)^{2}$, $p,q \in (1,\infty)$, $s > 0$,
and $\vec{w} \in A_{p}(\R^{\mathpzc{d}_{1}}) \times A_{q}(\R^{\mathpzc{d}_{2}})$. Then
\begin{equation}\label{functieruimten:eq:thm;aTL_rep_intersection}
F^{s,\vec{a}}_{(p,q),p,\mathpzc{d}}(\R^{d},\vec{w};X) =
F^{s/a_{2}}_{q,p}(\R^{\mathpzc{d}_{2}},w_{2};L^{p}(\R^{\mathpzc{d}_{1}},w_{1};X)) \cap
L^{q}(\R^{\mathpzc{d}_{2}},w_{2};F^{s/a_{1}}_{p,p}(\R^{\mathpzc{d}_{1}},w_{1};X))
\end{equation}
with equivalence of norms.
\end{thm}
This intersection representation is actually a corollary of a more general intersection representation in \cite{Lindemulder_Intersection}. In the above form it can also be found in \cite[Theorem~5.2.35]{Lindemulder_master-thesis}. For the case $X=\C$, $\mathpzc{d}_{1}=1$, $\vec{w}=\vec{1}$ we refer to \cite[Proposition~3.23]{Denk&Kaip}.

\section{The Main Result}\label{PIBVP:sec:main_result}

\subsection{Maximal $L^{q}_{\mu}$-$L^{p}_{\gamma}$-regularity}

In order to give a precise description of the maximal weighted $L^{q}$-$L^{p}$-regularity approach for \eqref{intro:pibvp},
let $\mathscr{O}$ be either $\R^{d}_{+}$ or a smooth domain in $\R^{d}$ with a compact boundary $\partial\mathscr{O}$.
Furthermore, let $X$ be a Banach space, let
\[
q \in (1,\infty),\: \mu \in (-1,q-1)  \quad \mbox{and} \quad p \in (1,\infty),\: \gamma \in (-1,p-1),
\]
let $v_{\mu}$ and $w^{\partial\mathscr{O}}_{\gamma}$ be as in \eqref{PIBVP:eq:intro:weights},
put
\begin{equation}\label{PIBVP:eq:basic_spaces_max-reg}
\begin{split}
\U^{p,q}_{\gamma,\mu} &:=  W^{1}_{q}(J,v_{\mu};L^{p}(\mathscr{O},w^{\partial\mathscr{O}}_{\gamma};X)) \cap L^{q}(J,v_{\mu};W^{2n}_{p}(\mathscr{O},w^{\partial\mathscr{O}}_{\gamma};X)), \quad \mbox{(space of solutions $u$)} \\
\F^{p,q}_{\gamma,\mu} &:= L^{q}(J,v_{\mu};L^{p}(\mathscr{O},w^{\partial\mathscr{O}}_{\gamma};X)),
 \quad\mbox{(space of domain inhomogeneities $f$)}\\
\B^{p,q}_{\mu} &:= L^{q}(J,v_{\mu};L^{p}(\partial\mathscr{O};X)), \quad\mbox{(boundary space)}
\end{split}
\end{equation}
and let $n,n_{1},\ldots,n_{n} \in \N$ be natural numbers with $n_{j} \leq 2n - 1$ for each $j \in \{1,\ldots,n\}$.
Suppose that for each $\alpha \in \N^{d}, |\alpha| \leq 2n$,
\[
a_{\alpha} \in \mathcal{D}'(\mathscr{O} \times J;\mathcal{B}(X)) \quad \mbox{with} \quad a_{\alpha}D^{\alpha} \in \mathcal{B}(\U^{p,q}_{\gamma,\mu},\F^{p,q}_{\gamma,\mu})
\]
and that for each $j \in \{1,\ldots,n\}$ and $\beta \in \N^{d}, |\beta| \leq n_{j}$,
\[
b_{j,\beta} \in \mathcal{D}'(\partial\mathscr{O} \times J;\mathcal{B}(X)) \quad \mbox{with} \quad b_{j,\beta}\mathrm{tr}_{\partial\mathscr{O}}D^{\beta} \in \mathcal{B}(\U^{p,q}_{\gamma,\mu},\B^{p,q}_{\mu}),
\]
where the conditions $a_{\alpha}D^{\alpha} \in \mathcal{B}(\U^{p,q}_{\gamma,\mu},\F^{p,q}_{\gamma,\mu})$ and $b_{j,\beta}\mathrm{tr}_{\partial\mathscr{O}}D^{\beta} \in \mathcal{B}(\U^{p,q}_{\gamma,\mu},\B^{p,q}_{\mu})$ have to be interpreted in the sense of bounded extension from the space of $X$-valued compactly supported smooth functions.
Define $\mathcal{A}(D) \in \mathcal{B}(\U^{p,q}_{\gamma,\mu},\F^{p,q}_{\gamma,\mu})$ and $\mathcal{B}_{1}(D),\ldots,\mathcal{B}_{n}(D) \in \mathcal{B}(\U^{p,q}_{\gamma,\mu},\B^{p,q}_{\mu})$ by
\begin{equation}\label{pbvp:eq:intro;diff_operators}
\begin{split}
\mathcal{A}(D) &:= \sum_{|\alpha| \leq 2n}a_{\alpha}D^{\alpha}, \\
\mathcal{B}_{j}(D) &:= \sum_{|\beta| \leq n_{j}}b_{j,\beta}\mathrm{tr}_{\partial\mathscr{O}}D^{\beta}, \quad\quad j=1,\ldots,n.
\end{split}
\end{equation}

In the above notation, given $f \in \F^{p,q}_{\gamma,\mu}$ and $g=(g_{1},\ldots,g_{n}) \in  [\B^{p,q}_{\mu}]^{n}$, one can ask the question whether the initial-boundary value problem
\begin{equation}\label{PIBVP:eq:rigorous_problem}
\begin{array}{rll}
\partial_{t}u + \mathcal{A}(D)u &= f,  \\
\mathcal{B}_{j}(D)u &= g_{j}, & j=1,\ldots,n, \\
\mathrm{tr}_{t=0}u &= u_{0}.
\end{array}
\end{equation}
has a unique solution $u \in \U^{p,q}_{\gamma,\mu}$.

\begin{definitie}\label{PIBVP:def:max-reg_Lq-Lp_gewichten}
We say that the problem \eqref{PIBVP:eq:rigorous_problem} enjoys the property of \emph{maximal $L^{q}_{\mu}$-$L^{p}_{\gamma}$-regularity} if there exists a (necessarily unique) linear space $\mathscr{D}_{i.b.} \subset [\B^{p,q}_{\mu}]^{n} \times L^{p}(\mathscr{O},w^{\partial\mathscr{O}}_{\gamma};X)$
such that \eqref{PIBVP:eq:rigorous_problem} admits a unique solution
$u \in \U^{p,q}_{\gamma,\mu}$ if and only if $(f,g,u_{0})  \in \mathscr{D} = \F^{p,q}_{\gamma,\mu} \times \mathscr{D}_{i.b.}$. In this situation we call $\mathscr{D}_{i.b.}$ the \emph{optimal space of initial-boundary data} and $\mathscr{D}$  the \emph{optimal space of data}.
\end{definitie}

\begin{remark}\label{PIBVP:rmk:def:max-reg_Lq-Lp_gewichten;top}
Let the notations be as above.
If the problem \eqref{PIBVP:eq:rigorous_problem} enjoys the property of \emph{maximal $L^{q}_{\mu}$-$L^{p}_{\gamma}$-regularity}, then there exists a unique Banach topology on the space of initial-boundary data $\mathscr{D}_{i.b.}$ such that
$\mathscr{D}_{i.b.} \hookrightarrow [\B^{p,q}_{\mu}]^{n} \times L^{p}(\mathscr{O},w^{\partial\mathscr{O}}_{\gamma};X)$.
Moreover, if $\mathscr{D}_{i.b.}$ has been equipped with a Banach norm generating such a topology, then the solution operator
\[
\mathscr{S}: \mathscr{D} = \F^{p,q}_{\gamma,\mu} \oplus \mathscr{D}_{i.b.} \longra \U^{p,q}_{\gamma,\mu},\,
(f,g,u_{0}) \mapsto \mathscr{S}(f,g,u_{0}) = u
\]
is an isomorphism of Banach spaces, or equivalently,
\[
\norm{u}_{\U^{p,q}_{\gamma,\mu}} \eqsim \norm{f}_{\F^{p,q}_{\gamma,\mu}} + \norm{(g,u_{0})}_{\mathscr{D}_{i.b.}},  \quad\quad u=\mathscr{S}(f,g,u_{0}), (f,g,u_{0}) \in \mathscr{D}.
\]
\end{remark}

The \emph{maximal $L^{q}_{\mu}$-$L^{p}_{\gamma}$-regularity problem} for \eqref{PIBVP:eq:rigorous_problem} consists of establishing maximal $L^{q}_{\mu}$-$L^{p}_{\gamma}$-regularity for \eqref{PIBVP:eq:rigorous_problem} and explicitly determining the space $\mathscr{D}_{i.b.}$ together with a norm as in Remark~\ref{PIBVP:rmk:def:max-reg_Lq-Lp_gewichten;top}.
As the main result of this paper, Theorem~\ref{PIBVP:thm:main_result}, we will solve the maximal $L^{q}_{\mu}$-$L^{p}_{\gamma}$-regularity problem for \eqref{PIBVP:eq:rigorous_problem} under the assumption that $X$ is a UMD space and under suitable assumptions on the operators $\mathcal{A}(D),\mathcal{B}_{1}(D),\ldots,\mathcal{B}_{n}(D)$.

\subsection{Assumptions on $(\mathcal{A},\mathcal{B}_{1},\ldots,\mathcal{B}_{n})$}\label{PIBVP:subsec:assumptions_operators}

As in \cite{DHP2,MeySchnau2}, we will pose two type of conditions on the operators $\mathcal{A},\mathcal{B}_{1},\ldots,\mathcal{B}_{n}$ for which we can solve the maximal $L^{q}_{\mu}$-$L^{p}_{\gamma}$-regularity problem for \eqref{PIBVP:eq:rigorous_problem}: smoothness assumptions on the coefficients and structural assumptions.

In order to describe the smoothness assumptions on the coefficients,
let $q,p \in (1,\infty)$, $\mu \in (-1,q-1)$, $\gamma \in (-1,p-1)$ and put
\begin{equation}\label{PBIBVP:eq:kappa}
\kappa_{j,\gamma} := 1-\frac{n_{j}}{2n}-\frac{1}{2np}(1+\gamma) \in (0,1), \quad\quad j=1,\ldots,n.
\end{equation}

\begin{itemize}
\item[\textbf{$(\mathrm{SD})$}] For $|\alpha| = 2n$ we have $a_{\alpha} \in BUC(\mathscr{O} \times J;\mathcal{B}(X))$ and for $|\alpha| < 2n$ we have $a_{\alpha} \in L^{\infty}(\mathscr{O} \times J ;\mathcal{B}(X))$. If $\mathscr{O}$ is unbounded, the limits $a_{\alpha}(\infty,t) := \lim_{|x| \to \infty}a_{\alpha}(x,t)$ exist uniformly with respect to $t \in J$, $|\alpha|=2n$.
\item[\textbf{$(\mathrm{SB})$}]
For each $j \in \{1,\ldots,m\}$ and $|\beta| \leq n_{j}$ there exist $s_{j,\beta} \in [q,\infty)$ and $r_{j,\beta} \in [p,\infty)$ with
\[
\kappa_{j,\gamma} > \frac{1}{s_{j,\beta}} + \frac{d-1}{2nr_{j,\beta}} + \frac{|\beta|-n_{j}}{2n}
\quad \mbox{and} \quad \mu > \frac{q}{s_{j,\beta}}-1
\]
such that
\[
b_{j,\beta} \in F^{\kappa_{j,\gamma}}_{s_{j,\beta},p}(J;L^{r_{j,\beta}}(\partial\mathscr{O};\mathcal{B}(X))) \cap L^{s_{j,\beta}}(J;B^{2n\kappa_{j,\gamma}}_{r_{j,\beta},p}(\partial\mathscr{O};\mathcal{B}(X))).
\]
If $\mathscr{O}=\R^{d}_{+}$, the limits $b_{j,\beta}(\infty,t) := \lim_{|x'| \to \infty}b_{j,\beta}(x',t)$ exist uniformly with respect to $t \in J$, $j \in \{1,\ldots,n\}$, $|\beta|=n_{j}$.
\end{itemize}

\begin{remark}\label{PBIVB:rmk:lower_order}
For the lower order parts of $(\mathcal{A},\mathcal{B}_{1},\ldots,\mathcal{B}_{n})$ we only need $a_{\alpha}D^{\alpha}$, $|\alpha| < 2n$, and $b_{j,\beta}\mathrm{tr}_{\partial\mathscr{O}}D^{\beta}$, $|\beta_{j}|<n_{j}$, $j=1,\ldots,n$, to act as lower order perturbations in the sense that there exists $\sigma \in [2n-1,2n)$ such that
$a_{\alpha}D^{\alpha}$ respectively $b_{j,\beta}\mathrm{tr}_{\partial\mathscr{O}}D^{\beta}$ is bounded from
\[
H^{\frac{\sigma}{2n}}_{q}(J,v_{\mu};L^{p}(\mathscr{O},w^{\partial\mathscr{O}}_{\gamma};X)) \cap L^{q}(J,v_{\mu};H^{\sigma}_{p}(\mathscr{O},w^{\partial\mathscr{O}}_{\gamma};X))
\]
to $L^{q}(J,v_{\mu};L^{p}(\mathscr{O},w^{\partial\mathscr{O}}_{\gamma};X)))$ respectively $F^{\kappa_{j,\gamma}}_{q,p}(J,v_{\mu};L^{p}(\partial\mathscr{O};X)) \cap L^{q}(J,v_{\mu};F^{2n\kappa_{j,\gamma}}_{p,p}(\partial\mathscr{O};X))$.
Here the latter space is the optimal space of boundary data, see the statement of the main result.
\end{remark}

Let us now turn to the two structural assumptions on $\mathcal{A},\mathcal{B}_{1},\ldots,\mathcal{B}_{n}$.
For each $\phi \in [0,\pi)$ we introduce the conditions $(\mathrm{E})_{\phi}$ and $(\mathrm{LS})_{\phi}$.

The condition $(\mathrm{E})_{\phi}$ is parameter ellipticity.
In order to state it, we denote by the subscript $\#$ the principal part of a differential operator: given a differential operator $P(D)=\sum_{|\gamma| \leq k}p_{\gamma}D^{\gamma}$ of order $k \in \N$, $P_{\#}(D) = \sum_{|\gamma| = k}p_{\gamma}D^{\gamma}$.

\begin{itemize}
\item[\textbf{$(\mathrm{E})_{\phi}$}] For all $t \in \overline{J}$, $x \in \overline{\mathscr{O}}$ and $|\xi|=1$ it holds that $\sigma(\mathcal{A}_{\#}(x,\xi,t)) \subset \Sigma_{\phi}$. If $\mathscr{O}$ is unbounded, then it in addition holds that $\sigma(\mathcal{A}_{\#}(\infty,\xi,t)) \subset \C_{+}$ for all $t \in \overline{J}$ and $|\xi|=1$.
\end{itemize}

The condition $(\mathrm{LS})_{\phi}$ is a condition of Lopatinskii-Shapiro type. Before we can state it, we need to introduce some notation. For each $x \in \partial\mathscr{O}$ we fix an orthogonal matrix $O_{\nu(x)}$ that rotates the outer unit normal $\nu(x)$ of $\partial\mathscr{O}$ at $x$ to $(0,\ldots,0,-1) \in \R^{d}$, and define the rotated operators $(\mathcal{A}^{\nu},\mathcal{B}^{\nu})$ by
\[
\mathcal{A}^{\nu}(x,D,t) := \mathcal{A}(x,O_{\nu(x)}^{T}D,t), \quad
\mathcal{B}^{\nu}(x,D,t) := \mathcal{B}(x,O_{\nu(x)}^{T}D,t).
\]

\begin{itemize}
\item[\textbf{$(\mathrm{LS})_{\phi}$}] For each $t \in \overline{J}$, $x \in \partial\mathscr{O}$, $\lambda \in \overline{\Sigma}_{\pi-\phi}$ and $\xi' \in \R^{d-1}$ with $(\lambda,\xi') \neq 0$ and all $h \in X^{n}$, the ordinary initial value problem
    \begin{equation*}\label{pbvp:eqOIVP.(I)}
\begin{array}{rlll}
\lambda w(y) + \mathcal{A}^{\nu}_{\#}(\xi',D_{y},t)w(y) &= 0, & y > 0 & \\
\mathcal{B}^{\nu}_{j,\#}(\xi',D_{y},t)w(y)|_{y=0} &= h_{j}, & j=1,\ldots,n.
\end{array}
\end{equation*}
has a unique solution $w \in C^{\infty}([0,\infty);X)$ with $\lim_{y \to \infty}w(y)=0$.
\end{itemize}

\subsection{Statement of the Main Result}
Let $\mathscr{O}$ be either $\R^{d}_{+}$ or a $C^{\infty}$-domain in $\R^{d}$ with a compact boundary $\partial\mathscr{O}$.
Let $X$ be a Banach space, $q,p \in (1,\infty)$, $\mu \in (-1,q-1)$, $\gamma \in (-1,p-1)$ and $n,n_{1},\ldots,n_{n} \in \N$ natural numbers with $n_{j} \leq 2n - 1$ for each $j \in \{1,\ldots,n\}$, and $\kappa_{1,\gamma},\ldots,\kappa_{n,\gamma} \in (0,1)$ as defined in \eqref{PBIBVP:eq:kappa}.
Put
\begin{equation}
\begin{split}
\I^{p,q}_{\gamma,\mu} &:= B^{2n(1-\frac{1+\mu}{q})}_{p,q}(\mathscr{O},w^{\partial\mathscr{O}}_{\gamma};X),
\quad\mbox{(initial data space)} \\
\G^{p,q}_{\gamma,\mu,j} & := F^{\kappa_{j,\gamma}}_{q,p}(J,v_{\mu};L^{p}(\partial\mathscr{O};X)) \cap L^{q}(J,v_{\mu};F^{2n\kappa_{j,\gamma}}_{p,p}(\partial\mathscr{O};X)), \quad j=1,\ldots,n, \\
\G^{p,q}_{\gamma,\mu}  &:= \G^{p,q}_{1,\mu,\gamma} \oplus \ldots \oplus \G^{p,q}_{n,\mu,\gamma}. \quad\mbox{(space of boundary data $g$)}
\end{split}
\end{equation}
Furthermore, let $\U^{p,q}_{\gamma,\mu}$ and $\F^{p,q}_{\gamma,\mu}$ be as in \eqref{PIBVP:eq:basic_spaces_max-reg}.

\begin{thm}\label{PIBVP:thm:main_result}
Let the notations be as above. Suppose that $X$ is a UMD space, that $\mathcal{A}(D),\mathcal{B}_{1}(D),\ldots,\mathcal{B}_{n}(D)$ satisfy the conditions $(\mathrm{SD})$, $(\mathrm{SB})$, $(\mathrm{E})_{\phi}$ and $(\mathrm{LS})_{\phi}$ for some $\phi \in (0,\frac{\pi}{2})$, and that
$\kappa_{j,\gamma} \neq \frac{1+\mu}{q}$ for all $j \in \{1,\ldots,n\}$.
Put
\[
\D^{p,q}_{\gamma,\mu} := \left\{ (g,u_{0}) \in \G^{p,q}_{\gamma,\mu} \oplus \I^{p,q}_{\gamma,\mu} : \mathrm{tr}_{t=0}g_{j} - \mathcal{B}^{t=0}_{j}(D)u_{0} = 0 \:\:\mbox{when}\:\: \kappa_{j,\gamma} > \frac{1+\mu}{q} \right\},
\]
where $\mathcal{B}^{t=0}_{j}(D) := \sum_{|\beta| \leq n_{j}}b_{j,\beta}(0,\,\cdot\,)\mathrm{tr}_{\partial\mathscr{O}}D^{\beta}$.
Then the problem \eqref{PIBVP:eq:rigorous_problem} enjoys the property of maximal $L^{q}_{\mu}$-$L^{p}_{\gamma}$-regularity with $\D^{p,q}_{\gamma,\mu}$
as the optimal space of initial-boundary data, i.e., the problem \eqref{PIBVP:eq:rigorous_problem}
admits a unique solution $u \in \U^{p,q}_{\gamma,\mu}$ if and only if $(f,g,u_{0}) \in \F^{p,q}_{\gamma,\mu} \oplus \D^{p,q}_{\gamma,\mu}$.
Moreover, the corresponding solution operator $\mathscr{S}: \F^{p,q}_{\gamma,\mu} \oplus \D^{p,q}_{\gamma,\mu} \longra \U^{p,q}_{\gamma,\mu}$ is an isomorphism of Banach spaces.
\end{thm}

\begin{remark}\label{PIBVP:rmk:thm:main_result;compatability_cond}
The compatibility condition $\mathrm{tr}_{t=0}g_{j} - \mathcal{B}^{t=0}_{j}(D)u_{0} = 0$ in the definition of $\D^{p,q}_{\gamma,\mu}$ is basically imposed when $(g_{j},u_{0}) \mapsto \mathrm{tr}_{t=0}g_{j} - \mathcal{B}^{t=0}_{j}(D)u_{0}$ makes sense as a continuous linear operator from $\G^{p,q}_{\gamma,\mu,j} \oplus \I^{p,q}_{\gamma,\mu}$ to some topological vector space $V$.
That it is indeed a well-defined continuous linear operator from $\G^{p,q}_{\gamma,\mu,j} \oplus \I^{p,q}_{\gamma,\mu}$ to $L^{0}(\partial\mathscr{O};X)$ when $\kappa_{j,\gamma} > \frac{1+\mu}{q}$ can be seen by combining the following two points:
\begin{itemize}
\item[(i)] Suppose $\kappa_{j,\gamma} > \frac{1+\mu}{q}$. Then the condition $(\mathrm{SB})$ yields $b_{j,\beta} \in F^{\kappa_{j,\gamma}}_{s_{j,\beta},p}(J;L^{r_{j,\beta}}(\mathscr{O};\mathcal{B}(X)))$ with $\kappa_{j,\gamma} > \frac{1+\mu}{q} > \frac{1}{s_{j,\beta}}$. By \cite[Proposition~7.4]{Meyries&Veraar_sharp_embd_power_weights},
    \[
    F^{\kappa_{j,\gamma}}_{s_{j,\beta},p}(J;L^{r_{j,\beta}}(\mathscr{O};\mathcal{B}(X))) \hookrightarrow BUC(J;L^{r_{j,\beta}}(\mathscr{O};\mathcal{B}(X))).
    \]
    Furthermore, it holds that $2n(1-\frac{1+\mu}{q}) > n_{j} + \frac{1+\gamma}{q}$, so each $\mathrm{tr}_{\partial\mathscr{O}}D^{\beta}$, $|\beta| \leq n_{j}$, is a continuous linear operator from $\I^{p,q}_{\gamma,\mu}$ to
    $B^{2n(1-\frac{1+\mu}{q})-n_{j}-\frac{1+\gamma}{p}}_{p,q}(\partial\mathscr{O};X) \hookrightarrow L^{p}(\partial\mathscr{O};X)$ by the trace theory from Section~\ref{PIBVP:subsec:sec:traces;isotropic}.
    Therefore,
    $\mathcal{B}^{t=0}_{j}(D) = \sum_{|\beta| \leq n_{j}}b_{j,\beta}(0,\,\cdot\,)\mathrm{tr}_{\partial\mathscr{O}}D^{\beta}$ makes sense as a continuous linear operator from $\I^{p,q}_{\gamma,\mu}$ to $L^{0}(\partial\mathscr{O};X)$.
\item[(ii)] Suppose $\kappa_{j,\gamma} > \frac{1+\mu}{q}$. The observation that
\[
\G^{p,q}_{\gamma,\mu,j}  \hookrightarrow
F^{\kappa_{j,\gamma}}_{q,p}(J,v_{\mu};L^{p}(\partial\mathscr{O};X))
\]
 in combination with the trace theory from Section~\ref{PIBVP:subsec:sec:traces;isotropic} yields that $\mathrm{tr}_{t=0}$ is a well-defined continuous linear operator from $\G^{p,q}_{\gamma,\mu,j}$ to $L^{p}(\partial\mathscr{O};X) \hookrightarrow L^{0}(\partial\mathscr{O};X)$.
\end{itemize}
\end{remark}

Notice the dependence of the space of initial-boundary data on the weight parameters $\mu$~and~$\gamma$.  For fixed $q,p \in  (1,\infty)$ we can roughly speaking decrease the required smoothness (or regularity) of $g$ and $u_{0}$ by increasing $\gamma$ and $\mu$, respectively. Furthermore, compatibility conditions can be avoided by choosing $\mu$ and $\gamma$ big enough. So the weights make it possible to solve \eqref{PIBVP:eq:rigorous_problem} for more initial-boundary data (compared to the unweighed setting). On the other hand, by choosing $\mu$ and $\gamma$ closer to $-1$ (depending on the initial-boundary data) we can find more information about the behavior of $u$ near the initial-time and near the boundary, respectively.

The dependence on the weight parameters $\mu$ and $\gamma$ is illustrated in the following example of the heat equation with Dirichlet and Neumann boundary conditions:
\begin{ex}\label{PIBVP:ex:thm:main_result;example}
Let $N \in \N$ and let $p,q,\gamma,\mu$ be as above.
\begin{itemize}
\item[(i)] \emph{The heat equation with Dirichlet boundary condition:}

If $2-\frac{2}{q}(1+\mu) \neq \frac{1}{p}(1+\gamma)$, then the problem
\[
\begin{array}{rll}
\partial_{t}u -\Delta u &= f,  \\
\mathrm{tr}_{\partial\mathscr{O}}u &= g,  \\
u(0) &= u_{0},
\end{array}
\]
has a unique solution $u \in W^{1}_{q}(J,v_{\mu};L^{p}(\mathscr{O},w^{\partial\mathscr{O}}_{\gamma};\C^{N})) \cap L^{q}(J,v_{\mu};W^{2}_{p}(\mathscr{O},w^{\partial\mathscr{O}}_{\gamma};\C^{N}))$ if and only the data $(f,g,u_{0})$ satisfy:
\begin{itemize}
\item[$\bullet$] $f \in L^{q}(J,v_{\mu};L^{p}(\mathscr{O},w^{\partial\mathscr{O}}_{\gamma};\C^{N}))$;
\item[$\bullet$] $g \in F^{1-\frac{1}{2p}(1+\gamma)}_{q,p}(J,v_{\mu};L^{p}(\partial\mathscr{O};\C^{N})) \cap L^{q}(J,v_{\mu};F^{2-\frac{1}{p}(1+\gamma)}_{p,p}(\partial\mathscr{O};\C^{N}))$;
\item[$\bullet$] $u_{0} \in B^{2-\frac{2}{q}(1+\mu)}_{p,q}(\mathscr{O},w^{\partial\mathscr{O}}_{\gamma};\C^{N})$;
\item[$\bullet$] $\mathrm{tr}_{t=0}g = \mathrm{tr}_{\partial\mathscr{O}}u_{0}$ when $2-\frac{2}{q}(1+\mu) > \frac{1}{p}(1+\gamma)$.
\end{itemize}

\item[(ii)] \emph{The heat equation with Neumann boundary condition:}

If $1-\frac{2}{q}(1+\mu) \neq \frac{1}{p}(1+\gamma)$, then the problem
\[
\begin{array}{rll}
\partial_{t}u  -\Delta u &= f,  \\
\partial_{\nu}u &= g,  \\
u(0) &= u_{0},
\end{array}
\]
has a unique solution $u \in W^{1}_{q}(J,v_{\mu};L^{p}(\mathscr{O},w^{\partial\mathscr{O}}_{\gamma};\C^{N})) \cap L^{q}(J,v_{\mu};W^{2}_{p}(\mathscr{O},w^{\partial\mathscr{O}}_{\gamma};\C^{N}))$ if and only the data $(f,g,u_{0})$ satisfy:
\begin{itemize}
\item[$\bullet$] $f \in L^{q}(J,v_{\mu};L^{p}(\mathscr{O},w^{\partial\mathscr{O}}_{\gamma};\C^{N}))$;
\item[$\bullet$] $g \in F^{\frac{1}{2}-\frac{1}{2p}(1+\gamma)}_{q,p}(J,v_{\mu};L^{p}(\partial\mathscr{O};\C^{N})) \cap L^{q}(J,v_{\mu};F^{1-\frac{1}{p}(1+\gamma)}_{p,p}(\partial\mathscr{O};\C^{N}))$;
\item[$\bullet$] $u_{0} \in B^{2-\frac{2}{q}(1+\mu)}_{p,q}(\mathscr{O},w^{\partial\mathscr{O}}_{\gamma};\C^{N})$;
\item[$\bullet$] $\mathrm{tr}_{t=0}g = \mathrm{tr}_{\partial\mathscr{O}}u_{0}$ when $1-\frac{2}{q}(1+\mu) > \frac{1}{p}(1+\gamma)$.
\end{itemize}
\end{itemize}
\end{ex}

\section{Trace Theory}\label{PIBVP:sec:traces}

In this section we establish the necessary trace theory for the maximal $L^{q}_{\mu}$-$L^{p}_{\gamma}$-regularity problem for \eqref{PIBVP:eq:rigorous_problem}.

\subsection{Traces of Isotropic Spaces}\label{PIBVP:subsec:sec:traces;isotropic}

In this subsection we state trace results for the isotropic spaces, for which we refer to \cite{LMV_interpolation_boundary_cond} (also see the references there).
Note that these are of course special cases of the more general anisotropic mixed-norm spaces, for which trace theory (for the model problem case of a half-space) can be found in the next subsections and in \cite{Lindemulder_master-thesis}.

The following notation will be convenient:
\[
\partial B^{s}_{p,q,\gamma}(\partial\mathscr{O};X) := B^{s-\frac{1+\gamma}{p}}_{p,q}(\partial\mathscr{O};X)
\quad\quad \mbox{and} \quad\quad
\partial F^{s}_{p,q,\gamma}(\partial\mathscr{O};X) := F^{s-\frac{1+\gamma}{p}}_{p,p}(\partial\mathscr{O};X).
\]

\begin{prop}\label{PIBVP:prop:trace_isotropic_B&F}
Let $X$ be a Banach space, $\mathscr{O} \subset \R^{d}$ either $\R^{d}_{+}$ or a $C^{\infty}$-domain in $\R^{d}$ with a compact boundary $\partial\mathscr{O}$, $\mathscr{A} \in \{B,F\}$, $p \in [1,\infty)$, $q \in [1,\infty]$, $\gamma \in (-1,\infty)$ and $s>\frac{1+\gamma}{p}$.
Then
\[
\mathcal{S}(\R^{d};X) \longra \mathcal{S}(\partial\mathscr{O};X),\,
f \mapsto f_{|\partial\mathscr{O}},
\]
uniquely extends to a retraction $\mathrm{tr}_{\partial\mathscr{O}}$ from $\mathscr{A}^{s}_{p,q}(\R^{d},w^{\partial\mathscr{O}}_{\gamma};X)$ onto $\partial \mathscr{A}^{s}_{p,q,\gamma}(\partial\mathscr{O};X)$.
There is a universal coretraction in the sense that there exists an operator $\mathrm{ext}_{\partial\mathscr{O}} \in \mathcal{L}(\mathcal{S}'(\partial\mathscr{O};X),\mathcal{S}'(\R^{d};X))$ (independent of $\mathscr{A},p,q,\gamma,s$) which restricts to a coretraction for the operator $\mathrm{tr}_{\partial\mathscr{O}} \in \mathcal{B}(\mathscr{A}^{s}_{p,q}(\R^{d},w^{\partial\mathscr{O}}_{\gamma};X),\partial \mathscr{A}^{s}_{p,q,\gamma}(\partial\mathscr{O};X))$.
The same statements hold true with $\R^{d}$ replaced by $\mathscr{O}$.
\end{prop}

\begin{remark}\label{PIBVP:rmk:prop:trace_isotropic_B&F}
Recall that $\mathcal{S}(\R^{d};X)$ is dense in $\mathscr{A}^{s}_{p,q}(\R^{d},w^{\partial\mathscr{O}}_{\gamma};X)$ for $q<\infty$ but not for $q=\infty$. For $q=\infty$ uniqueness of the extension follows from the trivial embedding $\mathscr{A}^{s}_{p,\infty}(\R^{d},w^{\partial\mathscr{O}}_{\gamma};X) \hookrightarrow B^{s-\epsilon}_{p,1}(\R^{d},w^{\partial\mathscr{O}}_{\gamma};X)$, $\epsilon > 0$.
\end{remark}

\begin{cor}\label{PIBVP:cor:prop:trace_isotropic_B&F;W&H}
Let $X$ be a Banach space, $\mathscr{O} \subset \R^{d}$ either $\R^{d}_{+}$ or a $C^{\infty}$-domain in $\R^{d}$ with a compact boundary $\partial\mathscr{O}$, $p \in (1,\infty)$, $\gamma \in (-1,p-1)$, $n \in \N_{>0}$ and $s > \frac{1+\gamma}{p}$. Then \[
\mathcal{S}(\R^{d};X) \longra \mathcal{S}(\partial\mathscr{O};X),\,
f \mapsto f_{|\partial\mathscr{O}},
\]
uniquely extends to retractions $\mathrm{tr}_{\partial\mathscr{O}}$ from $W^{n}_{p}(\R^{d},w^{\partial\mathscr{O}}_{\gamma};X)$ onto $F^{n-\frac{1+\gamma}{p}}_{p,p}(\partial\mathscr{O};X)$
and from $W^{s}_{p}(\R^{d},w^{\partial\mathscr{O}}_{\gamma};X)$ onto $F^{s-\frac{1+\gamma}{p}}_{p,p}(\partial\mathscr{O};X)$.
The same statement holds true with $\R^{d}$ replaced by $\mathscr{O}$.
\end{cor}

\subsection{Traces of Intersection Spaces}

For the maximal $L^{q}_{\mu}$-$L^{p}_{\gamma}$-regularity problem for \eqref{PIBVP:eq:rigorous_problem} we need to determine the temporal and spatial trace spaces of Sobolev and Bessel potential spaces of intersection type.
As the temporal trace spaces can be obtained from the trace results in  \cite{Meyries&Veraar_Traces}, we will focus on the spatial traces.

By the trace theory of the previous subsection, the trace operator $\mathrm{tr}_{\partial\mathscr{O}}$ can be defined pointwise in time on the intersection spaces in the following theorem.
It will be convenient to use the notation $\mathrm{tr}_{\partial\mathscr{O}}[\E]=\F$ to say that $\mathrm{tr}_{\partial\mathscr{O}}$ is a retraction from $\E$ onto $\F$.

\begin{thm}\label{PIBVP:thm:traces_main_result}
Let $\mathscr{O}$ be either $\R^{d}_{+}$ or a $C^{\infty}$-domain in $\R^{d}$ with a compact boundary $\partial\mathscr{O}$.
Let $X$ be a Banach space, $Y$ a UMD Banach space, $p,q \in (1,\infty)$, $\mu \in (-1,q-1)$ and $\gamma \in (-1,p-1)$.
If $n,m \in \Z_{>0}$ and $r,s \in (0,\infty)$ with $s > \frac{1+\gamma}{p}$, then
\begin{equation}\label{PIBVP:eq:traces_int_space;Sobolev}
\begin{split}
\mathrm{tr}_{\partial\mathscr{O}}\left[ W^{n}_{q}(J,v_{\mu};L^{p}(\mathscr{O},w^{\partial\mathscr{O}}_{\gamma};X)) \cap L^{q}(J,v_{\mu};W^{m}_{p}(\mathscr{O},w^{\partial\mathscr{O}}_{\gamma};X)) \right] = \\
\quad\quad\quad\quad\quad\quad F^{n-\frac{n}{m}\frac{1+\gamma}{p}}_{q,p}(J,v_{\mu};L^{p}(\partial\mathscr{O};X)) \cap L^{q}(J,v_{\mu};F^{m-\frac{1+\gamma}{p}}_{p,p}(\partial\mathscr{O};X))
\end{split}
\end{equation}
and
\begin{equation}\label{PIBVP:eq:traces_int_space;Bessel_potential}
\begin{split}
\mathrm{tr}_{\partial\mathscr{O}}\left[ H^{r}_{q}(J,v_{\mu};L^{p}(\mathscr{O},w^{\partial\mathscr{O}}_{\gamma};Y)) \cap L^{q}(J,v_{\mu};H^{s}_{p}(\mathscr{O},w^{\partial\mathscr{O}}_{\gamma};Y)) \right] = \\
\quad\quad\quad\quad\quad\quad F^{r-\frac{r}{s}\frac{1+\gamma}{p}}_{q,p}(J,v_{\mu};L^{p}(\partial\mathscr{O};Y)) \cap L^{q}(J,v_{\mu};F^{s-\frac{1+\gamma}{p}}_{p,p}(\partial\mathscr{O};Y)).
\end{split}
\end{equation}
\end{thm}

The main idea behind the proof of Theorem~\ref{PIBVP:thm:traces_main_result} is, as in \cite{Scharf&Schmeisser&Sickel_Traces_vector-valued_Sobolev}, to exploit the independence of the trace space of a Triebel-Lizorkin space on its microscopic parameter.
As in \cite{Scharf&Schmeisser&Sickel_Traces_vector-valued_Sobolev}, our approach does not require any restrictions on the Banach space $X$.

The UMD restriction on $Y$ comes from the localization procedure for Bessel potential spaces used in the proof, which can be omitted in the case $\mathscr{O} = \R^{d}_{+}$. This localization procedure for Bessel potential spaces could be replaced by a localization procedure for weighted anisotropic mixed-norm Triebel-Lizorkin spaces, which would not require any restrictions on the Banach space $Y$. However, we have chosen to avoid this as localization of such Triebel-Lizorkin spaces has not been considered in the literature before while we do not need that generality anyway. For localization in the scalar-valued isotropic non-mixed norm case we refer to \cite{LMV_interpolation_boundary_cond}.

\begin{proof}[Proof of Theorem~\ref{PIBVP:thm:traces_main_result}]
By standard techniques of localization, it suffices to consider the case $\mathscr{O} = \R^{d}_{+}$ with boundary $\partial\mathscr{O} = \R^{d-1}$.
Moreover, using a standard restriction argument, we may turn to the corresponding trace problem on the full space $\mathscr{O} \times J = \R^{d} \times \R$.

From the natural identifications
\[
W^{n}_{q,\mu}(L^{p}_{\gamma}) \cap L^{q}_{\mu}(W^{m}_{p,\gamma}) = W^{(m,n)}_{(p,q),(d,1)}(\R^{d+1},(w_{\gamma},v_{\mu});X)
\]
and
\[
H^{r}_{q,\mu}(L^{p}_{\gamma}) \cap L^{q}_{\mu}(H^{s}_{p,\gamma}) = H^{(s,r)}_{(p,q),(d,1)}(\R^{d+1},(w_{\gamma},v_{\mu});Y),
\]
\eqref{PIBVP:eq:elem_embedding_FEF} and Corollary~\ref{PIBVP:cor:thm:trace_TL} it follows that
\[
\mathrm{tr}\,[W^{n}_{q,\mu}(L^{p}_{\gamma}) \cap L^{q}_{\mu}(W^{m}_{p,\gamma})] = F^{1-\frac{1}{m}\frac{1+\gamma}{p},(\frac{1}{m},\frac{1}{n})}_{(p,q),p,(d-1,1)}(\R^{d},(1,v_{\mu});X)
\]
and
\[
\mathrm{tr}\,[H^{r}_{q,\mu}(L^{p}_{\gamma}) \cap L^{q}_{\mu}(H^{s}_{p,\gamma}) ] = F^{1-\frac{1}{s}\frac{1+\gamma}{p},(\frac{1}{s},\frac{1}{r})}_{(p,q),p,(d-1,1)}(\R^{d},(1,v_{\mu});Y).
\]
An application of Theorem~\ref{functieruimten:thm:aTL_rep_intersection} finishes the proof.
\end{proof}

\subsection{Traces of Anisotropic Mixed-Norm Spaces}\label{PIBVP:subsec:traces_anisotropic}
The goal of this subsection is to prove the trace result Theorem~\ref{PIBVP:thm:trace_TL},
which is a weighted vector-valued version of \cite[Theorem~2.2]{JS_traces}.

In contrast to Theorem~\ref{PIBVP:thm:trace_TL},
the trace result \cite[Theorem~2.2]{JS_traces} is formulated for the distributional trace operator; see Remark~\ref{PIBVP:rmk:thm:trace_TL;distr_trace} for more information.
However, all estimates in the proof of that result are carried out for the 'working definition of the trace'.
The proof of Theorem~\ref{PIBVP:thm:trace_TL} presented below basically consists of modifications of these estimates to our setting.
As this can get quite technical at some points, we have decided to give the proof in full detail.

\subsubsection{The working definition of the trace}
Let $\varphi \in \Phi^{\mathpzc{d},a}(\R^{d})$ with associated family of convolution operators $(S_{n})_{n \in \N} \subset \mathcal{L}(\mathcal{S}'(\R^{d};X))$ be fixed.
In order to motivate the definition to be given in a moment, let us first recall that $f = \sum_{n=0}^{\infty}S_{n}f$ in $\mathcal{S}(\R^{d};X)$ (respectively in $\mathcal{S}'(\R^{d};X)$) whenever $f \in \mathcal{S}(\R^{d};X)$ (respectively $f \in \mathcal{S}'(\R^{d};X)$), from which it is easy to see that
\[
f_{|\{0\} \times \R^{d-1}} = \sum_{n=0}^{\infty}(S_{n}f)_{|\{0\} \times \R^{d-1}} \:\:\mbox{in}\:\:\mathcal{S}(\R^{d-1};X), \quad\quad f \in \mathcal{S}(\R^{d};X).
\]
Furthermore, given a general tempered distribution $f \in \mathcal{S}'(\R^{d};X)$, recall that $S_{n}f \in \mathscr{O}_{M}(\R^{d};X)$; in particular, each $S_{n}f$ has a well defined classical trace with respect to $\{0\}  \times \R^{d-1}$.
This suggests to define the trace operator $\tau = \tau^{\varphi}: \mathcal{D}(\gamma^{\varphi}) \subset \mathcal{S}'(\R^{d};X) \longra \mathcal{S}'(\R^{d-1};X)$ by
\begin{equation}\label{PIBVP:eq:working_def_trace}
\tau^{\varphi}f := \sum_{n=0}^{\infty}(S_{n}f)_{|\{0\} \times \R^{d-1}}
\end{equation}
on the domain $\mathcal{D}(\tau^{\varphi})$ consisting of all $f \in \mathcal{S}'(\R^{d};X)$ for which this defining series
converges in $\mathcal{S}'(\R^{d-1};X)$. Note that $\mathscr{F}^{-1}\mathcal{E}'(\R^{d};X)$ is a subspace of $\mathcal{D}(\tau^{\varphi})$ on which $\tau^{\varphi}$ coincides with the classical trace of continuous functions with respect to $\{0\} \times \R^{d-1}$; of course, for an $f$ belonging to $\mathscr{F}^{-1}\mathcal{E}'(\R^{d};X)$ there are only finitely many $S_{n}f$ non-zero.

\subsubsection{The distributional trace operator}
Let us now introduce the concept of distributional trace operator.
The reason for us to introduce it is the right inverse from Lemma~\ref{PIBVP:prop:right_inverse_distr_trace}.

The distributional trace operator $r$ (with respect to the hyperplane $\{0\} \times \R^{d-1}$) is defined as follows.
Viewing $C(\R;\mathcal{D}'(\R^{d-1};X))$ as subspace of
$\mathcal{D}'(\R^{d};X) = \mathcal{D}'(\R \times \R^{d-1};X)$ via the canonical identification $\mathcal{D}'(\R;\mathcal{D}'(\R^{d-1};X)) = \mathcal{D}'(\R \times \R^{d-1};X)$ (arising from the Schwartz kernel theorem),
\[
C(\R;\mathcal{D}'(\R^{d-1};X)) \hookrightarrow  \mathcal{D}'(\R;\mathcal{D}'(\R^{d-1};X)) = \mathcal{D}'(\R \times \R^{d-1};X),
\]
we define $r \in \mathcal{L}(C(\R;\mathcal{D}'(\R^{d-1};X)),\mathcal{D}'(\R^{d-1};X))$ as the 'evaluation in $0$ map'
\begin{equation*}
r: C(\R;\mathcal{D}'(\R^{d-1};X)) \longra \mathcal{D}'(\R^{d-1};X),\,f \mapsto \mathrm{ev}_{0}f.
\end{equation*}
Then, in view of
\[
C(\R^{d};X) = C(\R \times \R^{d-1};X) = C(\R;C(\R^{d-1};X)) \hookrightarrow C(\R;\mathcal{D}'(\R^{d-1};X)),
\]
we have that the distributional trace operator $r$ coincides on $C(\R^{d};X)$ with the classical trace operator with respect to the hyperplane $\{0\} \times \R^{d-1}$, i.e.,
\begin{equation*}
r: C(\R^{d};X) \longra C(\R^{d-1};X),\,f \mapsto f_{| \{0\} \times \R^{d-1}}.
\end{equation*}

The following lemma can be established as in \cite[Section~4.2.1]{JS_traces}.
\begin{lemma}\label{PIBVP:prop:right_inverse_distr_trace}
Let $\rho \in \mathcal{S}(\R)$ such that $\rho(0) = 1$ and $\supp \hat{\rho} \subset [1,2]$, $a_{1} \in \R$, $\tilde{\mathpzc{d}} \in (\Z_{>0})^{l-1}$ with $\mathpzc{d}=(1,\tilde{\mathpzc{d}})$, $\tilde{\vec{a}} \in (0,\infty)^{l-1}$, and
$(\phi_{n})_{n \in \N} \in \Phi^{\tilde{\mathpzc{d}},\tilde{\vec{a}}}(\R^{d-1})$.
Then, for each $g \in \mathcal{S}'(\R^{d-1};X)$,
\begin{equation}\label{functieruimten:eq:prop;right_inverse_distr_trace_formula}
\mathrm{ext}\,g := \sum_{n=0}^{\infty} \rho(2^{na_{1}}\,\cdot\,) \otimes [\phi_{n}*g]
\end{equation}
defines a convergent series in $\mathcal{S}'(\R^{d};X)$ with
\begin{equation}\label{functieruimten:eq:prop;right_inverse_distr_trace_supports}
\begin{array}{l}
\supp \mathscr{F}[\rho \otimes [\phi_{0}*g]] \subset \{ \xi \mid |\xi|_{\mathpzc{d},a} \leq c \} \\
\supp \mathscr{F}[\rho(2^{na_{1}}\,\cdot\,) \otimes [\phi_{n}*g]] \subset \{ \xi \mid c^{-1}2^{n} \leq |\xi|_{\mathpzc{d},a} \leq c2^{n} \}  \:\:, n \geq 1,
\end{array}
\end{equation}
for some constant $c>0$ independent of $g$.
Moreover, the operator $\mathrm{ext}$ defined via this formula is a linear operator
\[
\mathrm{ext}:\mathcal{S}'(\R^{d-1};X) \longra C_{b}(\R;\mathcal{S}'(\R^{d-1};X))
\]
which acts as a right inverse of $r:C(\R;\mathcal{S}'(\R^{d-1};X)) \longra \mathcal{S}'(\R^{d-1};X)$.
\end{lemma}

\subsubsection{Trace spaces of Triebel-Lizorkin, Sobolev and Bessel potential spaces}

\begin{thm}\label{PIBVP:thm:trace_TL}
Let $X$ be a Banach space, $\mathpzc{d}_{1}=1$, $\vec{a} \in (0,\infty)^{l}$, $\vec{p} \in [1,\infty)^{l}$, $q \in [1,\infty]$, $\gamma \in (-1,\infty)$ and $s > \frac{a_{1}}{p_{1}}(1+\gamma)$. Let $\vec{w} \in \prod_{j=1}^{l}A_{\infty}(\R^{\mathpzc{d}_{j}})$ be such that
$w_{1}(x_{1}) = w_{\gamma}(x_{1}) = |x_{1}|^{\gamma}$ and $\vec{w}'' \in \prod_{j=2}^{l}A_{p_{j}/r_{j}}(\R^{\mathpzc{d}_{j}})$ for some $\vec{r}''=(r_{2},\ldots,r_{l}) \in (0,1)^{l-1}$ satisfying $s-\frac{a_{1}}{p_{1}}(1+\gamma) > \sum_{j=2}^{l}a_{j}\mathpzc{d}_{j}(\frac{1}{r_{j}}-1)$.\footnote{This technical condition on $\vec{w}''$ is in particular satisfied when $\vec{p}'' \in (1,\infty)^{l-1}$ and $\vec{w}'' \in \prod_{j=2}^{l}A_{p_{j}}(\R^{\mathpzc{d}_{j}})$.}
Then the trace operator $\tau = \tau^{\varphi}$ \eqref{PIBVP:eq:working_def_trace} is well-defined on
$F_{\vec{p},q,\mathpzc{d}}^{s,\vec{a}}(\R^{d},(w_{\gamma},\vec{w}'');X)$, where it is independent of $\varphi$, and
restricts to a retraction
\begin{equation}\label{PIBVP:eq:thm:trace_TL;mapping_prop}
\tau: F_{\vec{p},q,\mathpzc{d}}^{s,\vec{a}}(\R^{d},(w_{\gamma},\vec{w}'');X) \longra
F_{\vec{p}'',p_{1},\mathpzc{d}''}^{s-\frac{a_{1}}{p_{1}}(1+\gamma),\vec{a}''}(\R^{d-1},\vec{w}'';X)
\end{equation}
for which the extension operator $\mathrm{ext}$ from Lemma \ref{PIBVP:prop:right_inverse_distr_trace} (with $\tilde{\mathpzc{d}} = \mathpzc{d}''$ and $\tilde{\vec{a}}= \vec{a}''$) restricts to a corresponding coretraction.
\end{thm}

\begin{remark}\label{PIBVP:rmk:thm:trace_TL;trace_density}
In the situation of Theorem~\ref{PIBVP:thm:trace_TL}, suppose that $q < \infty$. Then $\mathcal{S}(\R^{d};X)$ is a dense linear subspace of $F_{\vec{p},q,\mathpzc{d}}^{s,\vec{a}}(\R^{d},(w_{\gamma},\vec{w}'');X)$ and $\tau$ is just the unique extension of the classical trace operator
\[
\mathcal{S}(\R^{d};X) \longra \mathcal{S}(\R^{d-1};X),\, f \mapsto f_{|\{0\} \times \R^{d-1}},
\]
to a bounded linear operator \eqref{PIBVP:eq:thm:trace_TL;mapping_prop}.
\end{remark}

\begin{remark}\label{PIBVP:rmk:thm:trace_TL;distr_trace}
In contrary to the unweighted case considered in \cite{JS_traces}, one cannot use translation arguments to show that
\[
F_{\vec{p},q,\mathpzc{d}}^{s,\vec{a}}(\R^{d},(w_{\gamma},\vec{w}'');X) \hookrightarrow C(\R;\mathcal{D}'(\R^{d-1};X))
\]
for $s> \frac{a_{1}}{p_{1}}(1+\gamma)$. However, for $s> \frac{a_{1}}{p_{1}}(1+\gamma_{+})$, $\vec{p} \in (1,\infty)^{l}$ and $\vec{w}'' \in \prod_{j=2}^{l}A_{p_{j}}(\R^{\mathpzc{d}_{j}})$, the inclusion
\[
F_{\vec{p},q,\mathpzc{d}}^{s,\vec{a}}(\R^{d},(w_{\gamma},\vec{w}'');X) \hookrightarrow C(\R;\mathcal{S}'(\R^{d-1};X))
\]
can be obtained as follows: picking $\tilde{s}$ with $s > \tilde{s} > \frac{a_{1}}{p_{1}}(1+\gamma_{+})$, there holds the chain of inclusions
\begin{eqnarray*}
F_{\vec{p},q,\mathpzc{d}}^{s,\vec{a}}(\R^{d},(w_{\gamma},\vec{w}'');X)
&\hookrightarrow&
B_{\vec{p},1,\mathpzc{d}}^{\tilde{s},\vec{a}}(\R^{d},(w_{\gamma},\vec{w}'');X) \\
&\stackrel{\eqref{PIBVP:eq:thm:trace_Besov:distr_trace}}{\hookrightarrow}&
C_{b}(\R,\rho_{p_{1},\gamma};
B_{\vec{p}^{''},1,\mathpzc{d}^{''}}^{\tilde{s}-\frac{a_{1}}{p_{1}}(1+\gamma_{+}),
\vec{a}^{''}}(\R^{d-1},\vec{w}^{''};X)) \\
&\hookrightarrow& C(\R;\mathcal{S}'(\R^{d-1};X)).
\end{eqnarray*}
Here the restriction $s >  \frac{a_{1}}{p_{1}}(1+\gamma_{+})$ when $\gamma < 0$ is natural in view of the necessity of $s>\frac{a_{1}}{p_{1}}$ in the unweighted case with $p_{1}>1$ (cf.\ \cite[Theorem~2.1]{JS_traces}).
\end{remark}

Note that the trace space of the weighted anisotropic Triebel-Lizorkin space is independent of the microscopic parameter $q \in [1,\infty]$.
As a consequence, if $\E$ is a normed space with
\[
F_{\vec{p},1,\mathpzc{d}}^{s,\vec{a}}(\R^{d},(w_{\gamma},\vec{w}'');X) \hookrightarrow \E \hookrightarrow
F_{\vec{p},\infty,\mathpzc{d}}^{s,\vec{a}}(\R^{d},(w_{\gamma},\vec{w}'');X),
\]
then the trace result of Theorem~\ref{PIBVP:thm:trace_TL} also holds for $\E$ in place of $F_{\vec{p},q,\mathpzc{d}}^{s,\vec{a}}(\R^{d},(w_{\gamma},\vec{w}'');X)$.
In particular, we have:

\begin{cor}\label{PIBVP:cor:thm:trace_TL}
Let $X$ be a Banach space, $\mathpzc{d}_{1}=1$, $\vec{a} \in (0,\infty)^{l}$, $\vec{p} \in (1,\infty)^{l}$, $\gamma \in (-1,p_{1}-1)$ and $s > \frac{a_{1}}{p_{1}}(1+\gamma)$. Let $\vec{w} \in \prod_{j=1}^{l}A_{p_{j}}(\R^{\mathpzc{d}_{j}})$ be such that
$w_{1}(x_{1}) = w_{\gamma}(x_{1}) = |x_{1}|^{\gamma}$. Suppose that either
\begin{itemize}
\item $\E = W^{\vec{n}}_{\vec{p},\mathpzc{d}}(\R^{d},(w_{\gamma},\vec{w}'');X)$, $\vec{n} \in (\Z_{\geq 1})^{l}$, $\vec{n}=s\vec{a}^{-1}$; or
\item $\E = H^{s,\vec{a}}_{\vec{p},\mathpzc{d}}(\R^{d},(w_{\gamma},\vec{w}'');X)$; or
\item $\E = H^{\vec{s}}_{\vec{p},\mathpzc{d}}(\R^{d},(w_{\gamma},\vec{w}'');X)$, $\vec{s} \in (0,\infty)^{l}$, $\vec{s}=s\vec{a}^{-1}$.
\end{itemize}
Then the trace operator $\tau = \tau^{\varphi}$ \eqref{PIBVP:eq:working_def_trace} is well-defined on
$\E$, where it is independent of $\varphi$, and restricts to a retraction
\[
\tau: \E \longra
F_{\vec{p}'',p_{1},\mathpzc{d}''}^{s-\frac{a_{1}}{p_{1}}(1+\gamma),\vec{a}''}(\R^{d-1},\vec{w}'';X)
\]
for which the extension operator $\mathrm{ext}$ from Lemma \ref{PIBVP:prop:right_inverse_distr_trace} (with $\tilde{\mathpzc{d}} = \mathpzc{d}''$ and $\tilde{\vec{a}}= \vec{a}''$) restricts to a corresponding coretraction.
\end{cor}

\subsubsection{Traces by duality for Besov spaces}

Let $i \in \{1,\ldots,l\}$. For $b \in \R^{\mathpzc{d}_{i}}$ we define the hyperplane
\[
\Gamma_{[\mathpzc{d};i],b} := \R^{\mathpzc{d}_{1}} \times \R^{\mathpzc{d}_{i-1}} \times \{b\} \times \R^{\mathpzc{d}_{i+1}} \times \R^{\mathpzc{d}_{l}}
\]
and we simply put $\Gamma_{[\mathpzc{d};i]} := \Gamma_{[\mathpzc{d};i],0}$.
Furthermore, given sets $S_{1},\ldots,S_{l}$ and $x = (x_{1},\ldots,x_{l}) \in \prod_{j=1}^{l}S_{j}$, we write $\vec{x}^{[i]} = (x_{1},\ldots,x_{i-1},x_{i+1},\ldots,x_{l})$.

\begin{prop}\label{PIBVP:thm:trace_Besov}
Let $X$ be a Banach space, $i \in \{1,\ldots,l\}$, $\vec{a} \in (0,\infty)^{l}$, $\vec{p} \in (1,\infty)^{l}$, $q \in [1,\infty)$, $\gamma \in (-\mathpzc{d}_{i},\infty)$ and $s > \frac{a_{i}}{p_{i}}(\mathpzc{d}_{i}+\gamma)$. Let $\vec{w} \in \prod_{j=1}^{l}A_{\infty}(\R^{\mathpzc{d}_{j}})$ be such that
$w_{i}(x_{i}) = w_{\gamma}(x_{i}) = |x_{i}|^{\gamma}$ and $w_{j} \in A_{p_{j}}$ for each $j \neq i$.
Then the trace operator
\[
\mathrm{tr}_{[\mathpzc{d};i],b} : \mathcal{S}(\R^{d};X) \longra \mathcal{S}(\R^{d-\mathpzc{d}_{i}};X),\, f \mapsto f_{|\Gamma_{[\mathpzc{d};i]} },
\]
extends to a retraction
\begin{equation}\label{PIBVP:eq:thm:trace_Besov:trace_estimate}
\mathrm{tr}_{[\mathpzc{d};i],b}: B_{\vec{p},q,\mathpzc{d}}^{s,\vec{a}}(\R^{d},\vec{w};X) \longra
B_{\vec{p}^{[i]},q,\mathpzc{d}^{[i]}}^{s-\frac{a_{i}}{p_{i}}(\mathpzc{d}_{i}+\gamma),\vec{a}^{[i]}}(\R^{d-\mathpzc{d}_{i}},
\vec{w}^{[i]};X)
\end{equation}
for which the extension operator $\mathrm{ext}$ from Lemma \ref{PIBVP:prop:right_inverse_distr_trace} (with $\tilde{\mathpzc{d}} = \mathpzc{d}^{[i]}$ and $\tilde{\vec{a}}= \vec{a}^{[i]}$, modified in the obvious way to the $i$-th multidimensional coordinate) restricts to a corresponding coretraction.
Furthermore, if $s > \frac{a_{i}}{p_{i}}(\mathpzc{d}_{i}+\gamma_{+})$, then
\begin{equation}\label{PIBVP:eq:thm:trace_Besov:distr_trace}
B_{\vec{p},q,\mathpzc{d}}^{s,\vec{a}}(\R^{d},\vec{w};X) \hookrightarrow
C_{b}(\R^{\mathpzc{d}_{i}},\rho_{p_{i},\gamma};
B_{\vec{p}^{[i]},q,\mathpzc{d}^{[i]}}^{s-\frac{a_{i}}{p_{i}}(\mathpzc{d}_{i}+\gamma_{+}),
\vec{a}^{[i]}}(\R^{d-1},\vec{w}^{[i]};X))
\hookrightarrow C(\R^{\mathpzc{d}_{i}};\mathcal{S}'(\R^{d-\mathpzc{d}_{i}};X)),
\end{equation}
where $\rho_{p_{i},\gamma} := \max\{|\,\cdot\,|,1\}^{-\frac{\gamma_{-}}{p_{i}}}$.
\end{prop}

\begin{cor}\label{PIBVP:cor:thm:trace_Besov:distr_trace}
Let $X$ be a Banach space, $\vec{a} \in (0,\infty)^{l}$, $\vec{p} \in (1,\infty)^{l}$, $q \in [1,\infty)$, $\vec{\gamma} \in \prod_{j=1}^{l}(-\mathpzc{d}_{j},\infty)$ and $s > \sum_{j=1}^{l}\frac{a_{j}}{p_{j}}(\mathpzc{d}_{j}+\gamma_{j,+})$.
Let $\vec{w} \in \prod_{j=1}^{l}A_{\infty}(\R^{\mathpzc{d}_{j}})$ be such that $w_{j}(x_{j}) = w_{\gamma}(x_{j}) = |x_{j}|^{\gamma}$ for each $j \in \{1,\ldots,l\}$.
Then
\[
B_{\vec{p},q,\mathpzc{d}}^{s,\vec{a}}(\R^{d},\vec{w};X) \hookrightarrow
C_{b}(\R^{\mathpzc{d}_{1}},\rho_{p_{l},\gamma_{l}}; \ldots C_{b}(\R^{\mathpzc{d}_{l}},\rho_{p_{1},\gamma_{1}};X) \ldots ).
\]
\end{cor}
\begin{proof}
Thanks to the Sobolev embedding of Proposition~\ref{PIBVP:Sobolev_embedding_Besov} it is enough to treat the case
$\vec{w} \in \prod_{j=1}^{l}A_{p_{j}}(\R^{\mathpzc{d}_{j}})$, which can be obtained by $l$ iterations of Proposition~\ref{PIBVP:thm:trace_Besov}.
\end{proof}

\begin{remark}
The above proposition and its corollary remain valid for $q=\infty$.
In this case the norm estimate corresponding to \eqref{PIBVP:eq:thm:trace_Besov:trace_estimate} can be obtained in a similar way, from which the unique extendability to a bounded linear operator \eqref{PIBVP:eq:thm:trace_Besov:trace_estimate} can be derived via the Fatou property,  \eqref{PIBVP:prelim:eq:elem_embd_epsilon} and the case $q=1$. The remaining statements can be established in the same way as for the case $q<\infty$.
\end{remark}

\begin{remark}
Note that if $\vec{\gamma} \in [0,\infty)^{l}$ in the situation of the above corollary, then
\[
B_{\vec{p},q,\mathpzc{d}}^{s,\vec{a}}(\R^{d},\vec{w};X) \hookrightarrow BUC(\R^{d};X)
\]
by density of the Schwartz space $\mathcal{S}(\R^{d};X) \subset BUC(\R^{d};X)$ in $B_{\vec{p},q,\mathpzc{d}}^{s,\vec{a}}(\R^{d},\vec{w};X)$.
This could also be established in the standard way by the Sobolev embedding Proposition~\ref{PIBVP:Sobolev_embedding_Besov}, see for instance \cite[Proposition~7.4]{Meyries&Veraar_sharp_embd_power_weights}.
\end{remark}

Let $X$ be a Banach space. Then
\[
[\mathcal{S}'(\R^{d};X)]' = \mathcal{S}(\R^{d};X^{*}) \quad\quad \mbox{and} \quad\quad
[\mathcal{S}(\R^{d};X)]' = \mathcal{S}'(\R^{d};X^{*})
\]
via the pairings induced by
\[
\ip{f \otimes x^{*}}{g \otimes x} = \ip{\ip{f}{x^{*}}}{\ip{g}{x}};
\]
see \cite[Corollary~1.4.10]{Amann_distributions}.

Let $i \in \{1,\ldots,l\}$ and $b \in \R^{\mathpzc{d}_{i}}$.
Let $\mathrm{tr}_{[\mathpzc{d};i],b} \in \mathcal{L}(\mathcal{S}(\R^{d};X),\mathcal{S}(\R^{d-1};X))$ be given by $\mathrm{tr}_{[\mathpzc{d};i],b}\,f := f_{|\Gamma_{[\mathpzc{d};i],b}}$.
Then the adjoint operator $T_{[\mathpzc{d};i],b} := [\mathrm{tr}_{[\mathpzc{d};i],b}]' \in \mathcal{L}(\mathcal{S}'(\R^{d-1};X^{*}),\mathcal{S}'(\R^{d};X^{*}))$ is given by $T_{[\mathpzc{d};i],b}f = \delta_{b} \otimes_{[\mathpzc{d};i]} f$, which can be seen by testing on the dense subspace $\mathcal{S}(\R^{\mathpzc{d}_{i}}) \otimes_{[\mathpzc{d};i]} \mathcal{S}(\R^{d-\mathpzc{d}_{i}})$ of $\mathcal{S}(\R^{d})$.
Now suppose that $\E$ is a locally convex space with $\mathcal{S}(\R^{d};X) \stackrel{d}{\hookrightarrow} \E$ and that $\F$ is a complete locally convex space with $\mathcal{S}(\R^{d-\mathpzc{d}_{i}};X) \stackrel{d}{\hookrightarrow} \F$.
Then $\E' \hookrightarrow \mathcal{S}'(\R^{d};X^{*})$ and $\F' \hookrightarrow \mathcal{S}'(\R^{d-\mathpzc{d}_{i}};X^{*})$ under the natural identifications, and $\mathrm{tr}_{[\mathpzc{d};i],b}$ extends to a continuous linear operator $\mathrm{tr}_{\E \to \F}$ from $\E$ to $\F$ if and only if $T_{[\mathpzc{d};i],b}$ restricts to a continuous linear operator $T_{\F' \to \E'}$ from $\F'$ to $\E'$, in which case $[\mathrm{tr}_{\E \to \F}]' = T_{\F' \to \E'}$.

Estimates in the classical Besov and Triebel-Lizorkin spaces for the tensor product with the one-dimensional delta-distribution $\delta_{0}$ can be found in \cite[Proposition~2.6]{Johnsen1996_Boutet_de_Monvel}, where a different proof is given than the one below.

\begin{lemma}\label{PIBVP:lemma:prop:trace_on_dual;adjoint_trace}
Let $X$ be a Banach space, $i \in \{1,\ldots,l\}$, $\vec{a} \in (0,\infty)^{l}$, $\vec{p} \in [1,\infty)^{l}$, $q \in [1,\infty]$, $\gamma \in (-\mathpzc{d}_{i},\infty)$.
Let $\vec{w} \in \prod_{j=1}^{l}A_{\infty}(\R^{\mathpzc{d}_{j}})$ be such that
$w_{i}(x_{i}) = w_{\gamma}(x_{i}) = |x_{i}|^{\gamma}$.
For each $b \in \R^{\mathpzc{d}_{i}}$ consider the linear operator
\[
T_{[\mathpzc{d};i],b}: \mathcal{S}'(\R^{d-\mathpzc{d}_{i}};X) \longra \mathcal{S}'(\R^{d};X),\, f \mapsto \delta_{b} \otimes_{[\mathpzc{d};i]} f.
\]
\begin{itemize}
\item[(i)] If $s \in (-\infty,a_{i}\left[\frac{\mathpzc{d}_{i}+\gamma}{p_{i}}-\mathpzc{d}_{i}\right])$, then $T_{[\mathpzc{d};i],0}$ is bounded from $B^{s+a_{i}\left(\mathpzc{d}_{i}-\frac{\mathpzc{d}_{i}+\gamma}{p_{i}}\right),\vec{a}^{[i]}}_{\vec{p}^{[i]},q,\mathpzc{d}}
    (\R^{d-\mathpzc{d}_{i}},\vec{w}^{[i]};X)$ to $B^{s,\vec{a}}_{\vec{p},q,\mathpzc{d}}(\R^{d},\vec{w};X)$.
\item[(ii)] If $s \in (-\infty,a_{i}\left[\frac{\mathpzc{d}_{i}+\gamma_{-}}{p_{i}}-\mathpzc{d}_{i}\right])$, then $T_{[\mathpzc{d};i],b}$ is bounded from $B^{s+a_{i}\left(\mathpzc{d}_{i}-\frac{\mathpzc{d}_{i}+\gamma_{-}}{p_{i}}\right),\vec{a}^{[i]}}_{\vec{p}^{[i]},q}
    (\R^{d-\mathpzc{d}_{i}},\vec{w}^{[i]};X)$ to $B^{s,\vec{a}}_{\vec{p},q,\mathpzc{d}}(\R^{d},\vec{w};X)$ with norm estimate
    \[
    \norm{T_{[\mathpzc{d};i],b}}_{\mathcal{B}(B^{s+a_{i}\left(\mathpzc{d}_{i}-\frac{\mathpzc{d}_{i}+\gamma}{p_{i}}\right),
    \vec{a}^{[i]}}_{\vec{p}^{[i]},q}(\R^{d-\mathpzc{d}_{i}},\vec{w}^{[i]};X),B^{s}_{p,q}(\R^{d},w_{\gamma}))}
    \lesssim \max\{|b|,1\}^{\frac{\gamma_{+}}{p}}.
    \]
\end{itemize}
\end{lemma}
In order to perform all the estimates in Lemma~\ref{PIBVP:lemma:prop:trace_on_dual;adjoint_trace} we need the following two lemmas.

\begin{lemma}\label{PIBVP:lemma:prop:trace_on_dual;adjoint_trace;afschatting_translatie}
Let $\psi:\R^{d} \longra \C$ be a rapidly decreasing measurable function and put $\psi_{R}:=R^{d}\psi(R\,\cdot\,)$ for each $R>0$.
Let $p \in [1,\infty)$ and $\gamma \in (-1,\infty)$. For every $R>0$ and $a \in \R^{d}$ the following estimate holds true:
\[
\norm{\psi_{R}(\,\cdot\,-a)}_{L^{p}(\R^{d},|\,\cdot\,|^{\gamma})} \lesssim R^{d-\frac{d+\gamma}{p}}(|a|R+1)^{\gamma_{+}/p}
\]
\end{lemma}
\begin{proof}
By \cite[Condition~$B_{p}$]{Bui_Weighted_Beov&Triebel-Lizorkin_spaces:interpolation...} (see \cite[Lemma~4.5]{Meyries&Veraar_sharp_embd_power_weights} for a proof), if $w$ is an $A_{q}$-weight on $\R^{d}$ with $q \in (1,\infty)$, then
\begin{equation}\label{DBVP:eq:lemma:prop:trace_on_dual;adjoint_trace;afschatting_translatie}
\int_{\R^{d}}(1+|x-y|)^{-dq}\,dy \lesssim_{[w]_{A_{q}},q} \int_{B(x,1)}w(y)\,dy.
\end{equation}
So let us pick $q \in (1,\infty)$ so that $|\,\cdot\,|^{\gamma} \in A_{q}$.
Then, as $\psi$ is rapidly decreasing, there exists $C>0$ such that $|\psi(x)| \leq C (1+|x|)^{-q/p}$ for every $x \in \R^{d}$.
We can thus estimate
\begin{align*}
\norm{\psi_{R}(\,\cdot\,-a)}_{L^{p}(\R^{d},|\,\cdot\,|^{\gamma})}
&= R^{d-\frac{d+\gamma}{p}}\norm{\psi(\,\cdot\,-Ra)}_{L^{p}(\R^{d},|\,\cdot\,|^{\gamma} )} \\
&\leq C R^{d-\frac{d+\gamma}{p}} \norm{t \mapsto (1+|t-Ra|)^{-q/p}}_{L^{p}(\R^{d},|\,\cdot\,|^{\gamma} )} \\
& \stackrel{\eqref{DBVP:eq:lemma:prop:trace_on_dual;adjoint_trace;afschatting_translatie}}{\lesssim}
R^{d-\frac{d+\gamma}{p}}\left( \int_{B(|a|R,1)}|y|^{\gamma}\,dy \right)^{1/p} \\
&\lesssim R^{d-\frac{d+\gamma}{p}}(|a|R+1)^{\gamma_{+}/p}. \qedhere
\end{align*}
\end{proof}

\begin{lemma}\label{PIBVP:lemma:elementaire_afschatting_rijtjes}
For every $r \in [1,\infty]$ and $t>0$ there exists a constant $C > 0$ such that, for all sequences $(b_{k})_{k \in \N} \in \C^{\N}$, the following two inequalities hold true:
\[
\begin{array}{rl}
\norm{\left( 2^{tk}\sum_{n=k+1}^{\infty}|b_{n}| \right)_{k \in \N}}_{\ell^{r}(\N)} &\leq
C\norm{(2^{tk}b_{k})_{k \in \N}}_{\ell^{r}(\N)}, \\
\norm{\left( 2^{-tk}\sum_{n=0}^{k}|b_{n}| \right)_{k \in \N}}_{\ell^{r}(\N)} &\leq
C\norm{(2^{-tk}b_{k})_{k \in \N}}_{\ell^{r}(\N)}. \\
\end{array}
\]
\end{lemma}
\begin{proof}
See \cite[Lemma~4.2]{JS_traces} (and the references given there).
\end{proof}

\begin{proof}[Proof of Lemma~\ref{PIBVP:lemma:prop:trace_on_dual;adjoint_trace}]
Take $\varphi = (\varphi_{n})_{n \in \N} \in \Phi^{(\mathpzc{d},\vec{a})}(\R^{d})$ with $\varphi_{0} = \phi_{0} \otimes_{[\mathpzc{d};i]} \psi_{0}$, where $\phi = (\phi_{n})_{n \in \N} \in \Phi^{a_{i}}(\R^{\mathpzc{d}_{i}})$ and $\psi = (\psi_{n})_{n} \in \Phi^{(\mathpzc{d}^{[i]},\vec{a}^{[i]})}(\R^{d-\mathpzc{d}_{i}})$.
For $f \in \mathcal{S}'(\R^{d-\mathpzc{d}_{i}};X)$ we then have
\[
S^{\varphi}_{0}(\delta_{b} \otimes_{[\mathpzc{d};i]} f) = S^{\phi}_{0}\delta_{b} \otimes_{[\mathpzc{d};i]} S^{\psi}_{0}f = [\phi_{0}*\delta_{b}] \otimes_{[\mathpzc{d};i]}  [S^{\psi}_{0}f] = \phi_{0}(\,\cdot\,-b) \otimes_{[\mathpzc{d};i]}  S^{\psi}_{0}f
\]
and, for $n \geq 1$,
\begin{eqnarray*}
S^{\varphi}_{n}(\delta_{b} \otimes_{[\mathpzc{d};i]}  f)
&=& \varphi_{n}*(\delta_{b} \otimes_{[\mathpzc{d};i]}  f)
 = 2^{n \vec{a}\cdot\mathpzc{d}}\varphi_{0}(\delta^{[\mathpzc{d},\vec{a}]}_{2^{n}}\,\cdot\,)*(\delta_{b} \otimes_{[\mathpzc{d};i]}  f) - 2^{(n-1)\vec{a}\cdot\mathpzc{d}}\varphi_{0}(\delta^{[\mathpzc{d},\vec{a}]}_{2^{n-1}}\,\cdot\,)*(\delta_{b} \otimes_{_{[\mathpzc{d};i]} } f) \\
&=& \left[2^{na_{i}\mathpzc{d}_{i}}\phi_{0}(2^{na_{i}}\,\cdot\,)*\delta_{b} \right] \otimes_{[\mathpzc{d};i]}  \left[ 2^{n \vec{a}^{[i]}\cdot\mathpzc{d}^{[i]}}\psi_{0}(\delta^{[\mathpzc{d}^{[i]},\vec{a}^{[i]}]}_{2^{n}}\,\cdot\,)*f \right] \\
&& - \left[2^{(n-1)a_{i}\mathpzc{d}_{i}}\phi_{0}(2^{(n-1)a_{i}}\,\cdot\,)*\delta_{b} \right] \otimes_{[\mathpzc{d};i]}  \left[ 2^{(n-1) \vec{a}^{[i]}\cdot\mathpzc{d}^{[i]}}\psi_{0}(\delta^{[\mathpzc{d}^{[i]},\vec{a}^{[i]}]}_{2^{n-1}}\,\cdot\,)*f \right] \\
&=& 2^{na_{i}\mathpzc{d}_{i}}\phi_{0}(2^{na_{i}}[\,\cdot\,-b]) \otimes_{[\mathpzc{d};i]}  \left[ 2^{n \vec{a}_{i}\cdot\mathpzc{d}^{[i]}}\psi_{0}(\delta^{[\mathpzc{d}^{i},\vec{a}^{i}]}_{2^{n}}\,\cdot\,)*f \right] \\
&&  - 2^{(n-1)a_{i}\mathpzc{d}_{i}}\phi_{0}(2^{(n-1)a_{i}}[\,\cdot\,-b]) \otimes_{[\mathpzc{d};i]}  \left[ 2^{(n-1) \vec{a}^{[i]}\cdot\mathpzc{d}^{[i]}}\psi_{0}(\delta^{[\mathpzc{d}^{[i]},\vec{a}^{[i]}]}_{2^{n-1}}\,\cdot\,)*f \right] \\
&=& 2^{na_{i}\mathpzc{d}_{i}}\phi_{0}(2^{na_{i}}[\,\cdot\,-b]) \otimes_{[\mathpzc{d};i]}  \sum_{j=0}^{n}S^{\psi}f  - 2^{(n-1)a_{i}\mathpzc{d}_{i}}\phi_{0}(2^{(n-1)a_{i}}[\,\cdot\,-b]) \otimes \sum_{j=0}^{n-1}S^{\psi}_{j}f.
\end{eqnarray*}
Applying Lemma~\ref{PIBVP:lemma:prop:trace_on_dual;adjoint_trace;afschatting_translatie} we obtain the estimate
\begin{align}\label{DBVP:eq:lemma:prop:trace_on_dual;adjoint_trace;estimate}
\norm{S^{\varphi}_{n}(\delta_{b} \otimes_{[\mathpzc{d};i]} f)}_{L^{\vec{p},\mathpzc{d}}(\R^{d},\vec{w};X)}
& \lesssim 2^{na_{i}\left(\mathpzc{d}_{i}-\frac{\mathpzc{d}_{i}+\gamma}{p}\right)}(|b|2^{na_{i}}+1)^{\frac{\gamma_{+}}{p_{i}}} \\
& \: \cdot\:
\left\{\begin{array}{ll}
\norm{S^{\psi}_{0}f}_{L^{\vec{p}^{[i]},\mathpzc{d}^{[i]}}(\R^{d-1},\vec{w}^{[i]};X)}, & n=0; \\
\norm{\sum_{j=0}^{n}S^{\psi}_{j}f}_{L^{\vec{p}^{[i]},\mathpzc{d}^{[i]}}(\R^{d-1},\vec{w}^{[i]};X)} + \norm{\sum_{j=0}^{n-1}S^{\psi}_{j}f}_{L^{\vec{p}^{[i]},\mathpzc{d}^{[i]}}(\R^{d-1},\vec{w}^{[i]};X)}, & n \geq 1.
\end{array}\right. \nonumber
\end{align}

(i) Using \eqref{DBVP:eq:lemma:prop:trace_on_dual;adjoint_trace;estimate}, we can estimate
\begin{eqnarray*}
\norm{\delta \otimes_{[\mathpzc{d};i]} f}_{B^{s,\vec{a}}_{\vec{p},q,\mathpzc{d}}(\R^{d},\vec{w};X)}
&=& \normb{ \left( 2^{sn}\norm{S^{\varphi}_{n}(\delta \otimes_{[\mathpzc{d};i]} f)}_{L^{\vec{p},\mathpzc{d}}(\R^{d},\vec{w};X)} \right)_{n \in \N} }_{\ell^{q}} \\
&\lesssim& \normB{ \Big( 2^{\left(s+a_{i}\left[\mathpzc{d}_{i}- \frac{\mathpzc{d}_{i}+\gamma}{p_{i}}\right] \right)n}\normb{\sum_{j=0}^{n}S^{\psi}_{j}f}_{L^{\vec{p}^{[i]},\mathpzc{d}^{[i]}}(\R^{d-1},\vec{w}^{[i]};X)} \Big)_{n \geq 0} }_{\ell^{q}}.
\end{eqnarray*}
As $s+a_{i}\left(\mathpzc{d}_{i}- \frac{\mathpzc{d}_{i}+\gamma}{p_{i}}\right) < 0$, we obtain the desired estimate by an application of the triangle inequality in $L^{\vec{p}^{[i]},\mathpzc{d}^{[i]}}(\R^{d-1},\vec{w}^{[i]};X)$ followed by Lemma~\ref{PIBVP:lemma:prop:trace_on_dual;adjoint_trace;afschatting_translatie}.

(ii) Observing that
\[
2^{na_{i}\left(\mathpzc{d}_{i}-\frac{\mathpzc{d}_{i}+\gamma}{p}\right)}(|b|2^{na_{i}}+1)^{\frac{\gamma_{+}}{p}} \lesssim
2^{na_{i}\left(\mathpzc{d}_{i}-\frac{\mathpzc{d}_{i}+\gamma_{-}}{p}\right)}\max\{|b|,1\}^{\frac{\gamma_{+}}{p}},
\]
the desired estimate can be derived in the same way as in (i). \qedhere

\end{proof}

\begin{proof}[Proof of Proposition~\ref{PIBVP:thm:trace_Besov}]
Let us first establish \eqref{PIBVP:eq:thm:trace_Besov:trace_estimate} and \eqref{PIBVP:eq:thm:trace_Besov:distr_trace}.
Thanks to the Sobolev embedding Proposition~\ref{PIBVP:Sobolev_embedding_Besov} we may restrict ourselves to the case $\gamma \in (-1,p-1)$, so that $\vec{w} \in \prod_{j=1}^{l}A_{p_{j}}(\R^{\mathpzc{d}_{j}})$.
As $\mathcal{S}(\R^{d};X) \stackrel{d}{\hookrightarrow} B_{\vec{p},q,\mathpzc{d}}^{s,\vec{a}}(\R^{d},\vec{w};X)$ and $\mathcal{S}(\R^{d-\mathpzc{d}_{i}};X) \stackrel{d}{\hookrightarrow} B_{\vec{p}^{[i]},q,\mathpzc{d}^{[i]}}^{t,\vec{a}^{[i]}}(\R^{d-\mathpzc{d}_{i}},
\vec{w}^{[i]};X)$ ($s,t \in \R$), we have
\[
[B_{\vec{p},q,\mathpzc{d}}^{s,\vec{a}}(\R^{d},\vec{w};X)]^{*} \hookrightarrow \mathcal{S}'(\R^{d};X^{*})
\quad \mbox{and} \quad
[B_{\vec{p}^{[i]},q,\mathpzc{d}^{[i]}}^{t,\vec{a}^{[i]}}(\R^{d-\mathpzc{d}_{i}},
\vec{w}^{[i]};X)]^{*} \hookrightarrow \mathcal{S}'(\R^{d-\mathpzc{d}_{i}};X^{*})
\]
under the natural identifications; also see the discussion preceding Lemma~\ref{PIBVP:lemma:prop:trace_on_dual;adjoint_trace}.
In this way we explicitly have
\[
[B_{\vec{p},q,\mathpzc{d}}^{s,\vec{a}}(\R^{d},\vec{w};X)]^{*} = B_{\vec{p}',q',\mathpzc{d}}^{-s,\vec{a}}(\R^{d},\vec{w}';X^{*})
\]
and
\[
[B_{\vec{p}^{[i]},q,\mathpzc{d}^{[i]}}^{t,\vec{a}^{[i]}}(\R^{d-\mathpzc{d}_{i}},
\vec{w}^{[i]};X)]^{*} =
B_{\vec{p}'^{[i]},q',\mathpzc{d}^{[i]}}^{-t,\vec{a}^{[i]}}(\R^{d-\mathpzc{d}_{i}},
\vec{w}'^{[i]};X)
\]
by \cite{Lindemulder_Intersection} as $\vec{w} \in \prod_{j=1}^{l}A_{p_{j}}(\R^{\mathpzc{d}_{j}})$, where $\vec{p}'=(p_{1}',\ldots,p_{l}')$ and $\vec{w} = (w_{1}^{-\frac{1}{p_{1}-1}},\ldots,w_{l}^{-\frac{1}{p_{l}-1}})$.
Note here that $\vec{w}'_{i}(x_{i}) = |x_{i}|^{\gamma'}$ with $\gamma'= -\frac{\gamma}{p_{i}-1}$.
Since $-[s-\frac{a_{i}}{p_{i}}(\mathpzc{d}_{i}+\gamma)] = -s + a_{i}\left( \mathpzc{d}_{i}-\frac{\mathpzc{d}_{i}+\gamma'}{p_{i}'} \right)$ and $-[s-\frac{a_{i}}{p_{i}}(\mathpzc{d}_{i}+\gamma_{+})] = -s + a_{i}\left( \mathpzc{d}_{i}-\frac{\mathpzc{d}_{i}+(\gamma')_{-}}{p_{i}'} \right)$,
it follows from Lemma~\ref{PIBVP:lemma:prop:trace_on_dual;adjoint_trace} and the discussion preceding that
\[
\norm{\mathrm{tr}_{[\mathpzc{d};i]}\,f}
_{B_{\vec{p}^{[i]},q,\mathpzc{d}^{[i]}}^{s-\frac{a_{i}}{p_{i}}(\mathpzc{d}_{i}+\gamma),\vec{a}^{[i]}}(\R^{d-\mathpzc{d}_{i}},
\vec{w}^{[i]};X)} \lesssim \norm{f}_{B_{\vec{p},q,\mathpzc{d}}^{s,\vec{a}}(\R^{d},\vec{w};X)}, \quad\quad f \in \mathcal{S}(\R^{d};X),
\]
and, if $s>\frac{a_{i}}{p_{i}}(\mathpzc{d}_{i}+\gamma_{+})$,
\[
\norm{\mathrm{tr}_{[\mathpzc{d};i],b}\,f}_{
B_{\vec{p}^{[i]},q,\mathpzc{d}^{[i]}}^{s-\frac{a_{i}}{p_{i}}(\mathpzc{d}_{i}+\gamma_{+}),\vec{a}^{[i]}}
(\R^{d-\mathpzc{d}_{i}},\vec{w}^{[i]};X)} \lesssim \rho_{p_{i},\gamma}(b) \norm{f}_{B_{\vec{p},q,\mathpzc{d}}^{s,\vec{a}}(\R^{d},\vec{w};X)}, \quad\quad f \in \mathcal{S}(\R^{d};X),
b \in \R^{\mathpzc{d}_{i}}.
\]
These two inequalities imply \eqref{PIBVP:eq:thm:trace_Besov:trace_estimate} and \eqref{PIBVP:eq:thm:trace_Besov:distr_trace}, respectively.

Let us finally show that the extension operator $\mathrm{ext}$ from Lemma \ref{PIBVP:prop:right_inverse_distr_trace} (with $\tilde{\mathpzc{d}} = \mathpzc{d}^{[i]}$ and $\tilde{\vec{a}}= \vec{a}^{[i]}$, modified in the obvious way to the $i$-th multidimensional coordinate) restricts to a coretraction for $\mathrm{tr}_{[\mathpzc{d};i]}$.
To this end we fix $g \in B_{\vec{p}^{[i]},q,\mathpzc{d}^{[i]}}^{s-\frac{a_{i}}{p_{i}}(\mathpzc{d}_{i}+\gamma),\vec{a}^{[i]}}(\R^{d-\mathpzc{d}_{i}},
\vec{w}^{[i]};X)$. In view of (the modified version of)  \eqref{functieruimten:eq:prop;right_inverse_distr_trace_supports} and Lemma~\ref{PIBVP:lemma:appendix:conv_series_dyadic_corona;Besov}, it suffices to estimate
\[
\norm{(2^{ns}\rho(2^{na_{i}}\,\cdot\,) \otimes_{[\mathpzc{d};i]} [\psi_{n}*g])_{n \in\N}}_{\ell_{q}(\N;L^{\vec{p},\mathpzc{d}}(\R^{d},\vec{w};X))} \lesssim
\norm{g}_{B_{\vec{p}^{[i]},q,\mathpzc{d}^{[i]}}^{s-\frac{a_{i}}{p_{i}}(\mathpzc{d}_{i}+\gamma),\vec{a}^{[i]}}(\R^{d-\mathpzc{d}_{i}},
\vec{w}^{[i]};X)}.
\]
A simple computation even shows that
\[
\norm{(2^{ns}\rho(2^{na_{1}}\,\cdot\,) \otimes_{[\mathpzc{d};i]} [\psi_{n}*g])_{n \in\N }}_{\ell_{q}(\N;L^{\vec{p},\mathpzc{d}}(\R^{d},\vec{w};X))}
= \norm{\rho}_{L^{p_{i}}(\R^{\mathpzc{d}_{i}},|\,\cdot\,|^{\gamma})}
\norm{g}_{B_{\vec{p}^{[i]},q,\mathpzc{d}^{[i]}}^{s-\frac{a_{i}}{p_{i}}(\mathpzc{d}_{i}+\gamma),\vec{a}^{[i]}}(\R^{d-\mathpzc{d}_{i}},
\vec{w}^{[i]};X)}.
\]
\end{proof}

\subsubsection{The proof of Theorem~\ref{PIBVP:thm:trace_TL}}

For the proof of Theorem~\ref{PIBVP:thm:trace_TL} we need three lemmas. Two lemmas concern estimates in Triebel-Lizorkin spaces for series satisfying certain Fourier support conditions, which can be found in Appendix~\ref{PIBVP:appendix:series_estimates}.
The other lemma is Lemma~\ref{PIBVP:lemma:elementaire_afschatting_rijtjes}.

\begin{proof}[Proof of Theorem~\ref{PIBVP:thm:trace_TL}.]
Let the notations be as in Proposition \ref{PIBVP:prop:right_inverse_distr_trace}.
We will show that, for an arbitrary $\varphi \in \Phi^{\mathpzc{d},a}(\R^{d})$,
\begin{itemize}
\item[(I)] $\tau^{\varphi}$ exists on $F_{\vec{p},q,\mathpzc{d}}^{s,\vec{a}}(\R^{d},(w_{\gamma},\vec{w}'');X)$ and defines a continuous operator
\[
\tau^{\varphi}:F_{\vec{p},q,\mathpzc{d}}^{s,\vec{a}}(\R^{d},(w_{\gamma},\vec{w}'');X) \longra F_{\vec{p}'',p_{1},\mathpzc{d}''}^{s-\frac{a_{1}}{p_{1}}(1+\gamma),\vec{a}''}(\R^{d-1},\vec{w}'';X);
\]
\item[(II)] The extension operator $\mathrm{ext}$ from Proposition \ref{PIBVP:prop:right_inverse_distr_trace} (with $\tilde{\mathpzc{d}} = \mathpzc{d}''$ and $\tilde{\vec{a}}=\vec{a}''$) restricts to a continuous operator
\[
\mathrm{ext}:F_{\vec{p}'',p_{1},\mathpzc{d}''}^{s-\frac{a_{1}}{p_{1}}(1+\gamma),\vec{a}''}(\R^{d-1},\vec{w}'';X) \longra
F_{\vec{p},q,\mathpzc{d}}^{s,\vec{a}}(\R^{d},(w_{\gamma},\vec{w}'');X).
\]
\end{itemize}
Since $\mathscr{F}^{-1}C^{\infty}_{c}(\R^{d};X) \subset \mathscr{F}^{-1}\mathcal{E}'(\R^{d-1};X) \cap F_{\vec{p}'',p_{1},\mathpzc{d}''}^{s,\vec{a}''}(\R^{d-1},\vec{w}'';X)$ is a dense subspace of $F_{\vec{p}'',p_{1},\mathpzc{d}''}^{s,\vec{a}''}(\R^{d},\vec{w}'';X)$, the right inverse part in the first assertion follows from (I) and (II). The independence of $\varphi$ in the first assertion follows from denseness of $\mathcal{S}(\R^{d};X)$ in $F_{\vec{p},q,\mathpzc{d}}^{s,\vec{a}}(\R^{d},(w_{\gamma},\vec{w}'');X)$ in case $q<\infty$, from which the case $q=\infty$ can be deduced via a combination of \eqref{PIBVP:prelim:eq:elem_embd_epsilon} and \eqref{PIBVP:prelim:eq:elem_embd_BF_rel}.

(I): We may with out loss of generality assume that $q=\infty$.
Let $f \in F_{\vec{p},\infty,\mathpzc{d}}^{s,\vec{a}}(\R^{d},(w_{\gamma},\vec{w}'');X)$ and write $f_{n} := S_{n}f$ for each $n$.
Then each $f_{n} \in \mathcal{S}'(\R^{d};X)$ has Fourier support
\[
\supp \hat{f}_{n} \subset \prod_{j=1}^{l}[-c2^{na_{j}},c2^{na_{j}}]^{\mathpzc{d}_{j}}
\]
for some constant $c>0$ only depending on $\varphi$.
Therefore, as a consequence of the Paley-Wiener-Schwartz theorem, we have $f_{n}(0,\cdot) \in \mathcal{S}'(\R^{d-1};X)$ with Fourier support contained in $\prod_{j=2}^{l}[-c2^{na_{j}},c2^{na_{j}}]^{\mathpzc{d}_{j}}$.
In view of Lemma-\ref{PIBVP:lemma:appendix:prop;trace_TL_conv_series},
it suffices to show that
\begin{equation}\label{functieruimten:eq:thm;trace_TL_I_suffices}
\norm{\left(2^{n[s-\frac{a_{1}}{p_{1}}(1+\gamma)]}f_{n}(0,\cdot)\right)_{n \geq 0}}_{L^{\vec{p}'',\mathpzc{d}''}(\R^{d-1},\vec{w}'';\ell^{p_{1}}(\N;X))}
\lesssim \norm{f}_{F_{\vec{p},\infty,\mathpzc{d}}^{s,\vec{a}}(\R^{d},(w_{\gamma},\vec{w}'');X)}.
\end{equation}

In order to establish the estimate \eqref{functieruimten:eq:thm;trace_TL_I_suffices}, we
pick an $r_{1} \in (0,1)$ such that $w_{\gamma} \in A_{p_{1}/r_{1}}(\R)$, and write $\vec{r}:=(r_{1},\vec{r}'') \in (0,1)^{l}$.
For all $x=(x_{1},x'') \in [2^{-na_{1}},2^{(1-n)a_{1}}] \times \R^{d-1}$ and every $n \in \N$ we have
\[
\norm{f_{n}(0,x'')} \leq C_{1}\frac{\norm{f_{n}(x_{1}-y_{1},x'')}}{1+|2^{na_{1}}y_{1}|^{1/r_{1}}}\Big|_{y_{1}=x_{1}}
\leq (1+2^{\frac{a_{1}}{r_{1}}})f_{n}^{*}(\vec{r},\vec{b}^{[n]},\mathpzc{d};x) = C_{1}f_{n}^{*}(\vec{r},\vec{b}^{[n]},\mathpzc{d};x),
\]
where $\vec{b}^{[n]}:= (2^{na_{1}},\ldots,2^{na_{l}}) \in (0,\infty)^{l}$ and where $f_{n}^{*}(\vec{r},\vec{b}^{[n]},\mathpzc{d};\,\cdot\,)$ is the maximal function of Peetre-Fefferman-Stein type given in \eqref{PIBVP:eq:appendix:max_funct_Peetre-Fefferman-Stein_type}.
Raising this to the $p_{1}$-th power, multiplying by $2^{nsp_{1}}|x_{1}|^{\gamma}$, and integrating over $x_{1} \in [2^{-na_{1}},2^{(1-n)a_{1}}]$, we obtain
\[
\frac{2^{a_{1}(\gamma+1)}-1}{1+\gamma}2^{n\left(s-\frac{a_{1}}{p_{1}}(1+\gamma)\right)p_{1}}\norm{f_{n}(0,x'')}^{p_{1}}  \leq C_{1}^{p}\int_{[2^{-na_{1}},2^{(1-n)a_{1}}]}
\left[2^{ns}f_{n}^{*}(\vec{r},\vec{b}^{[n]},\mathpzc{d};(x_{1},x''))\right]^{p_{1}}|x_{1}|^{\gamma}dx_{1}.
\]
It now follows that
\[
\sum_{n=0}^{\infty}2^{\left(s-\frac{a_{1}}{p_{1}}(1+\gamma)\right)np_{1}}\norm{f_{n}(0,x'')}^{p_{1}}  \leq C_{2}\int_{\R}\norm{\left(2^{ks}f_{k}^{*}(\vec{r},\vec{b}^{[n]},\mathpzc{d};(x_{1},x''))\right)_{k \geq 0}}_{\ell^{\infty}(\N)}^{p_{1}}|x_{1}|^{\gamma}dx_{1},
\]
from which we in turn obtain
\[
\norm{\left(2^{n[s-\frac{a_{1}}{p_{1}}(1+\gamma)]}f_{n}(0,\cdot)\right)_{n \geq 0}}_{L^{\vec{p}'',\mathpzc{d}''}(\R^{d-1},\vec{w}'';\ell^{p_{1}}(\N;X))}
\leq \norm{\left(2^{ks}f_{k}^{*}(\vec{r},\vec{b}^{[n]},\mathpzc{d};\,\cdot\,)\right)_{k \geq 0}}_{L^{\vec{p},\mathpzc{d}}(\R^{d},(w_{\gamma},\vec{w}'');\ell^{\infty}(\N))}.
\]
Since $(f_{k})_{k \in \N} \subset \mathcal{S}'(\R^{d};X)$ satisfies $\supp(\hat{f}_{k}) \subset \prod_{j=1}^{l}[-b^{[k]}_{j},b^{[k]}_{j}]^{\mathpzc{d}_{j}}$ for each $k \in \N$ and some $c>0$, the desired estimate \eqref{functieruimten:eq:thm;trace_TL_I_suffices} is now a consequence of Proposition~
\ref{PBIBV:lemma:appendix:Peetre-Fefferman-Stein_maximal_ineq}.

(II): We may with out loss of generality assume that $q=1$.
Let $g \in F_{\vec{p}'',p_{1},\mathpzc{d}''}^{s-\frac{a_{1}}{p_{1}}(1+\gamma),\vec{a}''}(\R^{d-1},\vec{w}'';X)$ and write $g_{n}=T_{n}g$ for each $n$.
By construction of $\mathrm{ext}$ we have  $\mathrm{ext}\,g = \sum_{n=0}^{\infty}\rho(2^{na_{1}}\,\cdot\,) \otimes g_{n}$ in $\mathcal{S}'(\R^{d};X)$ with each $\rho(2^{na_{1}}\,\cdot\,) \otimes g_{n}$ satisfying (\ref{functieruimten:eq:prop;right_inverse_distr_trace_supports}) for a $c > 1$ independent of $g$.
In view of Lemma~\ref{PIBVP:lemma:appendix:conv_series_dyadic_corona;TL}, it is thus enough to show that
\begin{equation}\label{functieruimten:eq:thm;trace_TL_II_suffices}
\norm{(2^{sn}\rho(2^{na_{1}}\,\cdot\,) \otimes g_{n} )_{n \geq 0}}_{L^{p,\mathpzc{d}}(\R^{d},(w_{\gamma},w'');\ell^{1}(X))}
\lesssim \norm{g}_{F_{p'',p_{1},\mathpzc{d}''}^{s-\frac{a_{1}}{p_{1}}(1+\gamma),a''}(\R^{d-1},w'';X)}.
\end{equation}

In order to establish the estimate \eqref{functieruimten:eq:thm;trace_TL_II_suffices}, we define, for each $x'' \in \R^{d-1}$,
\begin{equation}\label{functieruimten:eq:bewijs_extensie_op_TL}
I(x'') := \int_{\R}\left(\sum_{n=0}^{\infty}2^{sn}\norm{\rho(2^{na_{1}}x_{1})g_{n}(x'')} \right)^{p_{1}}|x_{1}|^{\gamma}dx_{1}.
\end{equation}
We furthermore first choose a natural number $N > \frac{1}{p_{1}}(1+\gamma)$ and subsequently pick a constant $C_{1} > 0$ for which the Schwartz function $\rho \in \mathcal{S}(\R)$ satisfies the inequality
$|\rho(2^{na_{1}}x_{1})| \leq C_{1}|2^{na_{1}}x_{1}|^{-N}$ for every $n \in \N$ and all $x_{1} \neq 0$.

Denoting by $I_{1}(x'')$ the integral over $\R \setminus [-1,1]$ in (\ref{functieruimten:eq:bewijs_extensie_op_TL}), we have
\begin{align}
I_{1}(x'')
&\leq
C_{1}\int_{\R \setminus [-1,1]}\left(\sum_{n=0}^{\infty}2^{-Na_{1}n} \,2^{sn}\norm{g_{n}(x'')}\right)^{p_{1}}|x_{1}|^{-Np_{1}+\gamma}dx_{1} \nonumber \\
&= C_{1}\int_{\R \setminus [-1,1]}|x_{1}|^{-Np_{1}+\gamma}dx_{1}\left(\sum_{n=0}^{\infty} 2^{\left(\frac{1}{p_{1}}(1+\gamma)-N\right)a_{1}n}\, 2^{\left(s-\frac{a_{1}}{p_{1}}(1+\gamma)\right)n}\norm{g_{n}(x'')} \right)^{p_{1}} \nonumber \\
&\leq \underbrace{\int_{\R \setminus [-1,1]}|x_{1}|^{-Np_{1}+\gamma}dx_{1}\norm{\left(\,2^{\left(\frac{1}{p_{1}}(1+\gamma)-N\right)a_{1}n}\,\right)_{n \geq 0}}_{\ell^{p_{1}'}(\N)}^{p_{1}}}_{=:C_{2} \in [0,\infty)}\norm{\left(\,2^{\left(s-\frac{a_{1}}{p_{1}}(1+\gamma)\right)n}\norm{g_{n}(x'')}\,\right)_{n \geq 0}}_{\ell^{p_{1}}(\N)}^{p_{1}}. \label{functieruimten:eq:bewijs_extensie_op_TL;1}
\end{align}

Next we denote, for each $k \in \N$, by $I_{0,k}(x'')$ the integral over $D_{k}:=\{ x_{1} \in \R \mid 2^{-(k+1)a_{1}} \leq |x_{1}| \leq 2^{-ka_{1}}\}$ in (\ref{functieruimten:eq:bewijs_extensie_op_TL}). Since the $D_{k}$ are of measure
$w_{\gamma}(D_{k}) \leq C_{3} 2^{-ka_{1}(\gamma+1)}$ for some constant $C_{3}>0$ independent of $k$, we can estimate
\begin{eqnarray*}
I_{0,k}(x'')
&\leq& \int_{D_{k}}\left(\sum_{n=0}^{k}2^{sn}\norm{\rho}_{\infty}\norm{g_{n}(x'')} + \sum_{n=k+1}^{\infty}C_{1}2^{(s-a_{1}N)n}|x_{1}|^{-N}\norm{g_{n}(x'')} \right)^{p_{1}}|x_{1}|^{\gamma}dx_{1} \\
&\leq& C_{3}2^{-ka_{1}(\gamma+1)}\left(\sum_{n=0}^{k}2^{sn}\norm{\rho}_{\infty}\norm{g_{n}(x'')} + \sum_{n=k+1}^{\infty}C_{1}2^{(s-a_{1}N)n}2^{Na_{1}(k+1)}\norm{g_{n}(x'')} \right)^{p_{1}} \\
&\leq& C_{3}2^{p_{1}}\norm{\rho}_{\infty}^{p_{1}}2^{-ka_{1}(\gamma+1)}\left(\sum_{n=0}^{k}2^{sn}\norm{g_{n}(x'')} \right)^{p_{1}}
 \\
& & \:\:+\:\: C_{3}2^{p_{1}}(C_{1}2^{Na_{1}})^{p_{1}} 2^{k\left(N-\frac{1}{p_{1}}(\gamma+1)\right)a_{1}p_{1}}\left(\sum_{n=k+1}^{\infty}2^{(s-a_{1}N)n}\norm{g_{n}(x'')} \right)^{p_{1}}.
\end{eqnarray*}
Writing $I_{0}(x''):= \sum_{k=0}^{\infty}I_{0,k}(x'')$, which is precisely the integral over $[-1,1]$ in (\ref{functieruimten:eq:bewijs_extensie_op_TL}), we obtain
\begin{eqnarray*}
I_{0}(x'')
&\leq& C_{4} \sum_{k=0}^{\infty}2^{-ka_{1}(\gamma+1)}\left(\sum_{n=0}^{k}2^{sn}\norm{g_{n}(x'')} \right)^{p_{1}} + C_{4}\sum_{k=0}^{\infty} 2^{k\left(N-\frac{1}{p_{1}}(\gamma+1)\right)a_{1}p_{1}}\left(\sum_{n=k+1}^{\infty}2^{(s-a_{1}N)n}\norm{g_{n}(x'')} \right)^{p_{1}}  \\
&=& C_{4}\norm{ \left( 2^{-\frac{a_{1}}{p_{1}}(1+\gamma)k}\sum_{n=0}^{k}2^{sn}\norm{g_{n}(x'')}\right)_{k \in \N} }_{\ell^{p_{1}}(\N)}^{p_{1}} \\
&& \quad\quad + \quad C_{4}\norm{ \left( 2^{\left(N-\frac{1}{p_{1}}(1+\gamma)\right)a_{1}k}\sum_{n=k+1}^{\infty}2^{(s-a_{1}N)n}\norm{g_{n}(x'')}\right)_{k \in \N} }_{\ell^{p_{1}}(\N)}^{p_{1}},
\end{eqnarray*}
which via an application of Lemma~\ref{PIBVP:lemma:elementaire_afschatting_rijtjes} can be further estimated as
\begin{align}
I_{0}(x'')
&\leq C_{5}\norm{\left(\,2^{-\frac{a_{1}}{p_{1}}(1+\gamma)k}2^{sk}\norm{g_{k}(x'')}\,\right)_{k \geq 0}}_{\ell_{p_{1}}(\N)}^{p_{1}} +
C_{5}\norm{\left(\,2^{\left(N-\frac{1}{p_{1}}(\gamma+1)\right)a_{1}k}2^{(s-a_{1}N)k}\norm{g_{k}(x'')}\,\right)_{k \geq 0}}_{\ell^{p_{1}}(\N)}^{p_{1}} \nonumber \\
&= 2C_{5}\norm{\left(\,2^{\left(s-\frac{a_{1}}{p_{1}}(1+\gamma)\right)k}\norm{g_{k}(x'')}\,\right)_{k \geq 0}}_{\ell^{p_{1}}(\N)}^{p_{1}}. \label{functieruimten:eq:bewijs_extensie_op_TL;0}
\end{align}

Combining the estimates \eqref{functieruimten:eq:bewijs_extensie_op_TL;1} and \eqref{functieruimten:eq:bewijs_extensie_op_TL;0}, we get
\[
I(x'')^{1/p_{1}} \leq C_{6}\norm{\left(\,2^{\left(s-\frac{a_{1}}{p_{1}}(1+\gamma)\right)n}\norm{g_{n}(x'')}\,\right)_{n \geq 0}}_{\ell^{p_{1}}(\N)},
\]
from which \eqref{functieruimten:eq:thm;trace_TL_II_suffices} follows by taking $L^{\vec{p}'',\mathpzc{d}''}(\R^{d-1},\vec{w}'')$-norms.
\end{proof}

\section{Sobolev embedding for Besov spaces}\label{PIBVP:sec:Sobolev_embedding_Besov}

The result below is a direct extension of part of \cite[Proposition~1.1]{Meyries&Veraar_sharp_embd_power_weights}.
We refer to \cite{JS_Sob_embeddings} for embedding results for unweighted anisotropic mixed-norm Besov space and we refer to \cite{Haroske&Skrzypczak2011_Entropy_and_aprrox} for embedding results of weighted Besov spaces.

\begin{prop}\label{PIBVP:Sobolev_embedding_Besov}
Let $X$ be a Banach space, $\vec{p},\tilde{\vec{p}} \in (1,\infty)^{l}$, $q,\tilde{q} \in [1,\infty]$, $s,\tilde{s} \in \R$, $\vec{a} \in (0,\infty)^{l}$, and $\vec{w},\tilde{\vec{w}} \in \prod_{j=1}^{l}A_{\infty}(\R^{\mathpzc{d}_{j}})$.
Suppose that $J \subset \{1,\ldots,l\}$ is such that
\begin{itemize}
\item $p_{j} = \tilde{p}_{j}$ and $w_{j} = \tilde{w}_{j}$ for $j \notin J$;
\item $w_{j}(x_{j}) = |x_{j}|^{\gamma_{j}}$ and $\tilde{w}_{j}(x_{j}) = |x_{j}|^{\tilde{\gamma}_{j}}$ for $j \in J$ for some $\gamma_{j},\tilde{\gamma}_{j} > -\mathpzc{d}_{j}$ satisfying
\[
\frac{\tilde{\gamma}_{j}}{\tilde{p}_{j}} \leq \frac{\gamma_{j}}{p_{j}} \quad \mbox{and} \quad
\frac{\mathpzc{d}_{j}+\tilde{\gamma}_{j}}{\tilde{p}_{j}} < \frac{\mathpzc{d}_{j}+\gamma_{j}}{p_{j}}.
 \]
\end{itemize}
Furthermore, assume that $q \leq \tilde{q}$ and that $s- \sum_{i \in I}a_{i}\frac{\mathpzc{d}_{i}+\gamma_{i}}{p_{i}} > \tilde{s} - a_{i}\sum_{i \in I}\frac{\mathpzc{d}_{i}+\tilde{\gamma}_{i}}{\tilde{p}_{i}}$.
Then
\[
B^{s,\vec{a}}_{\vec{p},q,\mathpzc{d}}(\R^{d},\vec{w};X) \hookrightarrow B^{\tilde{s},\vec{a}}_{\tilde{\vec{p}},\tilde{q},\mathpzc{d}}(\R^{d},\tilde{\vec{w}};X).
\]
\end{prop}
\begin{proof}
This is an immediate consequence of inequality of Plancherel-P\'olya-Nikol'skii type given in Lemma~\ref{PIBVP:PPN-type_ineq}.
\end{proof}

\begin{lemma}\label{PIBVP:PPN-type_ineq}
Let $X$ be a Banach space, $\vec{p},\tilde{\vec{p}} \in (1,\infty)^{l}$, and $\vec{w},\tilde{\vec{w}} \in \prod_{j=1}^{l}\mathcal{W}(\R^{\mathpzc{d}_{j}})$.
Suppose that $J \subset \{1,\ldots,l\}$ is such that
\begin{itemize}
\item $p_{j} = \tilde{p}_{j}$ and $w_{j} = \tilde{w}_{j}$ for $j \notin J$;
\item $w_{j}(x_{j}) = |x_{j}|^{\gamma_{j}}$ and $\tilde{w}_{j}(x_{j}) = |x_{j}|^{\tilde{\gamma}_{j}}$ for $j \in J$ for some $\gamma_{j},\tilde{\gamma}_{j} > -\mathpzc{d}_{j}$ satisfying
\[
\frac{\tilde{\gamma}_{j}}{\tilde{p}_{j}} \leq \frac{\gamma_{j}}{p_{j}} \quad \mbox{and} \quad
\frac{\mathpzc{d}_{j}+\tilde{\gamma}_{j}}{\tilde{p}_{j}} < \frac{\mathpzc{d}_{j}+\gamma_{j}}{p_{j}}.
 \]
\end{itemize}
Then there exists a constant $C>0$ such that, for all $f \in \mathcal{S}'(\R^{d};X)$ with
$\supp(\hat{f}) \subset \prod_{j=1}^{l}[-R_{1},R_{1}]^{\mathpzc{d}_{j}}$ for some $R_{1},\ldots,R_{l} > 0$, we have the inequality
\[
\norm{f}_{L^{\tilde{\vec{p}},\mathpzc{d}}(\R^{d},\tilde{\vec{w}};X)} \leq
C\left(\prod_{j \in J}R_{j}^{\delta_{j}}\right) \norm{f}_{L^{\vec{p},\mathpzc{d}}(\R^{d},\vec{w};X)},
\]
where $\delta_{j}:= (\mathpzc{d}_{j}+\gamma_{j})/p_{j} - (\mathpzc{d}_{j}+\tilde{\gamma}_{j})/\tilde{p}_{j} > 0$ for each $j \in J$.
\end{lemma}
\begin{proof}
\textbf{Step I.} \emph{The case $l=1$:}

We refer to \cite[Proposition~4.1]{Meyries&Veraar_sharp_embd_power_weights}.

\textbf{Step II.} \emph{The case $J=\{l\}$:}

Under the canonical isomorphism $\mathcal{D}'(\R^{d};X) \cong \mathcal{D}'(\R^{\mathpzc{d}_{l}};\mathcal{D}'(\R^{\mathpzc{d}_{1}+\ldots+\mathpzc{d}_{l-1}};X))$ (Schwartz kernel theorem),
$f$ corresponds to an element of $\mathcal{S}'(\R^{\mathpzc{d}_{l}};C(\R^{\mathpzc{d}_{1}+\ldots+\mathpzc{d}_{l-1}};X))$ having compact Fourier support contained in $[-R_{l},R_{l}]^{\mathpzc{d}_{l}}$. Given a compact subset $K \subset \R^{\mathpzc{d}_{1}+\ldots+\mathpzc{d}_{l-1}}$ we have
the continuous linear operator
\[
m_{1_{K}}:C(\R^{d'};X) \longra L^{\infty}_{K}(\R^{d'};X)  \hookrightarrow L^{\vec{p}',\mathpzc{d}'}(\R^{d'},\vec{w}';X),\, g \mapsto 1_{K}g,
\]
where $d':= \mathpzc{d}_{1}+\ldots+\mathpzc{d}_{l-1}$, $\mathpzc{d}' = (\mathpzc{d}_{1},\ldots,\mathpzc{d}_{l-1})$, $\vec{p}':=(p_{1},\ldots,p_{l-1})$, and $\vec{w}'=(w_{1},\ldots,w_{l-1})$.
Accordingly, for each compact $K \subset \R^{d'}$ we have $1_{K}f = m_{1_{K}}f \in \mathcal{S}'(\R^{\mathpzc{d}_{l}};L^{\vec{p}',\mathpzc{d}'}(\R^{d'},\vec{w}';X))$ with compact Fourier support contained in $[-R_{l},R_{l}]^{\mathpzc{d}_{l}}$, so that we may apply Step I to obtain that
\[
\norm{1_{K}f}_{L^{\tilde{p}_{l}}(\R^{\mathpzc{d}_{l}},\tilde{w}_{l};L^{\vec{p}',\mathpzc{d}'}(\R^{d'},\vec{w}';X))} \leq CR_{l}^{\delta_{l}}
\norm{1_{K}f}_{L^{p_{l}}(\R^{\mathpzc{d}_{l}},w_{l};L^{\vec{p}',\mathpzc{d}'}(\R^{d'},\vec{w}';X))}
\]
for some constant $C>0$ independent of $f$ and $K$. Since
$L^{\tilde{\vec{p}},\mathpzc{d}}(\R^{d},\tilde{\vec{w}};X) = L^{\tilde{p}_{l}}(\R^{\mathpzc{d}_{l}},\tilde{w}_{l};L^{\vec{p}',\mathpzc{d}'}(\R^{d'},\vec{w}';X))$
and $L^{\vec{p},\mathpzc{d}}(\R^{d},\vec{w};X) = L^{p_{l}}(\R^{\mathpzc{d}_{l}},w_{l};L^{\vec{p}',\mathpzc{d}'}(\R^{d'},\vec{w}';X))$, the desired result follows by taking $K=K_{n} = [-n,n]^{\mathpzc{d}_{l}}$ and letting $n \to \infty$.

\textbf{Step III.} \emph{The case $\#J=1$:}

Let's say that $J=\{j_{0}\}$. Then, as a consequence of the Banach space-valued Paley-Wiener-Schwartz theorem, for each fixed $x''= (x_{j_{0}+1},\ldots,x_{l}) \in \R^{\mathpzc{d}_{j_{0}+1}+\ldots+\mathpzc{d}_{l}}$ we have that $f(\cdot,x'')$ defines an $X$-valued tempered distribution having compact Fourier support contained in $\prod_{j=1}^{j_{0}}[-R_{j},R_{j}]^{\mathpzc{d}_{j}}$. The desired inequality follows by
applying Step II to $f(\cdot,x'')$ for each $x''$ and subsequently taking $L^{(p_{j_{0}+1},\ldots,p_{l}),(\mathpzc{d}_{j_{0}+1},\ldots,\mathpzc{d}_{l})}(\R^{\mathpzc{d}_{j_{0}+1}+\ldots+\mathpzc{d}_{l}},(w_{j_{0}+1},\ldots,w_{l});X)$-norms with respect to $x''$.

\textbf{Step IV.} \emph{The general case:}

Just apply Step III repeatedly ($\#J$ times).
\end{proof}

\section{Proof of the Main Result}\label{pbvp:sec:pibvp}

In this section we prove the main result of this paper, Theorem~\ref{PIBVP:thm:main_result}.

\subsection{Necessary Conditions on the Initial-Boundary Data}\label{PIBVP:subsection:space_initial-boundary_data}

Let the notations and assumptions be as in Theorem~\ref{PIBVP:thm:main_result}.
Suppose that $g=(\mathcal{B}_{1}(D)u,\ldots,\mathcal{B}_{n}(D)u)$ and $u_{0} = \mathrm{tr}_{t=0}u$ for some $u \in \U^{p,q}_{\gamma,\mu}$.
We show that $(g,u_{0}) \in \D^{p,q}_{\gamma,\mu}$.

It follows from \cite[Theorem~1.1]{Meyries&Veraar_Traces} (also see \cite[Theorem~3.4.8]{Pruess&Simonett2016_book}) that
\[
\mathrm{tr}_{t=0}\left[ W^{1}_{q}(\R,v_{\mu};L^{p}(\R^{d},w_{\gamma};X)) \cap L^{q}(\R,v_{\mu};W^{2n}_{p}(\R^{d},w_{\gamma};X)) \right]
= B^{2n(1-\frac{1+\mu}{q})}_{p,q}(\R^{d},w_{\gamma};X).
\]
Using standard techniques one can derive the same result with $\R$ replaced by $J$ and $\R^{d}$ replaced by $\mathscr{O}$:
\begin{equation}\label{PIBVP:eq:initial_data}
\mathrm{tr}_{t=0}[\U^{p,q}_{\gamma,\mu}] = \I^{p,q}_{\gamma,\mu}.
\end{equation}
In particular, we must have $u_{0} \in \I^{p,q}_{\gamma,\mu} $.

In order to show that $g = (g_{1},\ldots,g_{n}) \in \G^{p,q}_{\gamma,\mu}$, we claim that
\begin{equation}\label{PIBVP:eq:boundary_data}
\mathcal{B}_{j}(D) \in \mathcal{B}(\U^{p,q}_{\gamma,\mu},\G^{p,q}_{\gamma,\mu,j}), \quad\quad j=1,\ldots,n.
\end{equation}
Combining the fact that
\[
L^{q}(\R,v_{\mu};L^{p}(\R^{d},w_{\gamma};X)) = L^{(p,q),(d,1)}(\R^{d+1},(w_{\gamma},v_{\mu});X) \hookrightarrow \mathcal{S}'(\R^{d+1};X)
\]
is a $\left((d,1),(\frac{1}{2n},1)\right)$-admissible Banach space (cf.\ \eqref{PIBVP:eq:prelim:anisotrope_mixed-norm_fm}) with \eqref{PIBVP:eq:prelim:identities_Sobolev&Bessel-potential}, \eqref{PIBVP:eq:prelim:differential} and standard techniques of localization, we find
\[
D^{\beta}_{x} \in \mathcal{B}\left(\U^{p,q}_{\gamma,\mu},H^{1-\frac{|\beta|}{2n}}_{q,\mu}(J;
L^{p}(\mathscr{O},w^{\partial\mathscr{O}}_{\gamma};X)) \cap
L^{q}_{\mu}(J;W^{2n-|\beta|}_{p}(\mathscr{O},w^{\partial\mathscr{O}}_{\gamma};X)) \right), \quad\quad \beta \in \N^{d}, |\beta| < 2n.
\]
From Theorem~\ref{PIBVP:thm:traces_main_result} it thus follows that, for each $\beta \in \N^{d}$, $j \in \{1,\ldots,n\}$ with $|\beta| \leq n_{j}$, $\mathrm{tr}_{\partial\mathscr{O}} \circ D^{\beta}_{x}$ is continuous linear operator
\[
\mathrm{tr}_{\partial\mathscr{O}} \circ D^{\beta}_{x}:
\U^{p,q}_{\gamma,\mu} \longra F^{\kappa_{j,\gamma}+\frac{n_{j}-|\beta|}{2n}}_{q,p}(J,v_{\mu};L^{p}(\partial\mathscr{O};X)) \cap
L^{q}(J,v_{\mu};F^{2n\kappa_{j,\gamma}+n_{j}-|\beta|}(\partial\mathscr{O};X)).
\]
The regularity assumption $(\mathrm{SB})$ on the coefficients $b_{j,\beta}$ thus give \eqref{PIBVP:eq:boundary_data}, where we use Lemmas \ref{PIBVP:appendix:lemma:BUC_top_order_coeff}, \ref{PIBVP:lemma:appendix:sec:localization;pre-variant_(2.3.13)} and
\ref{PIBVP:lemma:appendix:sec:localization;Sob-embedding_pre-variant_(2.3.13)} for $|\beta|=n_{j}$ and Lemma~\ref{PIBVP:appendix:lemma:lower_order_term} for $|\beta_{j}|<n_{j}$.

Finally, suppose that $\kappa_{j,\gamma} > \frac{1+\mu}{q}$.
Then, by combination of \eqref{PIBVP:eq:initial_data}, \eqref{PIBVP:eq:boundary_data} and Remark \ref{PIBVP:rmk:thm:main_result;compatability_cond},
\[
\mathrm{tr}_{t=0} \circ \mathcal{B}_{j}(D), \mathcal{B}^{t=0}_{j}(D) \circ \mathrm{tr}_{t=0} \in \mathcal{B}(\U^{p,q}_{\gamma,\mu},L^{0}(\partial\mathscr{O};X)), \quad\quad j=1,\ldots,n.
\]
By a density argument these operators coincide. Hence,
\[
\mathrm{tr}_{t=0}g_{j} - \mathcal{B}^{t=0}_{j}(D)u_{0} = [\mathrm{tr}_{t=0} \circ \mathcal{B}_{j}(D) -\mathcal{B}^{t=0}_{j}(D) \circ \mathrm{tr}_{t=0}]u = 0.
\]

\subsection{Elliptic Boundary Value Model Problems}\label{PIBVP:sec:elliptic_problems}

Let $X$ be a UMD Banach space. Let $\mathcal{A}(D) = \sum_{|\alpha| = 2n}a_{\alpha}D^{\alpha}$
$\mathcal{B}_{j}(D) = \sum_{|\beta| = n_{j}}b_{j,\beta}\mathrm{tr}_{\partial\R^{d}_{+}}D^{\beta}$, $j=1,\ldots,n$ with constant coefficients $a_{\alpha},b_{\beta,j} \in \mathcal{B}(X)$.

In this subsection we study the elliptic boundary value problem
\begin{equation}\label{pbvp:eq:ebvp}
\begin{array}{rll}
\lambda v + \mathcal{A}(D)v &= 0, & \\
\mathcal{B}_{j}(D)v &= g_{j}, & j=1,\ldots,n,
\end{array}
\end{equation}
on $\R^{d}_{+}$.
By the trace result of Corollary~\ref{PIBVP:cor:prop:trace_isotropic_B&F;W&H}, in order to get a solution $v \in W^{2n}_{p}(\R^{d}_{+},w_{\gamma};X)$ we need $g=(g_{1},\ldots,g_{n}) \in \prod_{j=1}^{n}F_{p,p}^{2n\kappa_{j,\gamma}}(\R^{d-1};X)$.
In Proposition \ref{PIBVP:prop:opl_elliptisch_probleem_randwaarden} we will see that there is existence and uniqueness plus a certain representation for the solution (which we will use to solve \eqref{PIBVP:eq:pbvp_sol_for_just_boundary_data}).
In this representation we have the operator from the following lemma.

\begin{lemma}\label{PIBVP:lemma:isom_Bessel_pot_analytic_dependence}
Let $E$ be a UMD Banach space, let $p \in (1,\infty)$, $w \in A_{p}(\R^{d})$, and $n \in \Z_{>0}$.
For each $\lambda \in \C \setminus (-\infty,0]$ and $\sigma \in \R$ we define $L^{\sigma}_{\lambda} \in \mathcal{L}(\mathcal{S}'(\R^{d};E))$ by
\[
L^{\sigma}_{\lambda}f := \mathscr{F}^{-1}[(\lambda + |\,\cdot\,|^{2n})^{\sigma}\hat{f}] \quad\quad (f \in \mathcal{S}'(\R^{d};E)).
\]
Then $L^{\sigma}_{\lambda}$ restricts to a topological linear isomorphism from $H^{s+2n\sigma}_{p}(\R^{d},w;E)$ to $H^{s}_{p}(\R^{d},w;E)$ (with inverse $L^{-\sigma}_{\lambda}$) for each $s \in \R$. Moreover,
\begin{equation}\label{pbvp:eq:lemma;isom_Bessel_pot_analytic_dependence}
\C \setminus (-\infty,0]  \owns \lambda \mapsto L^{\sigma}_{\lambda} \in \mathcal{B}(H^{s+2n\sigma}_{p}(\R^{d},w;E),H^{s}_{p}(\R^{d},w;E))
\end{equation}
defines an analytic mapping for every $\sigma \in \R$ and $s \in \R$.
\end{lemma}
\begin{proof}

For the first part one only needs to check the Mikhlin condition corresponding to \eqref{PIBVP:eq:prelim:anisotrope_mixed-norm_fm} (with $l=1$ and $\vec{a}=1$) for the symbol $\xi \mapsto (1+|\xi|^{2})^{-(n\sigma)/2}(\lambda+|\xi|^{2n})^{\sigma}$.
So let us go to the analyticity statement.
We only treat the case $\sigma \in \R \setminus \N$, the case $\sigma \in \N$ being easy.
So suppose that $\sigma \in \R \setminus \N$ and fix a $\lambda_{0} \in \C \setminus (-\infty,0]$.
We shall show that $\lambda \mapsto L^{\sigma}_{\lambda}$ is analytic at $\lambda_{0}$.
Since $L^{\tau}_{\lambda_{0}}$ is a topological linear isomorphism from $H^{s+2n\tau}_{p}(\R^{d},w;E)$ to $H^{s}_{p}(\R^{d},w;E)$, $\tau \in \R$, for this it suffices to show that
\[
\C \setminus (-\infty,0]  \owns \lambda \mapsto  L^{\sigma}_{\lambda}L^{-\sigma}_{\lambda_{0}} = L^{\frac{s}{n}}_{\lambda_{0}}L^{\sigma}_{\lambda}L^{-\frac{1}{n}(s+2n\sigma)}_{\lambda_{0}} \in \mathcal{B}(L^{p}(\R^{d},w;E))
\]
is analytic at $\lambda_{0}$. To this end, we first observe that, for each $\xi \in \R^{d}$,
\[
\C \setminus (-\infty,0] \owns \lambda \mapsto (\lambda + |\xi|^{2n})^{\sigma}(\lambda_{0} + |\xi|^{2n})^{-\sigma} \in \C
\]
is an analytic mapping with power series expansion at $\lambda_{0}$ given by
\begin{equation}\label{pbvp:eq:lemma;isom_Bessel_pot_analytic_dependence;power_series_symbol}
(\lambda + |\xi|^{2n})^{\sigma}(\lambda_{0} + |\xi|^{2n})^{-\sigma} =
1 + \sigma(\lambda_{0}+|\xi|^{2n})^{-1}(\lambda-\lambda_{0})
+ \sigma(\sigma-1)(\lambda_{0}+|\xi|^{2n})^{-2}(\lambda-\lambda_{0})^{2} + \ldots
\end{equation}
for $\lambda \in B(\lambda_{0},\delta)$, where $\delta:= d\left(0,\{\lambda_{0}+t \mid t \geq 0\}\right) > 0$.
We next recall that $L_{\lambda_{0}}^{-1}$ restricts to a topological linear isomorphism from $L^{p}(\R^{d},w;E)$ to $H^{2n}_{p}(\R^{d},w;E)$; in particular, $L_{\lambda_{0}}^{-1}$ restricts to a bounded linear operator on $L^{p}(\R^{d},w;E)$.
Since $L_{\lambda_{0}}^{-k} = (L_{\lambda_{0}}^{-1})^{k}$ for every $k \in \N$, there thus exists a constant $C > 0$ such that
\begin{equation}\label{pbvp:eq:lemma;isom_Bessel_pot_analytic_dependence;norm_estimate_operators}
\norm{L_{\lambda_{0}}^{-k}}_{\mathcal{B}(L^{p}(\R^{d},w;E))} \leq C^{k}, \quad\quad \forall k \in \N.
\end{equation}
Now we let $\rho > 0$ be the radius of convergence of the power series $z \mapsto \sum_{k \in \N}\left[\prod_{j=0}^{k-1}(\sigma-j)\right]C^{k}z^{k}$, set $r:= \min(\delta,\rho) > 0$, and define, for each $\lambda \in B(\lambda_{0},r)$, the multiplier symbols $m^{\lambda},m^{\lambda}_{0},m^{\lambda}_{1},\ldots:\R^{d} \longra \C$ by
\[
m^{\lambda}(\xi) := (\lambda + |\xi|^{2n})^{\sigma}(\lambda_{0} + |\xi|^{2n})^{-\sigma} \quad \mbox{and} \quad
m^{\lambda}_{N}(\xi) := \sum_{k = 0}^{N}\left[\prod_{j=0}^{k-1}(\sigma-j)\right](\lambda_{0}+|\xi|^{2n})^{-k}(\lambda-\lambda_{0})^{k}.
\]
Then, by \eqref{pbvp:eq:lemma;isom_Bessel_pot_analytic_dependence;power_series_symbol} and \eqref{pbvp:eq:lemma;isom_Bessel_pot_analytic_dependence;norm_estimate_operators}, we get
\[
m^{\lambda}(\xi) = \lim_{N \to \infty}m^{\lambda}_{N}(\xi), \quad\quad \xi \in \R^{d}
\]
and
\[
\lim_{N,M \to \infty}[T_{m^{\lambda}_{N}}-T_{m^{\lambda}_{M}}] = 0 \:\:\:\mbox{in}\:\:\:\mathcal{B}(L^{p}(\R^{d},w;E)),
\]
respectively. Via the $A_{p}$-weighted version of \cite[Facts~3.3.b]{Kunstmann&Weis} we thus obtain that
\[
L^{\sigma}_{\lambda}L^{-\sigma}_{\lambda_{0}} = T_{m^{\lambda}} = \lim_{N \to \infty} T_{m^{\lambda}_{N}} =
\lim_{N \to \infty}\sum_{k = 0}^{N}\left[\prod_{j=0}^{k-1}(\sigma-j)\right]L_{\lambda_{0}}^{-k}(\lambda-\lambda_{0})^{k}
\quad \mbox{in}\:\:\:\mathcal{B}(L^{p}(\R^{d},w;E))
\]
for $\lambda \in B(\lambda_{0},r)$. This shows that the map $\C \setminus (-\infty,0] \owns \lambda \mapsto L^{\sigma}_{\lambda}L^{-\sigma}_{\lambda_{0}} \in \mathcal{B}(L^{p}(\R^{d},w;E))$ is analytic at $\lambda_{0}$, as desired.
\end{proof}

Before we can state Proposition~\ref{PIBVP:prop:opl_elliptisch_probleem_randwaarden}, we first need to introduce some notation.
Given a UMD Banach space $X$ and a natural number $k \in \N$, we have, for the UMD space $E=L^{p}(\R_{+},|\,\cdot\,|^{\gamma};X)$, the natural inclusion
\[
W^{k}_{p}(\R^{d}_{+},w_{\gamma};X) \hookrightarrow W^{k}_{p}(\R^{d-1};L^{p}(\R_{+},|\,\cdot\,|^{\gamma};X)) = H^{k}_{p}(\R^{d-1};E)
\]
and the natural identification
\[
L^{p}(\R^{d}_{+},w_{\gamma};X) = H^{0}_{p}(\R^{d-1};E).
\]
By Lemma \ref{PIBVP:lemma:isom_Bessel_pot_analytic_dependence} we accordingly have that, for $\lambda \in \C \setminus (-\infty,0]$, that the partial Fourier multiplier operator
\[
L^{k/2n}_{\lambda} \in \mathcal{L}(\mathcal{S}'(\R^{d-1};\mathcal{D}'(\R_{+};X))),\,
f \mapsto \mathscr{F}^{-1}_{x'}\left[ \left( \xi' \mapsto (\lambda+|\xi'|^{2n})^{k/2n} \right)\mathscr{F}_{x'}f \right],
\]
restricts to a bounded linear operator
\[
L_{\lambda}^{k/2n} \in \mathcal{B}(W^{k}_{p}(\R^{d}_{+},w_{\gamma};X),L^{p}(\R^{d}_{+},w_{\gamma};X)).
\]
Moreover, we even get an analytic operator-valued mapping
\[
\C \setminus (-\infty,0] \longra \mathcal{B}(W^{k}_{p}(\R^{d}_{+},w_{\gamma};X),L^{p}(\R^{d}_{+},w_{\gamma};X)),\, \lambda \mapsto L_{\lambda}^{k/2n}.
\]
In particular, we have
\begin{equation}\label{pbvp:eq:prop;opl_elliptisch_probleem_randwaarden;bddness_first_factors}
L_{\lambda}^{1-\frac{n_{j}}{2n}}, L_{\lambda}^{1-\frac{n_{j}+1}{2n}}D_{y} \in
\mathcal{B}(W^{2n-n_{j}}_{p}(\R^{d}_{+},w_{\gamma};X),L^{p}(\R^{d}_{+},w_{\gamma};X)), \quad\quad j=1,\ldots,n,
\end{equation}
with analytic dependence on the parameter $\lambda \in \C \setminus (-\infty,0]$.

\begin{prop}\label{PIBVP:prop:opl_elliptisch_probleem_randwaarden}
Let $X$ be a UMD Banach space, $p \in (1,\infty)$, $\gamma \in (-1,p-1)$, and assume that $(\mathcal{A},\mathcal{B}_{1},\ldots,\mathcal{B}_{n})$ satisfies $(\mathrm{E})$ and $(\mathrm{LS})$ for some $\phi \in (0,\pi)$.
Then, for each $\lambda \in \Sigma_{\pi-\phi}$, there exists an operator
\[
\mathcal{S}(\lambda) =
\left(\begin{array}{ccc}
\mathcal{S}_{1}(\lambda) & \ldots & \mathcal{S}_{n}(\lambda)
\end{array}\right)
\in \mathcal{B}\left(\bigoplus_{j=1}^{n}F_{p,p}^{2n\kappa_{j,\gamma}}(\R^{d-1};X),W^{2n}_{p}(\R^{d}_{+},w_{\gamma};X)\right)
\]
which assigns to a $g \in \bigoplus_{j=1}^{n}F_{p,p}^{2n\kappa_{j,\gamma}}(\R^{d-1};X)$ the unique solution $v=\mathcal{S}(\lambda)g \in W^{2n}_{p}(\R^{d}_{+},w_{\gamma};X)$
of the elliptic boundary value problem
\begin{equation}\label{PIBVP:eq:prop;opl_elliptisch_probleem_randwaarden}
\begin{array}{rll}
\lambda v + \mathcal{A}(D)v &= 0, & \\
\mathcal{B}_{j}(D)v &= g_{j}, & j=1,\ldots,n;
\end{array}
\end{equation}
recall here that $\kappa_{j,\gamma} = 1-\frac{n_{j}}{2n}-\frac{1}{2np}(1+\gamma)$.
Moreover, for each $j \in \{1,\ldots,n\}$ we have that
\[
\tilde{\mathcal{S}}_{j}: \Sigma_{\pi-\phi} \longra \mathcal{B}(W^{2n-n_{j}}_{p}(\R^{d}_{+},w_{\gamma};X),W^{2n}_{p}(\R^{d}_{+},w_{\gamma};X)),\, \lambda \mapsto \tilde{\mathcal{S}}_{j}(\lambda) := \mathcal{S}_{j}(\lambda) \circ \mathrm{tr}_{y=0}
\]
defines an analytic mapping, for which the operators $D^{\alpha}\tilde{\mathcal{S}}_{j}(\lambda) \in  \mathcal{B}(W^{2n-n_{j}}_{p}(\R^{d}_{+},w_{\gamma};X),L^{p}(\R^{d}_{+},w_{\gamma};X))$, $|\alpha| \leq 2n$,
can be represented as
\begin{equation}\label{PIBVP:eq:prop;opl_elliptisch_probleem_randwaarden;rep_formula}
D^{\alpha}\tilde{\mathcal{S}}_{j}(\lambda) = \mathcal{T}^{1}_{j,\alpha}(\lambda)L_{\lambda}^{1-\frac{n_{j}}{2n}} + \mathcal{T}^{2}_{j,\alpha}(\lambda)L_{\lambda}^{1-\frac{n_{j}+1}{2n}}D_{y}
\end{equation}
for analytic operator-valued mappings
\begin{equation}\label{PIBVP:eq:prop;opl_elliptisch_probleem_randwaarden;operator_T_rep_formula}
\mathcal{T}^{i}_{j,\alpha}: \Sigma_{\pi-\phi} \longra \mathcal{B}(L^{p}(\R^{d}_{+},w_{\gamma};X)),\, \lambda \mapsto
\mathcal{T}^{i}_{j,\alpha}(\lambda),
\quad\quad i \in \{1,2\},
\end{equation}
satisfying the $\mathcal{R}$-bounds
\begin{equation}\label{PIBVP:eq:prop;opl_elliptisch_probleem_randwaarden;R-bounds}
\mathcal{R}\{ \lambda^{k+1-\frac{|\alpha|}{2n}}\partial_{\lambda}^{k}\mathcal{T}^{i}_{j,\alpha}(\lambda) \mid \lambda \in \Sigma_{\pi-\phi} \} < \infty, \quad\quad k \in \N.
\end{equation}
\end{prop}

\begin{proof}[Comments on the proof of Proposition \ref{PIBVP:prop:opl_elliptisch_probleem_randwaarden}]
This proposition can be proved in the same way as \cite[Lemma~4.3$\&$Lemma~4.4]{DHP2}. In fact, in the unweighted case this is just a modification of \cite[Lemma~4.3$\&$Lemma~4.4]{DHP2} (also see the formulation of \cite[Lemma~2.2.6]{Mey_PHD-thesis}).
Here \cite[Lemma~4.3]{DHP2} corresponds to the existence of the solution operator, whose construction was essentially already contained in \cite{DHP1}, plus its representation, and \cite[Lemma~4.4]{DHP2} basically corresponds to the analytic dependence of \eqref{PIBVP:eq:prop;opl_elliptisch_probleem_randwaarden;operator_T_rep_formula} plus the $\mathcal{R}$-bounds \eqref{PIBVP:eq:prop;opl_elliptisch_probleem_randwaarden;R-bounds}.
The analytic dependence of the operators $\tilde{\mathcal{S}}_{j}(\lambda)$ on $\lambda$ subsequently follow from Lemma~\ref{PIBVP:lemma:isom_Bessel_pot_analytic_dependence} and \eqref{PIBVP:eq:prop;opl_elliptisch_probleem_randwaarden;rep_formula}.
For more details we refer to \cite[Chapter~6]{Lindemulder_master-thesis} and Remark~\ref{PIBVP:rmk:prop:opl_elliptisch_probleem_randwaarden;ext_operator}.
\end{proof}

\begin{remark}
We could have formulated Proposition~\ref{PIBVP:prop:opl_elliptisch_probleem_randwaarden} only in terms of the mappings $\tilde{\mathcal{S}}_{j}$. Namely, for each $j \in \{1,\ldots,n\}$ there exists an analytic mapping
\[
\tilde{\mathcal{S}}_{j}: \Sigma_{\pi-\phi} \longra \mathcal{B}(W^{2n-n_{j}}_{p}(\R^{d}_{+},w_{\gamma};X),W^{2n}_{p}(\R^{d}_{+},w_{\gamma};X)),\, \lambda \mapsto \tilde{\mathcal{S}}_{j}(\lambda)
\]
with the property that, for every $u \in W^{2n}_{p}(\R^{d}_{+},w_{\gamma};X)$, $v = \tilde{\mathcal{S}}_{j}u$ is the unique solution in $W^{2n}_{p}(\R^{d}_{+},w_{\gamma};X)$ of \eqref{PIBVP:eq:prop;opl_elliptisch_probleem_randwaarden} with $g_{i}=\delta_{i,j}\mathcal{B}_{i}(D)u$, for which the operators
\[
D^{\alpha}\tilde{\mathcal{S}}_{j}(\lambda) \in  \mathcal{B}(W^{2n-n_{j}}_{p}(\R^{d}_{+},w_{\gamma};X),L^{p}(\R^{d}_{+},w_{\gamma};X)), |\alpha| \leq 2n,
\]
can be represented as \eqref{PIBVP:eq:prop;opl_elliptisch_probleem_randwaarden;rep_formula} for analytic operator-valued mappings
\eqref{PIBVP:eq:prop;opl_elliptisch_probleem_randwaarden;operator_T_rep_formula} satisfying the $\mathcal{R}$-bounds
\eqref{PIBVP:eq:prop;opl_elliptisch_probleem_randwaarden;R-bounds}.
Then, given extension operators $\mathcal{E}_{j} \in \mathcal{B}(F^{2n\kappa_{j,\gamma}}_{p,p}(\R^{d-1};X),W^{2n-n_{j}}_{p}(\R^{d}_{+},w_{\gamma};X))$ (right inverse of the trace $\mathrm{tr}_{y=0} \in \mathcal{B}(W^{2n-n_{j}}_{p}(\R^{d}_{+},w_{\gamma};X),F^{2n\kappa_{j,\gamma}}_{p,p}(\R^{d-1};X))$), $j=1,\ldots,n$, the composition $\mathcal{S}(\lambda) = ( \mathcal{S}_{1}(\lambda) \ldots \mathcal{S}_{n}(\lambda) ) := ( \mathcal{S}_{1}(\lambda) \ldots \mathcal{S}_{n}(\lambda) ) \circ (\mathcal{E}_{1} \ldots \mathcal{E}_{n})$ defines the desired solution operator.

In this formulation the proposition the weight $w_{\gamma}$ can actually be replaced by any weight $w$ on $\R^{d}$ which is uniformly $A_{p}$ in the $y$-variable. Indeed, in the proof the weight only comes into play in \cite[Lemma~7.1]{DHP1}. For weights $w$ of the form $w(x',y) = v(x')|y|^{\gamma}$ with $v \in A_{p}(\R^{d-1})$ we can then still define $\mathcal{S}(\lambda)$ as above thanks to the available trace theory from Section~\ref{PIBVP:subsec:sec:traces;isotropic}.
\end{remark}

\begin{remark}\label{PIBVP:rmk:prop:opl_elliptisch_probleem_randwaarden;ext_operator}
In \cite{DHP2} the specific extension operator $\mathcal{E}_{\lambda} = e^{-\,\cdot\,L^{1/2n}_{\lambda}}$ was used in the construction of the solution operator $\mathcal{S}(\lambda) = (\mathcal{S}_{1}(\lambda),\ldots,\mathcal{S}_{n}(\lambda))$, which has the advantageous property that $D_{y}\mathcal{E}_{\lambda} = \imath L_{\lambda}^{1/2n}\mathcal{E}_{\lambda}$. Whereas the in this way obtained representation formulae $\mathcal{S}_{j}(\lambda) = \mathcal{T}_{j}(\lambda)L_{\lambda}^{1-\frac{n_{j}}{2n}}\mathcal{E}_{\lambda}$ can only be used in the case $q=p$ to solve
\eqref{PIBVP:eq:pbvp_sol_for_just_boundary_data} via a Fourier transformation in time (cf.\ \cite[Proposition~4.5]{DHP2} and \cite[Lemma~2.2.7]{Mey_PHD-thesis}), our representation formulae \eqref{PIBVP:eq:prop;opl_elliptisch_probleem_randwaarden;rep_formula} can (in combination with the theory of anisotropic function spaces) be used to solve \eqref{PIBVP:eq:pbvp_sol_for_just_boundary_data} in the full parameter range $q,p \in (1,\infty)$ (cf.\ Corollary~\ref{PIBVP:lemma:pbvp_sol_for_just_boundary_data}). However, the alternative more involved proof of Denk, Hieber $\&$ Pr\"uss \cite[Theorem~2.3]{DHP2} also contains several ingredients which are of independent interest.
\end{remark}

\subsection{Solving Inhomogeneous Boundary Data for a Model Problem}

Let the notations and assumptions be as in Theorem~\ref{PIBVP:thm:main_result}, but for the model problem case of top order constant coefficients on the half-space considered in Section~\ref{PIBVP:sec:elliptic_problems}.

The goal of this subsection is to solve the model problem
\begin{equation}\label{PIBVP:eq:pbvp_sol_for_just_boundary_data}
\begin{array}{rll}
\partial_{t}u + (1+\mathcal{A}(D))u &= 0,  \\
\mathcal{B}_{j}(D)u &= g_{j}, & j=1,\ldots,n, \\
\mathrm{tr}_{t=0}u &= 0,
\end{array}
\end{equation}
for $g=(g_{1},\ldots,g_{n})$ with $(0,g,0) \in \D^{p,q}_{\gamma,\mu}$.

Let us first observe that, in view of the compatibility condition in the definition of $\D^{p,q}_{\gamma,\mu}$,
$(0,g,0) \in \D^{p,q}_{\gamma,\mu}$ if and only if
\begin{eqnarray*}
g_{j} \in {_{0}}\G_{j}
&:=& {_{0,(0,d)}}F^{\kappa_{j,\gamma},(\frac{1}{2n},1)}_{(p,q),p,(d-1,1)}(\R^{d-1} \times \R_{+},(1,v_{\mu});X) \\
&:=& \left\{\begin{array}{ll}
F^{\kappa_{j,\gamma},(\frac{1}{2n},1)}_{(p,q),p,(d-1,1)}(\R^{d-1} \times \R_{+},(1,v_{\mu});X), & \kappa_{j,\gamma} < \frac{1+\mu}{q},\\
\left\{ w \in F^{\kappa_{j,\gamma},(\frac{1}{2n},1)}_{(p,q),p,(d-1,1)}(\R^{d-1} \times \R_{+},(1,v_{\mu});X) : \mathrm{tr}_{t=0}w=0 \right\}, & \kappa_{j,\gamma} > \frac{1+\mu}{q},
\end{array}\right.
\end{eqnarray*}
for all $j \in \{1,\ldots,n\}$. Defining
\[
{_{0}}\G := {_{0}}\G_{1} \oplus \ldots \oplus {_{0}}\G_{n},
\]
we thus have $(0,g,0) \in \D^{p,q}_{\gamma,\mu}$ if and only if $g \in {_{0}}\G$.
So we need to solve \eqref{PIBVP:eq:pbvp_sol_for_just_boundary_data} for $g \in {_{0}}\G$.

We will solve \eqref{PIBVP:eq:pbvp_sol_for_just_boundary_data} by passing to the corresponding problem on $\R$ (instead of $\R_{+}$).
The advantage of this is that it allows us to use the Fourier transform in time. This will give
\[
\mathscr{F}_{t}u(\theta) = \mathcal{S}(1+\imath\theta)(\mathscr{F}_{t}g_{1}(\theta),\ldots,\mathscr{F}_{t}g_{n}(\theta)),
\]
where $\mathcal{S}(1+\imath\theta)$ is the solution operator from Proposition \ref{PIBVP:prop:opl_elliptisch_probleem_randwaarden}.

Recall that for the operator $\tilde{\mathcal{S}}_{j}(\lambda) = \mathcal{S}_{j}(\lambda) \circ \mathrm{tr}_{y=0}$ we have the representation formula \eqref{PIBVP:eq:prop;opl_elliptisch_probleem_randwaarden;rep_formula} in which the operators
$L^{\sigma}_{\lambda}$ occur. It will be useful to note that, for $h \in \mathcal{S}(\R^{d}_{+} \times \R;X)$,
\begin{eqnarray}
L^{\sigma}_{1+\imath \theta_{0}}[(\mathscr{F}_{t}h)(\,\cdot\,,\theta)]
&=& \mathscr{F}^{-1}_{x'}[\left((y,\xi') \mapsto (1+\imath \theta_{0} + |\xi'|^{2n})\right)\mathscr{F}_{(x',t)}h(\,\cdot\,,\theta_{0})] \nonumber \\
&=& \left[\mathscr{F}_{t}\mathscr{F}^{-1}_{(x',t)}[\left((y,\xi',\theta) \mapsto (1+\imath \theta + |\xi'|^{2n})\right)\mathscr{F}_{(x',t)}h] \right](\,\cdot\,,\theta_{0}) \nonumber \\
&=& (\mathscr{F}_{t}L^{\sigma}h)(\,\cdot\,,\theta_{0}), \label{pibvp:eq:omwisselen_L}
\end{eqnarray}
where
\[
L^{\sigma} \in \mathcal{L}(\mathcal{S}'(\R^{d-1} \times \R;\mathcal{D}'(\R_{+};X))),\,
f \mapsto \mathscr{F}^{-1}_{(x',t)}\left[ \left( (\xi',\theta) \mapsto (1 + \imath\theta + |\xi'|^{2n})^{\sigma} \right)\mathscr{F}_{(x',t)}f\right].
\]

\begin{lemma}\label{PIBVP:lemma;lemma:pbvp_sol_for_just_boundary_data;isomorphism_between_BP}
Let $E$ be a UMD space, $p,q \in (1,\infty)$, $v \in A_{q}(\R)$, and $n \in \Z_{>0}$.
For each $\sigma \in \R$,
\[
\mathcal{S}'(\R^{d-1}\times\R;E) \longra \mathcal{S}'(\R^{d-1}\times\R;E),\, f \mapsto \mathscr{F}^{-1}\left[\left((\xi_{1},\xi_{2}) \mapsto (1+\imath\xi_{2}+ |\xi_{1}|^{2n})^{\sigma}\right)\hat{f}\right]
\]
restricts to a bounded linear operator
\[
H^{\sigma,(\frac{1}{2n},1)}_{(p,q),(d-1,1)}(\R^{d-1}\times\R,(1,v);E) \longra H^{0,(\frac{1}{2n},1)}_{(p,q),(d-1,1)}(\R^{d-1}\times\R,(1,v);E).
\]
\end{lemma}
\begin{proof}

This can be shown by checking that the symbol
\[
\R^{d-1} \times \R \owns (\xi_{1},\xi_{2}) \mapsto \frac{(1+\imath\xi_{2}+|\xi_{1}|^{2n})^{\sigma}}{(1+|\xi_{1}|^{4n}+|\xi_{2}|^{2})^{\sigma/2}} \in \C
\]
satisfies the anisotropic Mikhlin condition from \eqref{PIBVP:eq:prelim:anisotrope_mixed-norm_fm}.
\end{proof}

\begin{lemma}\label{PIBVP:lemma:lemma:pbvp_sol_for_just_boundary_data;real_line}
Let $X$ be a UMD space, $q,p \in (1,\infty)$, $\gamma \in (-1,p-1)$, $v \in A_{q}(\R)$.
Put
\begin{equation}
\begin{split}
\overline{\G}_{j} &:= F^{\kappa_{j,\gamma}}_{q,p}(\R,v;L^{p}(\R^{d-1};X)) \cap L^{q}(\R,v;F^{2n\kappa_{j,\gamma}}_{p,p}(\R^{d-1};X)), \quad\quad j=1,\ldots,n, \\
\overline{\G}  &:= \overline{\G}_{1} \oplus \ldots \oplus \overline{\G}_{n}, \\
\overline{\U} &:= W^{1}_{q}(\R,v;L^{p}(\R^{d}_{+},w_{\gamma};X)) \cap L^{q}(\R,v;W^{2n}_{p}(\R^{d}_{+},w_{\gamma};X)),
\end{split}
\end{equation}
where we recall that $\kappa_{j,\gamma} = 1-\frac{n_{j}}{2n}-\frac{1}{2np}(1+\gamma) \in (0,1)$.
Furthermore, define ${_{0}}\overline{\G}_{j}$ similarly to ${_{0}}\G_{j}$ and put ${_{0}}\overline{\G}_{j} := {_{0}}\overline{\G}_{1} \oplus \ldots \oplus {_{0}}\overline{\G}_{n}$.
Then the problem
\begin{equation}\label{PIBVP:eq:pbvp_sol_for_just_boundary_data;real_line}
\begin{array}{rll}
\partial_{t}u + (1+\mathcal{A}(D))u &= 0,  \\
\mathcal{B}_{j}(D)u &= g_{j}, & j=1,\ldots,n, \\
\end{array}
\end{equation}
admits a bounded linear solution operator $\overline{\mathscr{S}} :  \overline{\G} \longra \overline{\U}$ which maps ${_{0}}\overline{\G}$ to ${_{0}}\overline{\U} = \{ u \in \overline{\U} : u(0)=0 \}$.
\end{lemma}

For the statement that $\overline{\mathscr{S}}$ maps ${_{0}}\overline{\G}$ to ${_{0}}\overline{\U}$ we will use the following lemma.

\begin{lemma}\label{PIBVP:lemma:lemma:pbvp_sol_for_just_boundary_data;density_trace}
$\{g_{j} \in \mathcal{S}(\R^{d};X) : \mathrm{tr}_{t=0}g_{j}=0\}$ is dense in ${_{0}}\overline{\G}_{j}$
\end{lemma}
\begin{proof}
As a consequence of Theorem~\ref{functieruimten:thm:aTL_rep_intersection},
\[
{_{0}}\overline{\G}_{j} = {_{0}}F^{\kappa_{j,\gamma}}_{q,p}(\R,v_{\mu};L^{p}(\R^{d-1};X)) \cap L^{q}(\R,v_{\mu};F^{2n\kappa_{j,\gamma}}_{p,p}(\R^{d-1};X)),
\]
where
\[
{_{0}}F^{s}_{q,p}(\R,v_{\mu};Y) = \left\{\begin{array}{ll}
F^{s}_{q,p}(\R,v_{\mu};Y), & s < \frac{1+\mu}{q}, \\
\{ f \in F^{s}_{q,p}(\R,v_{\mu};Y) : \mathrm{tr}_{t=0}f=0 \}, & s > \frac{1+\mu}{q}.
\end{array}\right.
\]
Let $(S_{n})_{n \in \N}$ be the family of convolution operator corresponding to some $\varphi = (\varphi_{n})_{n \in \N} \in \Phi(\R^{d-1})$.
Then $S_{n} \stackrel{\mathrm{SOT}}{\longra} I$ as $n \to \infty$ in both $L^{p}(\R^{d-1};X)$ as $F^{2n\kappa_{j,\gamma}}_{p,p}(\R^{d-1};X)$. For the pointwise induced operator family we thus have $S_{n} \stackrel{\mathrm{SOT}}{\longra} I$ in ${_{0}}\overline{\G}_{j}$. Since
\[
L^{p}(\R^{d-1};X) \cap \mathscr{F}^{-1}\mathcal{E}'(\R^{d-1};X) \subset
F^{0}_{p,\infty}(\R^{d-1};X) \cap \mathscr{F}^{-1}\mathcal{E}'(\R^{d-1};X) \subset
F^{2n\kappa_{j,\gamma}}_{p,p}(\R^{d-1};X),
\]
it follows that
\[
{_{0}}F^{\kappa_{j,\gamma}}_{q,p}(\R,v_{\mu};F^{2n\kappa_{j,\gamma}}_{p,p}(\R^{d-1};X)) =
{_{0}}F^{\kappa_{j,\gamma}}_{q,p}(\R,v_{\mu};F^{2n\kappa_{j,\gamma}}_{p,p}(\R^{d-1};X)) \cap L^{q}(\R,v_{\mu};F^{2n\kappa_{j,\gamma}}_{p,p}(\R^{d-1};X))
\]
is dense in ${_{0}}\overline{\G}_{j}$; in fact,
\[
{_{0}}F^{\kappa_{j,\gamma}}_{q,p}(\R,v_{\mu};F^{2n\kappa_{j,\gamma}}_{p,p}(\R^{d-1};X))
\stackrel{d}{\hookrightarrow} {_{0}}\overline{\G}_{j}.
\]
Since
\begin{eqnarray*}
\{ f \in \mathcal{S}(\R) : f(0)=0 \} \otimes \mathcal{S}(\R^{d-1};X)
&\stackrel{d}{\subset}& \{ f \in \mathcal{S}(\R) : f(0)=0 \} \otimes F^{2n\kappa_{j,\gamma}}_{p,p}(\R^{d-1};X) \\
&\stackrel{d}{\subset}& {_{0}}F^{\kappa_{j,\gamma}}_{q,p}(\R,v_{\mu};F^{2n\kappa_{j,\gamma}}_{p,p}(\R^{d-1};X))
\end{eqnarray*}
by \cite{LMV_interpolation_boundary_cond}, the desired density follows.
\end{proof}

\begin{proof}[Proof of Lemma~\ref{PIBVP:lemma:lemma:pbvp_sol_for_just_boundary_data;real_line}]

(I) Put $\overline{\F}:= L^{q}(\R,v;L^{p}(\R^{d}_{+},w_{\gamma};X))$ and $V:= \mathscr{F}^{-1}C^{\infty}_{c}(\R^{d-1};X) \otimes \mathscr{F}^{-1}C^{\infty}_{c}(\R)$.
Then $V^{n}$ is dense in $\overline{\G}$. So, in view of
\[
\partial_{t} + (1+\mathcal{A}(D)) \in \mathcal{B}(\overline{\U},\overline{\F}) \quad\quad \mbox{and} \quad\quad
\mathcal{B}_{j}(D) \in \mathcal{B}(\overline{\U},\overline{\G}_{j}), \quad j=1,\ldots,n,
\]
it suffices to construct a solution operator $\overline{\mathscr{S}}:V^{n} \longra \overline{\U}$ which is bounded when $V^{n}$ carries the induced norm from $\overline{\G}$.
In order to define such an operator, fix $g=(g_{1},\ldots,g_{n}) \in V^{n}$.
Let
\begin{equation}\label{PIBVP:eq:lemma;pbvp_sol_for_just_boundary_date;extension_operator}
\mathcal{E}_{j} \in \mathcal{B}(\overline{\G},H^{1-\frac{n_{j}}{2n},(\frac{1}{2n},1)}_{(p,q),(d,1)}(\R^{d}_{+} \times \R,(w_{\gamma},v);X)), \quad j=1,\ldots,n,
\end{equation}
be extension operators (right-inverses of the trace operator $\mathrm{tr}_{y=0}$) as in Corollary~\ref{PIBVP:cor:thm:trace_TL}.
Then $\mathcal{E}_{j}$ maps $V^{n}$ into
$\mathcal{S}(\R^{d}_{+};X)) \otimes \mathscr{F}^{-1}(C^{\infty}_{c}(\R))$; in particular,
\[
\mathcal{E}_{j}g_{j} \in \mathcal{S}(\R^{d}_{+};X)) \otimes \mathscr{F}^{-1}(C^{\infty}_{c}(\R)),
\quad j=1,\ldots,n.
\]
So, for each $j \in \{1,\ldots,n\}$, we have
\[
\mathscr{F}_{t}\mathcal{E}_{j}g_{j} \in \mathcal{S}(\R^{d}_{+};X)) \otimes C^{\infty}_{c}(\R),
\]
and we may also view $\mathscr{F}_{t}\mathcal{E}_{j}g_{j}$ as a function
\[
[\theta \mapsto (\mathscr{F}_{t}\mathcal{E}_{j}g_{j})(\theta)] \in C^{\infty}_{c}(\R;W^{2n-n_{j}}_{p}(\R^{d}_{+},w_{\gamma};X)).
\]
Since
\[
[\theta \mapsto \tilde{\mathcal{S}}_{j}(1+\imath\theta)] \in C^{\infty}(\R;\mathcal{B}(W^{2n-n_{j}}_{p}(\R^{d}_{+},w_{\gamma};X),W^{2n}_{p}(\R^{d}_{+},w_{\gamma};X))), \quad j=1,\ldots,n,
\]
with $\tilde{\mathcal{S}}_{j}(1+\imath\theta)$ as in Proposition~\ref{PIBVP:prop:opl_elliptisch_probleem_randwaarden},
we may thus define
\[
\overline{\mathscr{S}}g := \mathscr{F}_{t}^{-1}\left[ \theta \mapsto \sum_{j=1}^{n}\tilde{\mathcal{S}}_{j}(1+\imath\theta)
(\mathscr{F}_{t}\mathcal{E}_{j}g_{j})(\theta) \right] \in \mathcal{S}(\R;W^{2n}_{p}(\R^{d}_{+},w_{\gamma};X))
\]

(II) We now show that $u=\overline{\mathscr{S}}g \in \mathcal{S}(\R;W^{2n}_{p}(\R^{d}_{+},w_{\gamma};X))$ is a solution of \eqref{PIBVP:eq:pbvp_sol_for_just_boundary_data;real_line} for $g \in V^{n}$.
To this end, let $\theta \in \R$ be arbitrary. Then we have that $(\mathscr{F}_{t}\mathcal{E}_{j}g_{j})(\theta) \in \mathcal{S}(\R^{d}_{+};X) \subset W^{2n-n_{j}}_{p}(\R^{d}_{+},w_{\gamma};X)$ and $(\mathscr{F}_{t}g_{j})(\theta) \in \mathcal{S}(\R^{d-1};X) \subset F^{2n\kappa_{j,\gamma}}_{p,p}(\R^{d-1};X)$ are related by $\mathrm{tr}_{y=0}(\mathscr{F}_{t}\mathcal{E}_{j}g_{j})(\theta) = (\mathscr{F}_{t}g_{j})(\theta)$; just note that
$(\mathscr{F}_{t}\mathcal{E}_{j}g_{j})(0,x',\theta) = (\mathscr{F}_{t}g_{j})(x',\theta)$ for every $x' \in \R^{d-1}$.
Therefore, by Proposition~\ref{PIBVP:prop:opl_elliptisch_probleem_randwaarden},
$v(\theta) = (\mathscr{F}_{t}u)(\theta)  =(\mathscr{F}_{t}\overline{\mathscr{S}}g)(\theta) = \sum_{j=1}^{n}\tilde{\mathcal{S}}_{j}(1+\imath\theta)
(\mathscr{F}_{t}\mathcal{E}_{j}g_{j})(\theta) \in W^{2n}_{p}(\R^{d}_{+},w_{\gamma};X)$ is the unique solution of the problem
\[
\begin{array}{rll}
(1+\imath\theta) v + \mathcal{A}(D)v &= 0, & \\
\mathcal{B}_{j}(D)v &= (\mathscr{F}_{t}g_{j})(\theta), & j=1,\ldots,n.
\end{array}
\]
Applying the inverse Fourier transform $\mathscr{F}_{t}^{-1}$ with respect to $\theta$, we find
\[
\begin{array}{rll}
\partial_{t}u + (1+\mathcal{A}(D))u &= 0, & \\
\mathcal{B}_{j}(D)u &= g_{j}, & j=1,\ldots,n. \\
\end{array}
\]

(III) We next derive a representation formula for $\overline{\mathscr{S}}$ that is well suited for proving the boundedness of $\overline{\mathscr{S}}$. To this end, fix a $g=(g_{1},\ldots,g_{n}) \in V^{n}$. Then we have, for each multi-index $\alpha \in \N^{d}, |\alpha| \leq 2n$,
\begin{eqnarray}
D^{\alpha}\overline{\mathscr{S}}g
&=& D^{\alpha} \mathscr{F}_{t}^{-1}\left[ \theta \mapsto \sum_{j=1}^{n}\tilde{\mathcal{S}}_{j}(1+\imath\theta)
(\mathscr{F}_{t}\mathcal{E}_{j}g_{j})(\theta) \right] \nonumber \\
&=& \sum_{j=1}^{n}\mathscr{F}_{t}^{-1}\left[ \theta \mapsto D^{\alpha}\tilde{\mathcal{S}}_{j}(1+\imath\theta)
(\mathscr{F}_{t}\mathcal{E}_{j}g_{j})(\theta) \right] \nonumber \\
&\stackrel{\eqref{PIBVP:eq:prop;opl_elliptisch_probleem_randwaarden;rep_formula}}{=}&
\sum_{j=1}^{n}\mathscr{F}_{t}^{-1}\left[ \theta \mapsto
\mathcal{T}^{1}_{j,\alpha}(1+\imath\theta)L_{1+\imath\theta}^{1-\frac{n_{j}}{2n}}(\mathscr{F}_{t}\mathcal{E}_{j}g_{j})(\theta) + \mathcal{T}^{2}_{j,\alpha}(1+\imath\theta)L_{1+\imath\theta}^{1-\frac{n_{j}+1}{2n}}D_{y}(\mathscr{F}_{t}\mathcal{E}_{j}g_{j})(\theta) \right] \nonumber \\
&=& \sum_{j=1}^{n}\mathscr{F}_{t}^{-1}\left[ \theta \mapsto
\mathcal{T}^{1}_{j,\alpha}(1+\imath\theta)L_{1+\imath\theta}^{1-\frac{n_{j}}{2n}}(\mathscr{F}_{t}\mathcal{E}_{j}g_{j})(\theta)\right]\nonumber \\
&& \quad + \:\:\: \sum_{j=1}^{n}\mathscr{F}_{t}^{-1}\left[\theta \mapsto \mathcal{T}^{2}_{j,\alpha}(1+\imath\theta)L_{1+\imath\theta}^{1-\frac{n_{j}+1}{2n}}(\mathscr{F}_{t}D_{y}\mathcal{E}_{j}g_{j})(\theta) \right]  \nonumber \\
&\stackrel{\eqref{pibvp:eq:omwisselen_L}}{=}& \sum_{j=1}^{n}\mathscr{F}_{t}^{-1}\left[ \theta \mapsto
\mathcal{T}^{1}_{j,\alpha}(1+\imath\theta)(\mathscr{F}_{t}L^{1-\frac{n_{j}}{2n}}\mathcal{E}_{j}g_{j})(\theta)\right] \nonumber \\
&& \quad + \:\:\: \sum_{j=1}^{n}\mathscr{F}_{t}^{-1}\left[\theta \mapsto \mathcal{T}^{2}_{j,\alpha}(1+\imath\theta)(\mathscr{F}_{t}L^{1-\frac{n_{j}+1}{2n}}D_{y}\mathcal{E}_{j}g_{j})(\theta) \right].\label{PIBVP:eq:lemma;pbvp_sol_for_trivial_boundary_data;abstract_max-reg;rep_formula_sol_operator}
\end{eqnarray}

(IV) We next show that $\norm{\overline{\mathscr{S}}g}_{\overline{\U}} \lesssim \norm{g}_{\overline{\G}}$ for $g \in V^{n}$.
Being a solution of \eqref{PIBVP:eq:pbvp_sol_for_just_boundary_data;real_line}, $\overline{\mathscr{S}}g$ satisfies
\[
\partial_{t}\overline{\mathscr{S}}g = -(1+\mathcal{A}(D))\overline{\mathscr{S}}g.
\]
Hence, it suffices to establish the estimate $\norm{D^{\alpha}\overline{\mathscr{S}}g}_{\overline{\F}} \lesssim \norm{g}_{\overline{\G}}$ for all multi-indices $\alpha \in \N^{d}, |\alpha| \leq 2n$.
So fix such an $|\alpha| \leq 2n$. Then, in view of the representation formula
\eqref{PIBVP:eq:lemma;pbvp_sol_for_trivial_boundary_data;abstract_max-reg;rep_formula_sol_operator}, it is enough to show that
\begin{equation}\label{PIBVP:eq:lemma;pbvp_sol_for_trivial_boundary_data;abstract_max-reg;final_estimate;1}
\norm{\mathscr{F}_{t}^{-1}\left[\theta \mapsto \mathcal{T}^{1}_{j,\alpha}(1+\imath\theta)(\mathscr{F}_{t}L^{1-\frac{n_{j}}{2n}}\mathcal{E}_{j}g_{j})(\theta) \right]\,}_{\overline{\F}} \lesssim \norm{g}_{\overline{\G}}, \quad j=1,\ldots,n,
\end{equation}
and
\begin{equation}\label{PIBVP:eq:lemma;pbvp_sol_for_trivial_boundary_data;abstract_max-reg;final_estimate;2}
\norm{\mathscr{F}_{t}^{-1}\left[\theta \mapsto \mathcal{T}^{2}_{j,\alpha}(1+\imath\theta)(\mathscr{F}_{t}L^{1-\frac{n_{j}+1}{2n}}D_{y}\mathcal{E}_{j}g_{j})(\theta) \right]\,}_{\overline{\F}} \lesssim \norm{g}_{\overline{\G}}, \quad j=1,\ldots,n.
\end{equation}
We only treat the estimate \eqref{PIBVP:eq:lemma;pbvp_sol_for_trivial_boundary_data;abstract_max-reg;final_estimate;2},
the estimate \eqref{PIBVP:eq:lemma;pbvp_sol_for_trivial_boundary_data;abstract_max-reg;final_estimate;1} being similar (but easier):
Fix a $j \in \{1,\ldots,n\}$.
For the full $(d+1)$-dimensional Euclidean space $\R^{d} \times \R$ instead of $\R^{d}_{+} \times \R$,
\[
D_{y} \in \mathcal{B}\left(H^{1-\frac{n_{j}}{2n},(\frac{1}{2n},1)}_{(p,q),(d,1)}(\R^{d}_{+} \times \R,(w_{\gamma},v);X),H^{1-\frac{n_{j}+1}{2n},(\frac{1}{2n},1)}_{(p,q),(d,1)}(\R^{d}_{+} \times \R,(w_{\gamma},v);X)\right).
\]
follows from \eqref{PIBVP:eq:prelim:differential} (and the fact that $L_{(p,q),(d,1)}(\R^{d+1},(w_{\gamma},v_{\mu});X)$ is an admissibile Banach space of $X$-valued tempered distributions on $\R^{d+1}$ in view of \eqref{PIBVP:eq:prelim:anisotrope_mixed-norm_fm}), from which the $\R^{d}_{+} \times \R$-case follows by restriction.
In combination with \eqref{PIBVP:eq:lemma;pbvp_sol_for_just_boundary_date;extension_operator} and Lemma \ref{PIBVP:lemma;lemma:pbvp_sol_for_just_boundary_data;isomorphism_between_BP} this yields
\begin{equation}\label{PIBVP:eq:lemma;pbvp_sol_for_trivial_boundary_data;abstract_max-reg;final_estimate;2;first_part}
L^{1-\frac{n_{j}+1}{2n}}D_{y}\mathcal{E}_{j} \in \mathcal{B}\left( \overline{\G}_{j},
\underbrace{H^{0,(\frac{1}{2n},1)}_{(p,q),(d-1,1)}(\R^{d-1} \times \R,(1,v);L^{p}(\R_{+},|\,\cdot\,|^{\gamma};X))}_{=\,L^{q}(\R,v;L^{p}(\R^{d}_{+},w_{\gamma};X)) \,=\, \overline{\F}} \right).
\end{equation}
Furthermore, we have that $\mathcal{T}^{2}_{j,\alpha}(1+\imath\cdot) \in C^{\infty}(\R;\mathcal{B}(L^{p}(\R^{d}_{+},w_{\gamma};X)))$ satisfies
\[
\mathcal{R}\left\{ \theta^{k}\partial_{\theta}^{k}\mathcal{T}^{2}_{j,\alpha}(1+\imath\theta) :\theta \in \R \right\} \leq
\mathcal{R}\left\{ (1+\imath\theta)^{k+1-\frac{|\alpha|}{2n}}\partial_{\theta}^{k}\mathcal{T}^{2}_{j,\alpha}(1+\imath\theta) : \theta \in \R \right\}
< \infty, \quad\quad k \in \N,
\]
by the Kahane contraction principle and \eqref{PIBVP:eq:prop;opl_elliptisch_probleem_randwaarden;R-bounds}; in particular,
$\mathcal{T}^{2}_{j,\alpha}(1+\imath\cdot)$ satisfies the Mikhlin condition corresponding to \eqref{PIBVP:eq:prelim:operator-valued_FM}.
As a consequence,
$\mathcal{T}^{2}_{j,\alpha}(1+\imath\cdot)$ defines a bounded Fourier multiplier operator on $L^{q}(\R,v;L^{p}(\R^{d}_{+},w_{\gamma};X))$.
In combination with \eqref{PIBVP:eq:lemma;pbvp_sol_for_trivial_boundary_data;abstract_max-reg;final_estimate;2;first_part}, this gives the estimate \eqref{PIBVP:eq:lemma;pbvp_sol_for_trivial_boundary_data;abstract_max-reg;final_estimate;2}.

(V) We finally show that $\overline{\mathscr{S}} \in \mathcal{B}(\overline{\G},\overline{\U})$ maps ${_{0}}\overline{\G}$ to ${_{0}}\overline{\U}$.
As in the proof of \cite[Lemma~2.2.7]{Mey_PHD-thesis} it can be shown that, if
\begin{equation*}\label{pbvp:eq:prep;lemma;pbvp_sol_for_just_boundary_data;problem_without_initial_data;space_Fourier_transform_time_g}
g = (g_{1},\ldots,g_{n})\in \prod_{j=1}^{n}C_{L^{1}}(\R;F_{p,p}^{2n\kappa_{j,\gamma}}(\R^{d-1};X))
\quad\mbox{with}\quad g_{1}(0)=\ldots=g_{n}(0)=0
\end{equation*}
and
\begin{equation*}\label{pbvp:eq:prep;lemma;pbvp_sol_for_just_boundary_data;problem_without_initial_data;space_Fourier_transform_time}
u \in C^{1}_{L^{1}}(\R;L^{p}(\R^{d}_{+},w_{\gamma};X)) \cap C_{L^{1}}(\R;W^{2n}_{p}(\R^{d}_{+},w_{\gamma};X))
\end{equation*}
satisfy \eqref{PIBVP:eq:pbvp_sol_for_just_boundary_data;real_line}, then $u(0)=0$.
The desired statement thus follows from Lemma~\ref{PIBVP:lemma:lemma:pbvp_sol_for_just_boundary_data;density_trace}.
\end{proof}

\begin{cor}\label{PIBVP:lemma:pbvp_sol_for_just_boundary_data}
Let the notations and assumptions be as in Theorem~\ref{PIBVP:thm:main_result}, but for the model problem case of top order constant coefficients on the half-space considered in Section~\ref{PIBVP:sec:elliptic_problems}.
Then the problem \eqref{PIBVP:eq:pbvp_sol_for_just_boundary_data} admits a bounded linear solution operator
\[
\mathscr{S}: \{g : (0,g,0) \in \D^{p,q}_{\gamma,\mu}\} \longra \U^{p,q}_{\gamma,\mu}.
\]
\end{cor}

\subsection{Proof of Theorem~\ref{PIBVP:thm:main_result}}

We can now finally prove the main result of this paper.
\begin{proof}[Proof of Theorem \ref{PIBVP:thm:main_result}]
In view of Section~\ref{PIBVP:subsection:space_initial-boundary_data}, it remains to establish existence and uniqueness of a solution $u \in \U^{p,q}_{\gamma,\mu}$ of \eqref{PIBVP:eq:rigorous_problem} for given $(f,g,u_{0}) \in \G^{p,q}_{\gamma,\mu} \oplus \D^{p,q}_{\gamma,\mu}$.
By a standard (but quite technical) perturbation and localization procedure, it is enough to consider the model problem
\[
\begin{array}{rll}
\partial_{t}u + (1+\mathcal{A}(D))u &= f,  \\
\mathcal{B}_{j}(D)u &= g_{j}, & j=1,\ldots,n, \\
u(0) &= u_{0},
\end{array}
\]
on the half-space, where $\mathcal{A}$ and $\mathcal{B}_{1},\ldots,\mathcal{B}_{n}$ are top-order constant coefficient operators as considered in Section~\ref{PIBVP:sec:elliptic_problems}. This procedure is worked out in full detail in \cite{Mey_PHD-thesis}; for further comments we refer to Appendix~\ref{PIBVP:appendix:sec:localization}.

Let $(f,g,u_{0}) \in \F^{p,q}_{\gamma,\mu} \oplus \D^{p,q}_{\gamma,\mu}$.
In view of Theorem~\ref{PIBVP:thm:traces_main_result} and the fact that $\mathrm{tr}_{t=0} \circ \mathcal{B}_{j}(D) = \mathcal{B}_{j}(D)$ on $\U^{p,q}_{\gamma,\mu} \circ \mathrm{tr}_{t=0}$ when $\kappa_{j,\gamma} < \frac{1+\mu}{q}$, we may without loss of generality assume that $u_{0}=0$.
By Corollary~\ref{PIBVP:lemma:pbvp_sol_for_just_boundary_data} we may furthermore assume that $g=0$.
Defining $A_{B}$ as the operator on $Y = L^{p}(\R^{d}_{+},w_{\gamma})$ with domain
\[
D(A_{B}) := \{ u \in W^{2n}_{p}(\R^{d}_{+},w_{\gamma}) : \mathcal{B}_{j}(D)v=0, j=1,\ldots,n \}
\]
and given by the rule $A_{B}v :=\mathcal{A}(D)v$, we need to show that $1+A_{B}$ enjoys the property of maximal $L^{q}_{\mu}$-regularity: for every $f \in L^{q}(\R_{+},v_{\mu};Y)$ there exists a unique $u \in {_{0}}W^{1}_{q}(\R_{+},v_{\mu};Y) \cap L^{q}(\R_{+},v_{\mu};D(A_{B}))$ with $u'+(1+A_{B})u=f$.
In the same way as in \cite[Theorem~7.4]{DHP1} it can be shown that $A_{B} \in \mathcal{H}^{\infty}(Y)$ with angle $\phi^{\infty}_{A_{B}} < \frac{\pi}{2}$.
As $Y$ is a UMD space, $1+A_{B}$ enjoys maximal $L^{q}_{\mu}$-regularity for $\mu=0$; see e.g.\ \cite[Section~4.4]{Weis_Survey_H-fc} and the references therein.
By \cite{Chill&Fiorenza_extrapolation_max-reg,pruss_simonett} this extrapolates to all $\mu \in (-1,q-1)$ (i.e.\ all $\mu$ for which $v_{\mu} \in A_{q}$).
\end{proof}

\appendix

\section{Series Estimates in Triebel-Lizorkin and Besov Spaces}\label{PIBVP:appendix:series_estimates}

\begin{lemma}\label{PIBVP:lemma:appendix:prop;trace_TL_conv_series}
Let $X$ be a Banach space, $\vec{a} \in (0,\infty)^{l}$, $\vec{p} \in [1,\infty)^{l}$, $q \in [1,\infty]$,
$s > 0$, and $\vec{w} \in \prod_{j=1}^{l}A_{\infty}(\R^{\mathpzc{d}_{j}})$.
Suppose that there exists an $\vec{r} \in (0,1)^{l}$ such that
$s> \sum_{j=1}^{l}a_{j}\mathpzc{d}_{j}(\frac{1}{r_{j}}-1)$ and $\vec{w} \in \prod_{j=1}^{l}A_{p_{j}/r_{j}}(\R^{\mathpzc{d}_{j}})$.
Then, for every $c>0$, there exists a constant $C>0$ such that, for all $(f_{k})_{k \in \N} \subset \mathcal{S}'(\R^{d};X)$ satisfying
$\supp \hat{f_{k}} \subset \prod_{j=1}^{l}[-c2^{ka_{j}},-c2^{ka_{j}}]^{\mathpzc{d}_{j}}$ and
\[
(2^{ks}f_{k})_{k \geq 0} \in L^{\vec{p},\mathpzc{d}}(\R^{d},\vec{w})[\ell^{q}(\N)](X)
\]
it holds that $\sum_{k \in \N}f_{k}$ defines a convergent series in $\mathcal{S}'(\R^{d};X)$ with limit $f \in F^{s,\vec{a}}_{\vec{p},q,\mathpzc{d}}(\R^{d},\vec{w};X)$ of norm $\leq C\norm{(2^{ks}f_{k})_{k \geq 0}}_{L^{\vec{p},\mathpzc{d}}(\R^{d},\vec{w})[\ell^{q}(\N)](X)}$.
\end{lemma}
\begin{proof}
This can be proved in the same way as \cite[Lemma~3.19]{JS_traces}, using Lemma~\ref{PIBVP:lemma:appendix:prop:ineq_needed_for_series_conv_TL} below instead of \cite[Proposition~3.14]{JS_traces}.
For more details we refer to \cite[Lemma~5.2.22]{Lindemulder_master-thesis}.
\end{proof}

\begin{lemma}\label{PIBVP:lemma:appendix:conv_series_dyadic_corona;TL}
Let $X$ be a Banach space, $\vec{a} \in (0,\infty)^{l}$, $\vec{p} \in [1,\infty)^{l}$, $q \in [1,\infty]$, $s \in \R$, and $\vec{w} \in \prod_{j=1}^{l}A_{\infty}(\R^{\mathpzc{d}_{j}})$.
For every $c>1$ there exists a constant $C>0$ such that, for all $(f_{k})_{k \in \N} \subset \mathcal{S}'(\R^{d};X)$ satisfying
\begin{equation}\label{functieruimten:eq:lemma;conv_series_dyadic_corona;Fourier_supports;TL}
\supp\hat{f}_{0} \subset \{\xi \in \R^{d} : |\xi|_{\mathpzc{d},\vec{a}} \leq c\}, \quad\quad \supp\hat{f}_{k} \subset \{\xi \in \R^{d} : c^{-1}2^{k} \leq |\xi|_{\mathpzc{d},\vec{a}} \leq c2^{k}\}  \:\:\: (k \geq 1),
\end{equation}
and
\[
(2^{ks}f_{k})_{k \geq 0} \in L^{\vec{p},\mathpzc{d}}(\R^{d},\vec{w})[\ell^{q}(\N)](X)
\]
it holds that $\sum_{k \in \N}f_{k}$ defines a convergent series in $\mathcal{S}'(\R^{d};X)$ with limit $f \in F^{s,\vec{a}}_{\vec{p},q,\mathpzc{d}}(\R^{d},\vec{w};X)$ of norm $\leq C\norm{(2^{ks}f_{k})_{k \geq 0}}_{L^{\vec{p},\mathpzc{d}}(\R^{d},\vec{w})[\ell^{q}(\N)](X)}$.
\end{lemma}
\begin{proof}
This can be proved in the same way as \cite[Lemma~3.20]{JS_traces}. In fact, one only needs a minor modification of the proof of Lemma~\ref{PIBVP:lemma:appendix:prop;trace_TL_conv_series}.
\end{proof}

\begin{lemma}\label{PIBVP:lemma:appendix:conv_series_dyadic_corona;Besov}
 Let $X$ be a Banach space, $\vec{a} \in (0,\infty)^{l}$, $\vec{p} \in [1,\infty)^{l}$, $q \in [1,\infty]$, $s \in \R$, and $\vec{w} \in \prod_{j=1}^{l}A_{\infty}(\R^{\mathpzc{d}_{j}})$.
For every $c>1$ there exists a constant $C>0$ such that, for all $(f_{k})_{k \in \N} \subset \mathcal{S}'(\R^{d};X)$ satisfying \eqref{functieruimten:eq:lemma;conv_series_dyadic_corona;Fourier_supports;TL} and
\[
(2^{s k}f_{k})_{k \geq 0} \in L^{\vec{p},\mathpzc{d}}(\R^{d},\vec{w})[\ell^{q}(\N)](X)
\]
it holds that $\sum_{k \in \N}f_{k}$ defines a convergent series in $\mathcal{S}'(\R^{d};X)$ with limit $f \in F^{s,\vec{a}}_{\vec{p},q,\mathpzc{d}}(\R^{d},\vec{w};X)$ of norm $\leq C\norm{(2^{s k}f_{k})_{k \geq 0}}_{L^{\vec{p},\mathpzc{d}}(\R^{d},\vec{w})[\ell^{q}(\N)](X)}$.
\end{lemma}

The above two lemmas are through Lemma~\ref{PIBVP:lemma:appendix:prop:ineq_needed_for_series_conv_TL} based on the following maximal inequality:

\begin{lemma}\label{PIBVP:lemma:appendix:part_HL-max-op_weighted_mixed_norm_space}
Let $\vec{a} \in (0,\infty)^{l}$ and $\vec{w} \in \prod_{j=1}^{l}\mathcal{W}(\R^{\mathpzc{d}_{j}})$.
Let $j_{0} \in \{1,\ldots,l\}$ and $r_{j_{0}} \in (0,\min\{p_{j_{0}},\ldots,p_{l}\})$ be such that $w_{j_{0}} \in
A_{p_{j_{0}}/r_{j_{0}}}(\R^{\mathpzc{d}_{j_{0}}})$.
Then
\[
M_{[\mathpzc{d};j_{0}],r_{j_{0}}}(f)(x) := \sup_{\delta > 0} \left(
\fint_{B(x_{j_{0}},\delta)}|f(x_{1},\ldots,x_{j_{0}-1},y,x_{j_{0}+1},\ldots,x_{l})|^{r_{j_{0}}}\,dy \right)^{1/r_{j_{0}}}, \quad\quad x \in \R^{d},
\]
gives rise to a well-defined bounded sublinear operator $M_{[\mathpzc{d};j_{0}],r_{j_{0}}}$ on $L^{\vec{p},\mathpzc{d}}(\R^{d},\vec{w})$. Moreover, there holds a
Fefferman-Stein inequality for $M_{[\mathpzc{d};j_{0}],r_{j_{0}}}$:
for every $q \in (\max\{1,r\},\infty]$ there exists a constant $C \in (0,\infty)$ such that, for all sequences $(f_{i})_{i \in \Z} \subset
L^{p,\mathpzc{d}}(\R^{d},w)$,
\[
\norm{\norm{(M_{[\mathpzc{d};j_{0}],r_{j_{0}}}(f_{i}))_{i \in \Z}}_{\ell^{q}(\Z)}}_{L^{p,\mathpzc{d}}(\R^{d},w)} \leq
C\norm{\norm{(f_{i})_{i \in \Z}}_{\ell^{q}(\Z)}}_{L^{p,\mathpzc{d}}(\R^{d},w)}.
\]
\end{lemma}
\begin{proof}
This can be easily derived from \cite[Theorem~2.6]{Gallarati&Lorist&Veraar}, which is a weighted version of the special case of the $L^{p}$-boundedness of the Banach lattice version of the Hardy-Littlewood maximal function \cite{Bourgain,H.L.-property,de_Francia_max,Toz2} for mixed-norm spaces (also see \cite[Remark~2.7]{Gallarati&Lorist&Veraar}).
\end{proof}

\begin{lemma}\label{PIBVP:lemma:appendix:prop:ineq_needed_for_series_conv_TL}
Let $X$ be a Banach space, $\vec{p} \in [1,\infty)^{l}$, $q \in [1,\infty]$, and $\vec{w} \in \prod_{j=1}^{l}A_{\infty}(\R^{\mathpzc{d}_{j}})$.
Suppose $\vec{r} \in (0,1)^{l}$ is such that $w_{j} \in A_{p_{j}/r_{j}}(\R^{\mathpzc{d}_{j}})$ for $j=1,\ldots,l$.
Let $\psi \in \mathcal{S}(\R^{d})$ be such that $\supp \hat{\psi} \subset \{ \xi \in \R^{d} \mid |\xi|_{\mathpzc{d},\vec{a}} \leq  2 \}$, and
set $\psi_{n}:=\psi(\delta^{[\mathpzc{d},\vec{a}]}_{2^{n}}\,\cdot\,)$ for each $n \in \N$. Then there exists a constant $C>0$ such that, for all
$(f_{n})_{n \in \N} \subset \mathcal{S}'(\R^{d};X)$ with $\supp \hat{f}_{n} \subset \prod_{j=1}^{l}[-R2^{na_{j}},R2^{na_{j}}]^{\mathpzc{d}_{j}}$ for
some $R \geq 1$, the following inequality holds true:
\[
\norm{(\psi_{n}*f_{n})_{n \geq 0}}_{L^{\vec{p},\mathpzc{d}}(\R^{d},\vec{w};\ell^{q}(\N;X))} \leq
CR^{\sum_{j=1}^{l}a_{j}\mathpzc{d}_{j}(\frac{1}{r_{j}}-1)}\norm{(f_{n})_{n \geq 0}}_{L^{\vec{p},\mathpzc{d}}(\R^{d},\vec{w};\ell^{q}(\N;X))}.
\]
\end{lemma}
\begin{proof}
As in the proof of \cite[Proposition~3.14]{JS_traces}, it can be shown that
\[
\norm{(\psi_{n}*f_{n})(x)}_{X} \leq c R^{\sum_{j=1}^{m}a_{j}\mathpzc{d}_{j}(\frac{1}{r_{j}}-1)}\left[M_{[\mathpzc{d};l],r_{l}}(\ldots
M_{[\mathpzc{d};1],r_{1}}(\norm{f_{n}}_{X})\ldots)\right](x), \quad\quad n \in \N, x \in \R^{d},
\]
for some constant $c>0$ independent of $(f_{n})_{n}$. The desired result now follows from Lemma~\ref{PIBVP:lemma:appendix:part_HL-max-op_weighted_mixed_norm_space}.
\end{proof}

Given a function $f:\R^{d} \longra X$, $\vec{r} \in (0,\infty)^{l}$ and $\vec{b} \in (0,\infty)^{l}$, we define the \emph{maximal function of
Peetre-Fefferman-Stein type} $f^{*}(\vec{r},\vec{b},\mathpzc{d};\,\cdot\,)$ by
\begin{equation}\label{PIBVP:eq:appendix:max_funct_Peetre-Fefferman-Stein_type}
f^{*}(\vec{r},\vec{b},\mathpzc{d};x):= \sup_{z \in \R^{d}}\frac{\norm{f(x-z)}_{X}}{(1+|b_{1}z_{1}|^{\mathpzc{d}_{1}/r_{1}}) \ldots
(1+|b_{l}z_{l}|^{\mathpzc{d}_{l}/r_{l}})}, \quad\quad x \in \R^{d}.
\end{equation}

\begin{lemma}\label{PBIBV:lemma:appendix:Peetre-Fefferman-Stein_maximal_ineq}
Let $X$ be a Banach space, $\vec{p} \in [1,\infty)^{l}$, $q \in [1,\infty]$, and $\vec{w} \in \prod_{j=1}^{l}A_{\infty}(\R^{\mathpzc{d}_{j}})$.
Let $\vec{r} \in (0,1)^{l}$ be such that $w_{j} \in A_{p_{j}/r_{j}}(\R^{\mathpzc{d}_{j}})$ for $j=1,\ldots,l$. Then there exists a constant $C>0$ such that,
for all $(f_{n})_{n \in \N} \subset \mathcal{S}'(\R^{d};X)$
and $(\vec{b}^{[n]})_{n \in \N} \subset (0,\infty)^{l}$ with $\supp \hat{f} \subset \prod_{j=1}^{l}[-b^{[n]}_{j},b^{[n]}_{j}]^{\mathpzc{d}_{j}}$ for all $n \in \N$, we have the inequality
\[
\norm{ (f_{n}^{*}(\vec{r},\vec{b}^{[n]},\mathpzc{d};\,\cdot\,))_{n \geq 0} }_{L^{\vec{p},\mathpzc{d}}(\R^{d},\vec{w};\ell^{q}(\N)} \leq
C\norm{(f_{n})_{n}}_{L^{\vec{p},\mathpzc{d}}(\R^{d},\vec{w};\ell^{q}(\N;X))}.
\]
\end{lemma}
\begin{proof}
As in the proof of \cite[Proposition~3.12]{JS_traces},
it can be shown that
\[
f^{*}_{n}(\vec{r},\vec{b},\mathpzc{d};x) \leq c \left[M_{[\mathpzc{d};l],r_{l}}(\ldots M_{[\mathpzc{d};1],r_{1}}(\norm{f_{n}}_{X})\ldots)\right](x), \quad\quad n
\in \N, x \in \R^{d}
\]
for some constant $c>0$ only depending on $\vec{r}$.
The desired result now follows from Lemma~\ref{PIBVP:lemma:appendix:part_HL-max-op_weighted_mixed_norm_space}.
\end{proof}

\section{Comments on the localization and perturbation procedure}\label{PIBVP:appendix:sec:localization}

As already mentioned in the proof of Theorem \ref{PIBVP:thm:main_result},
the localization and perturbation procedure for reducing to the model problem case on $\R^{d}_{+}$ is worked out in full detail in \cite{Mey_PHD-thesis}.
However, there only the case $q=p$ with temporal weights having a positive power is considered.
For some of the estimates used there (parts) of the proofs do not longer work in our setting, where the main difficulty comes from $q \neq p$.
It is the goal of this appendix to consider these estimates.

\subsection*{Top order coefficients having small oscillations}

The most crucial part in the localization and perturbation procedure where we need to take care of the estimates is \cite[Proposition~2.3.1]{Mey_PHD-thesis} on top order coefficients having small oscillations.
To be more specific, we only consider the estimates in Step (IV) of its proof.

Before we go to these estimates, let us start with the lemma that makes it possible to reduce to the situation of top order coefficients having small oscillations.

\begin{lemma}\label{PIBVP:appendix:lemma:BUC_top_order_coeff}
Let $X$ be a Banach space, $J \subset \R$ and interval, $\mathscr{O} \subset \R^{d}$ a domain with compact boundary $\partial\mathscr{O}$, $\kappa \in \R$, $n \in \N_{>0}$, $s,r \in (1,\infty)$ and $p \in [1,\infty]$.
If $\kappa > \frac{1}{s}+\frac{d-1}{2nr}$, then
\[
F^{\kappa}_{s,p}(J;L^{r}(\partial\mathscr{O};X)) \cap L^{s}(J;B^{2n\kappa}_{r,p}(\partial\mathscr{O};X))
\hookrightarrow BUC(\partial\mathscr{O} \times J;X).
\]
\end{lemma}
\begin{proof}
By a standard localization procedure, we may restrict ourselves to the case that $J=\R$ and $\mathscr{O}=\R^{d}_{+}$ (so that $\partial\mathscr{O} = \R^{d}_{+}$).
By \cite{Lindemulder_Intersection},
\begin{equation}\label{PIBVP:eq:appendix:sec:localization;int_rep_bound_coeff}
F^{\kappa}_{s,p}(\R;L^{r}(\R^{d-1};X)) \cap L^{s}(\R;B^{2n\kappa}_{r,p}(\R^{d-1};X))
= \{ f \in \mathcal{S}'(\R^{d-1}\times\R;X) : (S_{n}f)_{n} \in L^{s}(\R)[[\ell^{p}_{\kappa}(\N)]L^{r}(\R^{d-1})](X) \}
\end{equation}
with equivalence of norms, where $(S_{n})_{n \in \N}$ correspond to some fixed choice of $\varphi \in \Phi^{(d-1,1),(\frac{1}{2n},1)}(\R^{d})$.
For $\epsilon > 0$ we thus obtain
\[
F^{\kappa}_{s,p}(\R;L^{r}(\R^{d-1};X)) \cap L^{s}(\R;B^{2n\kappa}_{r,p}(\R^{d-1};X))
\hookrightarrow
B^{\kappa-\epsilon,(\frac{1}{2n},1)}_{(r,s),s,(d-1,1)}(\R^{d};X).
\]
Choosing $\tilde{\kappa}$ with $\kappa > \tilde{\kappa} > \frac{1}{s}+\frac{d-1}{2nr}$, the desired inclusion follows from Corollary~\ref{PIBVP:cor:thm:trace_Besov:distr_trace}.
\end{proof}

\begin{lemma}\label{PIBVP:lemma:appendix:sec:localization;ext_op}
Let $X$ be a Banach space, $i \in \{1,\ldots,l\}$, $T \in \R$ and
\[
\R^{d}_{[\mathpzc{d};i],T} := \{ x \in \R^{d} : x_{i,1} < T \} = \iota_{[\mathpzc{d};i]}[(-\infty,T) \times \R^{\mathpzc{d}_{i}-1}].
\]
Then there exists an extension operator $\mathcal{E}_{[\mathpzc{d};i],T} \in \mathcal{L}(\mathcal{S}'(\R^{d}_{[\mathpzc{d};i],T};X),\mathcal{S}'(\R^{d};X)$ which, for every
$\vec{a} \in (0,\infty)^{l}$, $s \in \R$, $\vec{p} \in [1,\infty)^{l}$, $q \in [1,\infty]$ and $\vec{w} \in A_{\infty}(\R^{\mathpzc{d}_{j}})$, restricts to a bounded linear operator from $F^{s,\vec{a}}_{\vec{p},q,\mathpzc{d}}(\R^{d}_{[\mathpzc{d};i],T},\vec{w};X)$ to $F^{s,\vec{a}}_{\vec{p},q,\mathpzc{d}}(\R^{d},\vec{w};X)$ whose operator norm can be estimated by a constant independent of $X$ and $T$.
\end{lemma}
\begin{proof}
This can be shown in the same way as in \cite{Rychkov1999_restrictions&extensions}.
\end{proof}

\begin{lemma}\label{PIBVP:lemma:appendix:sec:localization;pre-variant_(2.3.13)}
Let $X$ be a Banach space, $I=(-\infty,T)$ with $T \in (-\infty,\infty]$, $\kappa > 0$, $n \in \N_{>0}$, $p,q \in [1,\infty)$, $r,u \in (p,\infty)$, $s,v \in (q,\infty)$ with $\frac{1}{p}=\frac{1}{r}+\frac{1}{u}$ and $\frac{1}{q}=\frac{1}{s}+\frac{1}{v}$.
Let $\mu \in (-1,\infty)$ be such that $\frac{v}{q}\mu \in (-1,v-1)$.
Then
\begin{eqnarray*}
\norm{fg}_{F^{\kappa,(\frac{1}{2n},1)}_{(p,q),p,(d-1,1)}(\R^{d-1}\times I,(1,v_{\mu});X)}
&\lesssim& \norm{f}_{L^{\infty}(\R^{d-1}\times I;\mathcal{B}(X))} \norm{g}_{F^{\kappa,(\frac{1}{2n},1)}_{(p,q),p,(d-1,1)}(\R^{d-1}\times I,(1,v_{\mu});X)} \\
&& + \norm{f}_{F^{\kappa}_{s,p}(I;L^{r}(\R^{d-1};\mathcal{B}(X))) \cap L^{s}(I;B^{2n\kappa}_{r,p}(\R^{d-1};\mathcal{B}(X)))} \norm{g}_{F^{0,(\frac{1}{2n},1)}_{(u,v),1,(d-1,1)}(\R^{d-1}\times I,(1,v_{\frac{v}{q}\mu});X)}
\end{eqnarray*}
with implicit constant independent of $X$ and $T$.
\end{lemma}
Note here that $\frac{v}{q}\mu < v-1$ when $\mu < q-1$.

\begin{proof}
Extending $f$ from $\R^{d-1} \times I$ to $\R^{d-1} \times \R$ by using an extension operator of Fichtenholz type and extending $g$ from $\R^{d-1} \times I$  to $\R^{d-1} \times \R$ by using an extension operator as in Lemma~\ref{PIBVP:lemma:appendix:sec:localization;ext_op}, we may restrict ourselves to the case $I=\R$.

Let $(S_{n})_{n \in \N}$ correspond to some fixed choice of $\varphi \in \Phi^{(d-1,1),(\frac{1}{2n},1)}(\R^{d})$, say with $A=1$ and $B=\frac{3}{2}$.
As in \cite[Chapter~4]{RS1996} (the isotropic case) we can use paraproducts associated with $(S_{n})_{n \in \N}$ in order to treat the pointwise product $fg$.
For this it is convenient to define $S^{k} \in \mathcal{L}(\mathcal{S}'(\R^{d};X))$ by $S^{k}:=\sum_{n=0}^{k}S_{n}$.
Given $f \in L^{\infty}(\R^{d};\mathcal{B}(X))$ and $g \in F^{\kappa,(\frac{1}{2n},1)}_{(p,q),p,(d-1,1)}(\R^{d},(1,v_{\mu});X) \hookrightarrow L^{(p,q),(d-1,1)}(\R^{d},(1,v_{\mu});X)$, if the paraproducts
\[
\Pi_{1}(f,g) := \sum_{k=2}^{\infty}(S^{k-2}f)(S_{k}g),
\Pi_{2}(f,g) := \sum_{k=0}^{\infty}\sum_{j=-1}^{1}(S_{k+j}f)(S_{k}g),
\Pi_{3}(f,g) := \sum_{k=2}^{\infty}(S_{k}f)(S^{k-2}g),
\]
exist (as convergent series) in $\mathcal{S}'(\R^{d};X)$, then
\[
fg = \Pi_{1}(f,g) + \Pi_{2}(f,g) + \Pi_{3}(f,g).
\]
Here the Fourier supports of the summands in the paraproducts satisfy
\begin{align*}
 \supp \mathscr{F}[(S^{k-2}f)(S_{k}g)] &\subset \{ \xi : 2^{k-3} \leq |\xi|_{\mathpzc{d},\vec{a}} \leq 2^{k+1} \}, \quad\quad k \geq 2, \\
 \supp \mathscr{F}[(S_{k+j}f)(S_{k}g)] &\subset \{ \xi : |\xi|_{\mathpzc{d},\vec{a}} \leq 2^{k+1} \}, \quad\quad k \geq 0, j \in \{-1,0,1\}, \\
 \supp \mathscr{F}[(S_{k}f)(S^{k-2}g)] &\subset \{ \xi : 2^{k-3} \leq |\xi|_{\mathpzc{d},\vec{a}} \leq 2^{k+1} \}, \quad\quad k \geq 2.
\end{align*}
Using Lemma~\ref{PIBVP:lemma:appendix:prop;trace_TL_conv_series} it can be shown as in \cite[Lemma~1.3.19]{Mey_PHD-thesis} that
\[
\norm{\Pi_{i}(f,g)}_{F^{\kappa,(\frac{1}{2n},1)}_{(p,q),p,(d-1,1)}(\R^{d},(1,v_{\mu});X)}
\lesssim \norm{f}_{L^{\infty}(\R^{d};\mathcal{B}(X))} \norm{g}_{F^{\kappa,(\frac{1}{2n},1)}_{(p,q),p,(d-1,1)}(\R^{d},(1,v_{\mu});X)}, \quad\quad i=1,2,
\]
and
\[
\norm{\Pi_{3}(f,g)}_{F^{\kappa,(\frac{1}{2n},1)}_{(p,q),p,(d-1,1)}(\R^{d},(1,v_{\mu});X)}
\lesssim \norm{(2^{n\kappa}S_{n}f)_{n}}_{L^{s}(\R)[[\ell^{p}(\N)]L^{r}(\R^{d-1})](X)}
\norm{g}_{L^{(u,v),(d-1,1)}(\R^{d-1}\times I,(1,v_{\frac{v}{q}\mu});X)}.
\]
The desired estimate now follows from \eqref{PIBVP:eq:elem_embedding_FEF} and \eqref{PIBVP:eq:appendix:sec:localization;int_rep_bound_coeff}.
\end{proof}

\begin{lemma}\label{PIBVP:lemma:appendix:sec:localization;Sob-embedding_pre-variant_(2.3.13)}
Let the notations and assumptions be as in Lemma~\ref{PIBVP:lemma:appendix:sec:localization;pre-variant_(2.3.13)}.
For each $\delta > \frac{1}{s}+\frac{d-1}{2nr}$ the inclusion
\[
F^{\delta,(\frac{1}{2n},1)}_{(p,q),\infty,(d-1,1)}(\R^{d-1}\times I,(1,v_{\mu});X)
\hookrightarrow F^{0,(\frac{1}{2n},1)}_{(u,v),1,(d-1,1)}(\R^{d-1}\times I,(1,v_{\frac{v}{q}\mu});X)
\]
holds true with a norm that can be estimated by a constant independent of $T$ and $X$.
\end{lemma}
\begin{proof}
Thanks to Lemma~\ref{PIBVP:lemma:appendix:sec:localization;ext_op} we only need to establish the inclusion for $I=\R$. Writing $\epsilon := \delta - \left[\frac{1}{s}+\frac{d-1}{2nr}\right] > 0$, we have
\begin{eqnarray*}
F^{\delta,(\frac{1}{2n},1)}_{(p,q),\infty,(d-1,1)}(\R^{d},(1,v_{\mu});X)
&\hookrightarrow& B^{\delta,(\frac{1}{2n},1)}_{(p,q),\infty,(d-1,1)}(\R^{d},(1,v_{\mu});X) \\
&\hookrightarrow& B^{\epsilon,(\frac{1}{2n},1)}_{(u,v),\infty,(d-1,1)}(\R^{d},(1,v_{\frac{v}{q}\mu});X) \\
&\hookrightarrow& F^{0,(\frac{1}{2n},1)}_{(u,v),1,(d-1,1)}(\R^{d},(1,v_{\frac{v}{q}\mu});X),
\end{eqnarray*}
where the second inclusion is obtained from Proposition~\ref{PIBVP:Sobolev_embedding_Besov}.
\end{proof}

Let us write
\[
{_{0}}F^{s,(\frac{1}{2n},1)}_{(p,q),p,(d-1,1)}(\R^{d-1}\times I,(1,v_{\mu});X)
:= \left\{\begin{array}{ll} F^{s,(\frac{1}{2n},1)}_{(p,q),p,(d-1,1)}(\R^{d-1}\times I,(1,v_{\mu});X), & s < \frac{1+\mu}{q}, \\
\{ f \in F^{s,(\frac{1}{2n},1)}_{(p,q),p,(d-1,1)}(\R^{d-1}\times I,(1,v_{\mu});X) : \mathrm{tr}_{t=0}f=0 \}, & s > \frac{1+\mu}{q}.
 \end{array}\right.
\]
A combination of Lemmas \ref{PIBVP:lemma:appendix:sec:localization;pre-variant_(2.3.13)} and \ref{PIBVP:lemma:appendix:sec:localization;Sob-embedding_pre-variant_(2.3.13)} followed by extension by zero for $g$ and extension of Fichtenholz type for $f$ yields
\begin{eqnarray*}
\norm{fg}_{{_{0}}F^{\kappa,(\frac{1}{2n},1)}_{(p,q),p,(d-1,1)}(\R^{d-1}\times J,(1,v_{\mu});X)}
&\lesssim& \norm{f}_{L^{\infty}(\R^{d-1}\times I;\mathcal{B}(X))} \norm{g}_{{_{0}}F^{\kappa,(\frac{1}{2n},1)}_{(p,q),p,(d-1,1)}(\R^{d-1}\times J,(1,v_{\mu});X)} \\
&& + \norm{f}_{F^{\kappa}_{s,p}(J;L^{r}(\R^{d-1};\mathcal{B}(X))) \cap L^{s}(J;B^{2n\kappa}_{r,p}(\R^{d-1};\mathcal{B}(X)))} \norm{g}_{{_{0}}F^{\delta,(\frac{1}{2n},1)}_{(p,q),\infty,(d-1,1)}(\R^{d-1}\times J,(1,v_{\mu});X)}
\end{eqnarray*}
with implicit constant independent of $X$ and $T$, which is a suitable substitute for the key estimate in the proof of \cite[Proposition~2.3.1]{Mey_PHD-thesis}.

\subsection*{Lower order terms}

By the trace result Theorem~\ref{PIBVP:thm:traces_main_result}, in order that the condition for the boundary operators in Remark~\ref{PBIVB:rmk:lower_order} is satisfied, it is enough that there exist $\sigma_{j,\beta} \in [0,\frac{n_{j}-|\beta|}{2n})$ such that $b_{j,\beta}$ is a pointwise multiplier from
\[
F^{\kappa_{j,\gamma}+\sigma_{j,\beta}}_{q,p}(J,v_{\mu};L^{p}(\partial\mathscr{O};X)) \cap L^{q}(J,v_{\mu};F^{2n(\kappa_{j,\gamma}+\sigma_{j,\beta})}_{p,p}(\partial\mathscr{O};X))
\]
to
\[
F^{\kappa_{j,\gamma}}_{q,p}(J,v_{\mu};L^{p}(\partial\mathscr{O};X)) \cap L^{q}(J,v_{\mu};F^{2n\kappa_{j,\gamma}}_{p,p}(\partial\mathscr{O};X)).
\]
This is achieved by the next lemma.

\begin{lemma}\label{PIBVP:appendix:lemma:lower_order_term}
Let $X$ be a Banach space, $I=(-\infty,T)$ with $T \in (-\infty,\infty]$, $\kappa,\sigma > 0$, $n \in \N_{>0}$, $p,q \in [1,\infty)$, $r,u \in (p,\infty)$, $s,v \in (q,\infty)$ with $\frac{1}{p}=\frac{1}{r}+\frac{1}{u}$ and $\frac{1}{q}=\frac{1}{s}+\frac{1}{v}$.
Let $\mu \in (-1,\infty)$ be such that $\frac{v}{q}\mu \in (-1,v-1)$.
If $\kappa + \sigma > \frac{1}{s}+\frac{d-1}{2nr}$, then
\[
\norm{fg}_{F^{\kappa,(\frac{1}{2n},1)}_{(p,q),p,(d-1,1)}(\R^{d-1}\times I,(1,v_{\mu});X)}
\lesssim
\norm{f}_{F^{\kappa}_{s,p}(J;L^{r}(\R^{d-1};\mathcal{B}(X))) \cap L^{s}(J;B^{2n\kappa}_{r,p}(\R^{d-1};\mathcal{B}(X)))}
\norm{g}_{F^{\kappa+\sigma,(\frac{1}{2n},1)}_{(p,q),p,(d-1,1)}(\R^{d-1}\times I,(1,v_{\mu});X)}
\]
with implicit constant independent of $X$ and $T$.
\end{lemma}

Note that for $\mu \in (-1,\infty)$ to be such that $\frac{v}{q}\mu \in (-1,v-1)$ it is sufficient that $\mu \in (-1,q-1)$ with $\mu > \frac{q}{s}-1$.

\begin{proof}
As in the proof of Lemma~\ref{PIBVP:lemma:appendix:sec:localization;pre-variant_(2.3.13)},
we may restrict ourselves to the case $I=\R$ and use paraproducts.
Using Lemma~\ref{PIBVP:lemma:appendix:prop;trace_TL_conv_series} and Lemma~\ref{PIBVP:lemma:elementaire_afschatting_rijtjes}, we find
\[
\norm{\Pi_{1}(f,g)}_{F^{\kappa,(\frac{1}{2n},1)}_{(p,q),p,(d-1,1)}(\R^{d},(1,v_{\mu});X)}
\lesssim \norm{f}_{B^{-\sigma,(\frac{1}{2n},1)}_{(\infty,\infty),\infty,(d-1,1)}(\R^{d};X)}
\norm{g}_{F^{\kappa+\sigma,(\frac{1}{2n},1)}_{(p,q),p,(d-1,1)}(\R^{d},(1,v_{\mu});X)}.
\]
Using Lemma~\ref{PIBVP:lemma:appendix:prop;trace_TL_conv_series}, for $i=2,3$ we find
\[
\norm{\Pi_{i}(f,g)}_{F^{\kappa,(\frac{1}{2n},1)}_{(p,q),p,(d-1,1)}(\R^{d},(1,v_{\mu});X)}
\lesssim \norm{(2^{n\kappa}S_{n}f)_{n}}_{L^{s}(\R)[[\ell^{p}(\N)]L^{r}(\R^{d-1})](X)}
\norm{g}_{L^{(u,v),(d-1,1)}(\R^{d-1}\times I,(1,v_{\frac{v}{q}\mu});X)}.
\]
Similarly to Lemma~\ref{PIBVP:appendix:lemma:BUC_top_order_coeff}, choosing $\tilde{\kappa}$ with $\kappa + \sigma > \tilde{\kappa} + \sigma > \frac{1}{s}+\frac{d-1}{2nr}$, we have
\[
F^{\kappa}_{s,p}(\R;L^{r}(\R^{d-1};X)) \cap L^{s}(\R;B^{2n\kappa}_{r,p}(\R^{d-1};X))
\hookrightarrow B^{-\sigma,(\frac{1}{2n},1)}_{(\infty,\infty),\infty,(d-1,1)}(\R^{d};X),
\]
where we now use (the vector-valued version of) \cite[Theorem~7]{JS_Sob_embeddings} instead of Corollary~\ref{PIBVP:cor:thm:trace_Besov:distr_trace}.
The desired estimate follows from Lemma~\ref{PIBVP:lemma:appendix:sec:localization;Sob-embedding_pre-variant_(2.3.13)} and \eqref{PIBVP:eq:elem_embedding_FEF}.
\end{proof}

\bibliographystyle{plain}
\def\cprime{$'$}

\end{document}